%% file: nlpdegm.tex
\pdfoutput=1
\newcommand{\ExecuteBeforeCleveref}{%
    \usepackage[noend]{algpseudocode}%
}

\documentclass{shinyart}

\usepackage[utf8]{inputenc}

\usepackage{shinybib}

\usepackage{enumitem}
\usepackage{autonum}

\usepackage{asymptote}
\begin{asydef}
    import graph;
    import math;
\end{asydef}

\usepackage{graphicx}
\graphicspath{{./pics/}}
\usepackage{tikz,pgfplots}
\usepgfplotslibrary{colorbrewer}
\pgfplotsset{compat=newest}
\pgfplotsset{plot coordinates/math parser=false}

\newcommand{\eps}{\varepsilon}
\newcommand{\N}{\mathbb{N}}
\newcommand{\R}{\mathbb{R}}
\newcommand{\extR}{\overline{\R}}

\newcommand{\norm}[1]{\|#1\|}
\newcommand{\adaptnorm}[1]{\left\|#1\right\|}
\newcommand{\iprod}[2]{\langle #1,#2\rangle}
\newcommand{\wkto}{\rightharpoonup}
\DeclareMathOperator{\sign}{\mathrm{sign}}
\DeclareMathOperator{\dom}{\mathrm{dom}}
\DeclareMathOperator{\Id}{\mathrm{Id}}
\DeclareMathOperator{\interior}{int}
\DeclareMathOperator{\conv}{conv}
\DeclareMathOperator{\graph}{Graph}
\newcommand{\prox}{\mathrm{prox}}
\newcommand{\proj}{\mathrm{proj}}
\newcommand{\polar}[1]{#1^{\circ}}

\newcommand{\inv}[1]{#1^{-1}}
\newcommand{\grad}[1]{#1'}
\newcommand{\hess}[1]{#1''}
\newcommand{\freevar}{\,\boldsymbol\cdot\,}
\newcommand{\union}\cup
\newcommand{\defeq}{:=}
\newcommand{\downto}{\searrow}
\newcommand{\setto}{\rightrightarrows}
\newcommand{\subdiff}{\partial}
\newcommand{\ind}{\iota}
\newcommand{\noise}{\delta}
\newcommand{\theData}{y^\noise}
\newcommand{\target}{y^d}

\def\DerivConeSym{V}
\newcommand{\DerivCone}[1][\relax]{\DerivConeSym\ifx\relax#1\relax\else(#1)\fi}
\newcommand{\DerivConeX}[2][]{\DerivConeSym_{#2}\ifx\relax#1\relax\else(#1)\fi}
\newcommand{\DerivConeF}[1][]{\DerivConeX[#1]{\subdiff F^*}}
\newcommand{\DerivConeGSimple}[1][]{\DerivConeX[#1]{\bar G}}
\newcommand{\DerivConeFSimple}[1][]{\DerivConeX[#1]{\bar F}}
\newcommand{\frechetCod}[1][]{\widehat D^*_{#1}}

\def\opt#1{\widetilde #1}
\def\realopt#1{\widehat #1}
\def\dir#1{\Delta #1}
\def\alt#1{#1'}
\def\Primal{u}
\def\Dual{v}
\def\coPrimal{\xi}
\def\coDual{\eta}

\def\coFdVar{\zeta}
\def\dcVar{z}

\def\baseu{{\bar{u}}}
\def\realoptq{{\realopt{q}}}
\def\realoptu{{\realopt{u}}}
\def\realoptpsi{{\realopt{\Dual}}}
\def\thisq{{{q}^i}}
\def\thisu{{{u}^i}}
\def\nextq{{{q}^{i+1}}}
\def\nextu{{{u}^{i+1}}}
\def\overnextu{{{\bar u}^{i+1}}}
\def\nextpsi{{{\Dual}^{i+1}}}

\def\thisoptq{{\opt{q}^i}}
\def\thisoptu{{\opt{u}^i}}
\def\thisoptpsi{{\opt{\Dual}^i}}
\def\nextrealoptq{{\realopt{q}^{i+1}}}

\def\thisrealoptq{{\realopt{q}^i}}

\def\somesetmap{R}
\def\somesetmapfindim{R}
\def\MMax{\Theta}
\def\MMin{\theta}
\def\Rad{R}
\def\lipnum{\ell}
\newcommand{\lip}[1]{\lipnum_{#1}}
\def\lipopt{\lip{\inv H_\realoptu}}
\def\MCond{\kappa}
\def\accel{\mu}
\newcommand{\kvbound}{b}
\def\fsquaredpart{\beta}

\title{Primal-dual extragradient methods for nonlinear nonsmooth PDE-constrained optimization}
\author{%
    Christian Clason\thanks{Faculty of Mathematics, University Duisburg-Essen, 45117 Essen, Germany (\email{christian.clason@uni-due.de})}
    \and 
    Tuomo Valkonen\thanks{Department of Mathematical Sciences, University  of Liverpool, UK (\email{tuomov@iki.fi})}
}
\date{March 15, 2017}

\hypersetup{
    pdftitle={Primal-dual extragradient methods for nonlinear nonsmooth PDE-constrained optimization},
    pdfauthor={Christian Clason, Tuomo Valkonen},
    pdfkeywords={primal-dual extragradient, PDE-constrained optimization}
}

\begin{document}

\maketitle

\begin{abstract}
    We study the extension of the Chambolle--Pock primal-dual algorithm to nonsmooth optimization problems involving nonlinear operators between function spaces. Local convergence is shown under technical conditions including metric regularity of the corresponding primal-dual optimality conditions. We also show convergence for a Nesterov-type accelerated variant provided one part of the functional is strongly convex.
    We show the applicability of the accelerated algorithm to examples of inverse problems with $L^1$ and $L^\infty$ fitting terms as well as of state-constrained optimal control problems, where convergence can be guaranteed after introducing an (arbitrary small, still nonsmooth) Moreau--Yosida regularization. This is verified in numerical examples.
\end{abstract}

\section{Introduction}

This work is concerned with the numerical solution of optimization problems of the form
\begin{equation}
    \label{eq:prob_convex}
    \min_{u\in X} F(K(u))+G(u),
\end{equation}
where $F:Y\to\extR\defeq \R\cup\{+\infty\}$ and $G:X\to\extR$ are proper, convex, and lower semicontinuous functionals, and $K:X\to Y$ is a (nonlinear) Fréchet-differentiable operator between two Hilbert spaces $X$ and $Y$ with locally Lipschitz-continuous derivative $\grad K$. 
Such problems arise for example in inverse problems with nonsmooth discrepancy or regularization terms or in optimal control problems subject to state or control constraints. We are particularly interested in the situation where $K$ is a \emph{nonlinear} operator involving the solution of a partial differential equation and $F$ is a \emph{nonsmooth} discrepancy or tracking term. 

To fix ideas, a prototypical example is the $L^1$ fitting problem
\begin{equation}
    \label{eq:l1fit_problem}
    \min_{u\in L^2(\Omega)} \norm{S(u) - \theData}_{L^1} + \frac{\alpha}2 \norm{u}_{L^2}^2,
\end{equation}
i.e., $G(u) = \frac{\alpha}2\norm{u}_{L^2}^2$, $F(y) = \norm{y}_{L^1}$, and $K(u)=S(u) - \theData$, where $S$ maps $u$ to the solution $y$ of $-\Delta y+uy=f$ for given $f$ and $\theData$ is a given noisy measurement; see \cite{ClasonJin:2011}. This problem occurs in parameter identification from data corrupted by impulsive noise instead of the usual Gaussian noise. Other examples are the $L^\infty$ fitting problem from \cite{Clason:2012} or optimal control with state constraints; see \cref{sec:application} for details.

One possible approach to solving \eqref{eq:prob_convex} is to apply a Moreau--Yosida regularization to the nonsmooth functional $F$, which allows deriving classical first-order necessary optimality conditions that can be solved by a semismooth Newton method in function spaces; see, e.g. \cite{Kunisch:2008a,Ulbrich:2011} for semismooth Newton methods in general as well as \cite{ClasonJin:2011,Clason:2012,ItoKunischState} for their application to $L^1$ fitting, $L^\infty$ fitting, and optimal control with state constraints, respectively. Such methods are very efficient due to their superlinear convergence and mesh independence; however, they suffer from local convergence, with the convergence region depending strongly on the choice of the Moreau--Yosida parameter. For this reason, usually continuation methods are employed where a sequence of minimization problems with gradually diminishing parameter are solved, although the range of parameter values for which convergence can be observed is still limited in practice.

An alternative approach which has become very popular in the context of imaging problems are primal-dual extragradient methods. 
One widely used example, introduced in \cite{Pock_PD_2010} for linear operators and extended in \cite{Valkonen:2014} to nonlinear operators, applied to \eqref{eq:prob_convex} reads as follows.

\begin{algorithm}
    \begin{algorithmic}[1]
        \State choose $u^0,v^0$
        \For {$i=0,\dots$}
        \State $u^{i+1} = \prox_{\tau G}(u^i - \tau K'(u^i)^*v^i)$
        \State $\bar u^{i+1} = 2u^{i+1}-u^i$
        \State $v^{i+1} = \prox_{\sigma F^*}(v^{i} + \sigma K(\bar u^{i+1}))$
        \EndFor
    \end{algorithmic}
    \caption{Nonlinear primal-dual extragradient method}\label{alg:nlpdegm}
\end{algorithm}
Here, $\sigma,\tau>0$ are appropriately chosen step lengths, $K'(u)^*$ denotes the adjoint of the Fréchet derivative of $K$, and $\prox_{F^*}$ denotes the proximal mapping of the Fenchel conjugate of $F$; we postpone precise definitions to later and only remark that if $F^*$ is the indicator function of a convex set $C$, the proximal mapping coincides with the metric projection onto $C$.
Such methods do not require (for linear $K$) choosing the initial guess sufficiently close to the solution to ensure convergence or solving---possibly ill-conditioned---linear systems in each iteration.
Consequently, they recently have received increasing interest also in the context of optimal control; see, e.g., \cite{Khoroshilova,Kalise}.
In addition, other proximal point methods for optimal control problems have been treated in \cite{Azhmyakov} and \cite{Borzi}; in particular, the latter is concerned with classical forward--backward splitting for sparse control of linear elliptic PDEs.
However, so far these methods have only been considered in the finite-dimensional setting, i.e., after discretizing \eqref{eq:prob_convex}, or for specific (linear) problems. 
One of the goals of this work is therefore to show convergence of \cref{alg:nlpdegm} in Hilbert spaces and to demonstrate that it can be applied to problems of the form~\eqref{eq:l1fit_problem}. 

While the general convergence theory is a straightforward extension of the analysis in \cite{Valkonen:2014} (in fact, the proof is virtually identical), it requires verifying a set-valued Lipschitz property---known as the \emph{Aubin} or \emph{pseudo-Lipschitz property}---on the inverse of a monotone operator $H_{\realopt\Primal}$ encoding the optimality conditions. This is also called \emph{metric regularity} of $H_{\realopt\Primal}$. This verification is significantly more involved in infinite dimensions. For problems of the form \eqref{eq:prob_convex} where $F$ and $G$ are given by integral functionals for regular integrands, we can apply the theory from \cite{ClasonValkonen15} to obtain an explicit, verifiable, condition for metric regularity to hold. While our analysis will show that for problems such as \eqref{eq:l1fit_problem}, this condition does in fact not hold in general unless a Moreau--Yosida regularization is introduced---or the data $y^\delta$ and the fitting term are finite-dimensional---we do obtain convergence for arbitrarily small regularization parameter, and numerical examples show that this can be observed in practice independent of the discretization. 
Similarly, although for nonlinear operators, the convergence is only local since smallness conditions on the distance to the solution enter via bounds on the nonlinearity of the operator, in contrast to semismooth Newton methods we actually observe convergence for any starting point and arbitrarily small regularization parameter without the need for continuation. 

In addition, Moreau--Yosida regularization results in a strongly convex functional, which can be exploited for accelerating \cref{alg:nlpdegm} as in \cite{ChambollePock:2014} via adaptive step length and extrapolation parameters. This leads to the following iteration.

\begin{algorithm}
    \begin{algorithmic}[1]
        \State choose $u^0,v^0$
        \For {$i=0,\dots$}
        \State $u^{i+1} = \prox_{\tau G}(u^i - \tau K'(u^i)^*v^i)$
        \State $\omega_i = 1/\sqrt{1+2\accel\sigma_i}, \quad \tau_{i+1}=\tau_i/\omega_i, \quad \sigma_{i+1} = \sigma_{i}\omega_i$ \label{alg:nlpdegm-accel:stepsize}
        \State $\bar u^{i+1} = u^{i+1}+\omega_i(u^{i+1}-u^i)$
        \State $v^{i+1} = \prox_{\sigma F^*}(v^{i} + \sigma K(\bar u^{i+1}))$
        \EndFor
    \end{algorithmic}
    \caption{Accelerated nonlinear primal-dual extragradient method}\label{alg:nlpdegm-accel}
\end{algorithm}
Here, $\accel\geq 0$ is a fixed acceleration parameter; setting $\accel=0$ coincides with the unaccelerated \cref{alg:nlpdegm}.
The appropriate choice for $\accel>0$ is related to the constant of strong convexity of $F^*$, and in the convex case  yields the optimal convergence rate of $\mathcal{O}(1/k^2)$ for the functional values rather than the rate $\mathcal{O}(1/k)$ for the original version; see \cite{Pock_PD_2010,ChambollePock:2014,ValkonenPock:2015}.
A similar acceleration is possible if $G$ is strongly convex by swapping the roles of $\sigma_i$ and $\tau_i$ in \cref{alg:nlpdegm-accel:stepsize}; we will refer to both variants as \cref{alg:nlpdegm-accel} in the following.
Such an acceleration was not considered in \cite{Valkonen:2014}. While a proof of the optimal convergence rate is outside the scope of this work, we show that \cref{alg:nlpdegm-accel} converges (locally) in infinite-dimensional Hilbert spaces under the same conditions as for \cref{alg:nlpdegm} and demonstrate the accelerated convergence in numerical examples.

\bigskip

This work is organized as follows. In the remainder of this section, we summarize some notations and definitions necessary for what follows. \Cref{sec:convergence} is concerned with the convergence analysis of the accelerated \cref{alg:nlpdegm-accel} in infinite-dimensional Hilbert spaces, where we discriminate the case of $F^*$ (\cref{sec:convergence:F}) or $G$ (\cref{sec:convergence:G}) being strongly convex. We also briefly address the verification of metric regularity for functionals of the form \eqref{eq:prob_convex} in \cref{sec:convergence:metric}. A more detailed discussion for the specific case of the motivating problems ($L^1$ fitting, $L^\infty$ fitting, and optimal control with state constraints) is given in \cref{sec:application}, where we also derive the explicit form of the accelerated \cref{alg:nlpdegm-accel} in these cases. \Cref{sec:results} concludes with numerical examples for the three model problems.

\subsection{Notation and definitions}

\paragraph{Convex analysis}
We assume $G: X \to \extR$ and $F: Y \to \extR$ to be convex, proper, lower semicontinuous functionals on Hilbert spaces $X$ and $Y$, satisfying $\interior \dom G\ne \emptyset$ and $\interior \dom F \ne \emptyset$. 
We call, e.g., $F$ strongly convex with constant $\gamma_{F} > 0$ if
\begin{equation}\label{eq:stronglyconvex}
    F(\alt{\Dual})-F(\Dual) \ge \iprod{z}{\alt{\Dual}-\Dual} + \frac{\gamma_{F}}{2}\norm{\alt{\Dual}-\Dual}^2
    \quad
    (\Dual, \alt{\Dual} \in Y;\, z \in \subdiff F(\Dual)),
\end{equation}
where $\partial F$ denotes the convex subdifferential of $F$.
We denote by 
\begin{equation}
    F^*:Y^*\to\extR,\qquad F^*(p) = \sup_{y\in Y}\ \iprod{p}{y}_Y - F(y),
\end{equation}
the \emph{Fenchel conjugate} of $F$, which is convex, proper, and lower semicontinous. As usual, we identify the topological dual $Y^*$ of $Y$ with itself. 
The \emph{Moreau--Yosida} regularization of $F$ for the parameter $\gamma>0$ is defined as
\begin{equation}\label{eq:my_def}
    F_\gamma(y) \defeq \min_{y' \in Y}~ F(y') + \frac{1}{2\gamma}\norm{y'-y}^2,
\end{equation}
whose Fenchel conjugate is (cf., e.g., \cite[Prop.~13.21\,(i)]{Bauschke:2011})
\begin{equation}\label{eq:my_conjugate}
    F_\gamma^*(p) = F^*(p) + \frac{\gamma}{2}\norm{p}^2.
\end{equation}
Note that $F^*_\gamma$ is strongly convex with constant at least $\gamma$.

For convex $F,G$ and continuously Fréchet-differentiable $K$, we can apply the calculus of Clarke's generalized derivative (which reduces to the convex subdifferential for convex functionals; see, e.g., \cite[Chap.~2.3]{Clarke}) to deduce for \eqref{eq:prob_convex} the overall system of critical point conditions
\begin{equation}
    \label{eq:oc-v2}
    \left\{ \begin{aligned}
            K(\realoptu) &\in \subdiff F^*(\realoptpsi),\\
            - \grad K(\realoptu)^* \realoptpsi &\in \subdiff G(\realoptu).
    \end{aligned} \right.
\end{equation}
\Cref{alg:nlpdegm} can be derived from these conditions with the help of the \emph{proximal mapping} (or \emph{resolvent}) of $G$,
\begin{equation}
    \prox_G(v) = \arg\min_{w\in X} \frac12\norm{w-v}_X^2 + G(w) = (\Id + \partial G)^{-1}(v),
\end{equation}
and similarly for $F^*$.
We recall the following useful calculus rules for proximal mappings, e.g., from \cite[Prop.~23.29\,(i),\,(viii)]{Bauschke:2011}:
\begin{enumerate}[label=(\textsc{p}\arabic*),ref=\textsc{p}\arabic*]
    \item For any $\sigma>0$ it holds that
        \begin{equation}
            \prox_{\sigma F^*}(v) = v - \sigma\, \prox_{\sigma^{-1} F}(\sigma^{-1}v).
        \end{equation}
    \item\label{enum:proxrule2} For any $\gamma>0$ it holds that
        \begin{equation}
            \prox_{F_\gamma^*}(v) = \prox_{(1+\gamma)^{-1}F^*}((1+\gamma)^{-1} v).
        \end{equation}
\end{enumerate}

\paragraph{Set-valued analysis}
We first define for $U \subset X$ the set of \emph{Fréchet (or regular) normals} to $U$ at $u\in U$ by 
\[
    \widehat N(u; U)
    \defeq
    \left\{
        z \in X
        \,\middle|\,
        \limsup_{U \ni \alt{u} \to u}\frac{\iprod{z}{\alt{u}-u}}{\norm{\alt{u}-u}} \le 0
    \right\}
\]
and the set of \emph{tangent vectors} by 
\begin{equation}
    \label{eq:tangent}
    T(u; U) \defeq \left\{ z \in X \,\middle|\, \text{exist } \tau^i \downto 0 \text{ and } u^i \in U \text{ such that } z=\lim_{i \to \infty} \frac{u^i-u}{\tau^i} \right\}.
\end{equation}
For a convex set $U$, these coincide with the usual normal and tangent cones of convex analysis.

For any cone $\DerivCone \subset X$, we also define the \emph{polar cone}
\begin{equation}
    \label{eq:polar}
    \polar \DerivCone \defeq \{ \dcVar \in X \mid \iprod{\dcVar}{\dcVar'} \le 0 \text{ for all } \dcVar' \in \DerivCone\}.
\end{equation}

We use the notation $\somesetmap: Q \setto W$ to denote a set-valued mapping $\somesetmap$ from $Q$ to $W$; i.e., for every $q \in Q$ holds $\somesetmap(q) \subset W$. For $\somesetmap: Q \setto W$, we define the domain $\dom \somesetmap \defeq \{q\in Q\mid\somesetmap(q)\neq \emptyset\}$ and graph $\graph\somesetmap \defeq \{(q,w)\subset Q\times W\mid w\in\somesetmap(q)\}$. 
The regular coderivatives of such maps are defined graphically with the help of the normal cones. 
Let $Q$ and $W$ be Hilbert spaces, and
$\somesetmap: Q \setto W$ with $\dom \somesetmap \ne \emptyset$. 
We then define the \emph{regular coderivative} 
$\frechetCod \somesetmap(q | w): W \setto Q$ of $\somesetmap$ at $q\in Q$ for $w\in W$ as 
\begin{equation}
    \label{eq:frechetcod}
    \frechetCod \somesetmap(q | w)(\dir{w}) \defeq
    \left\{ \dir{q} \in Q \mid (\dir{q}, -\dir{w}) \in \widehat N((q, w); \graph \somesetmap) \right\}.
\end{equation}
We also define the \emph{graphical derivative} $D\somesetmap(q|w): Q \setto W$ by
\begin{equation}
    \label{eq:graphderiv-0}
    D \somesetmap(q | w)(\dir{q}) \defeq
    \left\{ \dir{w} \in W \mid (\dir{q}, \dir{w}) \in T((q, w); \graph \somesetmap) \right\}
\end{equation}
and its convexification $\widetilde{D \somesetmapfindim}(q|w)$ via
\[
    \graph \widetilde{D \somesetmapfindim}(q|w)
    =\conv \graph[D \somesetmapfindim(q|w)].
\]

Finally, we say  that the set-valued mapping $\somesetmap: Q \setto W$ is \emph{metrically regular} at $\realopt w$ for $\realopt q$ if $\graph \somesetmap$ is locally closed and there exist $\rho, \delta, \lipnum > 0$ such that
\begin{equation}
    \label{eq:inverse-aubin}
    \inf_{p\,:\, w \in \somesetmap(p)} \norm{q-p}
    \le \lipnum \norm{w-\somesetmap(q)}
    \quad\text{ for any $q,w$ such that }
    \norm{q-\realopt{q}} \le \delta,
    \,
    \norm{w-\realopt{w}} \le \rho
    .
\end{equation}
We note that metric regularity of $\somesetmap$ is equivalent to the \emph{Aubin property} of $\inv\somesetmap$.
Hence, for the sake of consistency with \cite{ClasonValkonen15}, we denote the infimum over valid constants $\lipnum$ by $\lip{\inv \somesetmap}(\realopt w|\realoptq)$,
or $\lip{\inv \somesetmap}$ for short when there is no ambiguity about the point $(\realopt w, \realoptq)$. 
Metric regularity is then equivalent to $\lip{\inv \somesetmap}(\realopt w|\realoptq)>0$.

\section{Convergence}\label{sec:convergence}

We now show the convergence in infinite-dimensional Hilbert spaces of \cref{alg:nlpdegm-accel}, where the acceleration is stopped at some iteration $N$.
We begin by observing from the definition of the proximal mapping that \cref{alg:nlpdegm-accel} may be written in the form
\begin{equation}
    \label{eq:prox-update-discrepancy}
    0 \in H_{\thisu}(\nextq) + \nu^i + M_{i}(\nextq-\thisq)
\end{equation}
for the family, over a base point $\baseu\in X$, of monotone operators
\begin{equation}
    \label{eq:h-def}
    H_\baseu(u, \Dual) \defeq
    \begin{pmatrix}
        \subdiff G(u) + \grad K(\baseu)^* \Dual \\
        \subdiff F^*(\Dual) -\grad K(\baseu) u - c_\baseu
    \end{pmatrix}
    \quad
    \text{where}
    \quad
    c_\baseu \defeq K(\baseu)-\grad K(\baseu)\baseu,
\end{equation}
the preconditioning operator
\[
    M_i \defeq
    \begin{pmatrix}
        \tau_i^{-1} \Id & -\grad K(u^i)^* \\
        -\grad K(u^i) & \sigma_i^{-1} \Id
    \end{pmatrix},
\]
and the discrepancy term
\[
    \nu^i \defeq \bar \nu^i + \nu^i_\omega \defeq
    \underbrace{
        \begin{pmatrix}
            0 \\
            \grad K(u^i) \overnextu + c_{\thisu} - K(\overnextu)
        \end{pmatrix}
    }_{\text{linearization discrepancy}}
    +
    \underbrace{
        \begin{pmatrix}
            0 \\
            (1-\omega_i) \grad K(u^i)(\nextu-\thisu)
        \end{pmatrix}
    }_{\text{acceleration discrepancy}}.
\]
Observe (or see \cite[Lem.~3.2]{Valkonen:2014}) that $|\nu^i|\le C|u^{i+1}-u^i|$ for some constant $C>0$. This is the only property needed from $\nu^i$ for the convergence proof. Therefore $\nu^i$ can also incorporate further discrepancies, e.g., from inexact evaluation of $K$ which can be useful in the context of PDE-constrained optimization.

\begin{subequations}
    \label{eq:constants}
    Throughout, we set
    \[
        q=(u, \Dual) \in X \times Y 
        \quad\text{and}\quad
        w=(\coPrimal, \coDual) \in X \times Y,
    \]
    extending this notation to $\realoptq$, etc., in the obvious way.
    Here we fix $\Rad > 0$ such that there exists a solution $\realoptq$ to
    \begin{equation}
        \label{eq:rad}
        0 \in H_{\realoptu}(\realoptq) \quad \text{with} \quad \norm{\realoptq} \le \Rad/2.
    \end{equation}
    Note that the condition $0 \in H_{\realoptu}(\realoptq)$ is equivalent to the necessary optimality condition \eqref{eq:oc-v2} for \eqref{eq:prob_convex}. 
    Regarding the operator $K: X \to Y$ and the step length parameters $\sigma_i, \tau_i > 0$, we require that $K$ is Fréchet-differentiable with locally Lipschitz-continuous derivative $\grad K$ satisfying
    \begin{equation}
        \label{eq:ass-k}
        \sigma_i\tau_i\left(\sup_{\norm{u} \le \Rad}\norm{\grad K(u)}^2\right) < 1.
    \end{equation}
    Observe that $\sigma_i\tau_i=\sigma_0\tau_0$ is maintained under acceleration schemes such as the one in \cref{alg:nlpdegm-accel}; it is therefore sufficient to ensure this condition for the initial choice. 
    We denote by $L_2$ the local Lipschitz factor of $u \mapsto \grad K(u)$ on the
    closed ball $B(0, \Rad) \subset X$.
    We define the uniform condition number
    \begin{equation+} 
        \label{eq:mcond}
        \MCond \defeq \MMax/\MMin
    \end{equation+}
\end{subequations}
based on $\MMax$ and $\MMin$ from the condition
\begin{equation}
    \label{eq:m-bounds}
    \MMin^2 \Id \le M_i \le \MMax^2 \Id.
\end{equation}
If $\tau_i,\sigma_i>0$ are constant, $\norm{u^i} \le \Rad$,
and \eqref{eq:ass-k} holds, such $\MMin$ and $\MMax$ exist \cite[Lem.~3.1]{Valkonen:2014}.
Easily this extends to $0 < C_1 \le \tau_i,\sigma_i \le C_2 < \infty$.
\begin{remark}
    The bound \eqref{eq:m-bounds}, on which the analysis from \cite{Valkonen:2014} depends, is the reason we need to stop the acceleration: Since $\tau_i \to 0$ and $\sigma_i \to \infty$, no uniform bound exists for $M_i$ if the acceleration is not stopped. Possibly the convergence proofs from \cite{Valkonen:2014} can be extended to the fully accelerated case, but such an endeavour is outside the scope of the present work. In numerical practice, in any case, we stop the algorithm---and hence a fortiori the acceleration---at a suitable iteration $N$.
\end{remark}

We now distinguish whether $F^*$ or $G$ is strongly convex. The former is always guaranteed by Moreau--Yosida regularization of $F$, while the latter---if it holds in addition, which is the case in the examples considered here---might allow stronger acceleration, independent of the Moreau--Yosida parameter.
In both cases, the convergence proof follows closely the original proof in \cite[\S\,2--3]{Valkonen:2014}. Although this was stated in finite-dimensional spaces, none of the arguments rely on this fact. Aside from the inverse mapping theorem for set-valued functions extracted from \cite[Lem.~3.8]{Valkonen:2014}, which holds in general complete metric spaces, the arguments are entirely algebraic manipulations. They therefore hold in infinite-dimensional Hilbert spaces as well. 

Some modifications are, however, necessary since the accelerated step sizes are no longer constant.
The original proof starts with a basic descent inequality obtained from the monotonicity of $H$ and assumes strong convexity properties. It then modifies this inequality through a sequence of lemmas to obtain an estimate from which a generic telescoping result quickly  produces convergence \cite[Thm.~2.1]{Valkonen:2014}. In the following, we detail the first two elementary steps which contain changes to the original proof (for $F^*$ strongly convex only the second step changes, for $G$ strongly convex both do). The remaining steps heavily employ the metric regularity of $H$ and are unchanged and therefore only summarized briefly. 

\subsection{Convergence for strongly convex \texorpdfstring{$\scriptstyle F^*$}{F*}}\label{sec:convergence:F}

We begin by considering the case of $F^*$ being strongly convex, which is closest to the setting of \cite{Valkonen:2014}. In this case, we chose for $\accel\geq 0$ the acceleration sequence 
\begin{equation}
    \label{eq:sigma-accel-rule}
    \sigma_{i+1} \defeq \omega_i \sigma_i
    \quad\text{and}\quad \tau_{i+1} \defeq \tau_i/\omega_i
    \quad\text{with}\quad \omega_i \defeq 1/\sqrt{1+2\accel\sigma_i}.
\end{equation}

Under the above assumptions, and if metric regularity holds for $H_{\realoptu}$, \cref{alg:nlpdegm-accel} locally converges to a solution of \eqref{eq:rad}.

We begin from the basic descent estimate obtained from the monotonicity of $H_{\thisu}$ and the assumed strong convexity.
\begin{lemma}[\protect{\cite[Lem.~2.1]{Valkonen:2014}}]
    \label{lem:convergence:descent_estim}
    Let $\thisq \in X \times Y$ and $\baseu \in X$.
    Suppose $\nextq \in X \times Y$ solves \eqref{eq:prox-update-discrepancy} 
    and that $\thisoptq\in X\times Y$ is a solution to
    \begin{equation}
        0\in H_{\thisu}(\thisoptq) + \nu^i.
    \end{equation}
    If $F^*$ is strongly convex on $Y$ with constant $\gamma_{F^*}>0$, then we have 
    \begin{equation}
        \label{eq:linear-descent-strong-fstar}
        \norm{\thisq-\thisoptq}_{M_i}^2
        \ge \norm{\nextq-\thisq}_{M_i}^2
        + \norm{\nextq-\thisoptq}_{M_i}^2
        + \gamma_{F^*} 
        \norm{\nextpsi-\thisoptpsi}^2.
        \tag{$\widehat{\mathrm{D}^2}$-loc-$\gamma$-F$^*$}
    \end{equation}
\end{lemma}
\begin{remark}
    Strong convexity of $F^*$ with factor $\gamma_{F^*}$ is equivalent \cite{Hiriart-Urruty:1993} to strong monotonicity of $\subdiff F^*$ in the sense that
    \[
        \iprod{\subdiff F^*(\alt{\Dual})-\subdiff F^*(\Dual)}{\alt{\Dual}-\Dual} \ge \gamma_{F^*}\norm{\alt{\Dual}-\Dual}^2
        \qquad(\alt{\Dual},\Dual \in Y),
    \]
    observing that there is no factor $1/2$ in the latter unlike mistakenly written at \cite[the end of page 7]{Valkonen:2014}. Hence the slight difference in the statement of \eqref{eq:linear-descent-strong-fstar} in comparison to the similarly-named equation in \cite{Valkonen:2014}. In the cited article, the exact factors make no difference however; in the present work they do for the acceleration.
\end{remark}

Note that \eqref{eq:linear-descent-strong-fstar} still uses the old norm $\norm{\freevar}_{M_i}$ for the new iterate. To pass to  $\norm{\freevar}_{M_{i+1}}$ under acceleration requires replacing \cite[Lem.~3.6]{Valkonen:2014}.
For this, we first need the following bound on the step lengths.
\begin{lemma}
    \label{lemma:step-lengths_F}
    If $\{\sigma_i\}_{i\in\N}$ satisfies \eqref{eq:sigma-accel-rule},
    then $\accel+(\inv\sigma_i - \inv\sigma_{i+1}) \ge 0$. 
\end{lemma}
\begin{proof}
    We first note that
    \[
        2\accel+(\inv\sigma_i - \inv\sigma_{i+1})
        = \inv\sigma_i(2\accel\sigma_i + 1 - \inv\omega_i).
    \]
    Thus the claim holds if
    \[
        2\accel\sigma_i +1 - \inv\omega_i \ge \accel \sigma_i,
    \]
    i.e., after multiplying both sides by $\omega_i^2$ and using the definition of $\omega_i$,
    \[
        1-\omega_i \ge \accel \omega_i^2 \sigma_i.
    \]
    In other words, we need to show that
    \begin{equation}
        \label{eq:update-c-cond}
        \accel \le \frac{1-\omega_i}{\omega_i^2\sigma_i} = \frac{\inv\omega_i-1}{\omega_i\sigma_i}.
    \end{equation}
    But using the concavity of the square root, we can estimate
    \[
        \inv\omega_i - 1 = (-\sqrt{1})-(-\sqrt{1+2\accel\sigma_i})
        \ge -\frac{1}{2\sqrt{1+2\accel\sigma_i}}(1-(1+2\accel\sigma_i))
        =\accel\sigma_i\omega_i.
    \]
    This proves \eqref{eq:update-c-cond}.
\end{proof}
The following lemma is the crucial step towards extending the results of \cite{Valkonen:2014} to the accelerated case.
\begin{lemma}
    \label{lemma:norm-switch_F}
    Suppose \eqref{eq:constants} and \eqref{eq:linear-descent-strong-fstar} hold.
    Let $\Rad$, $L_2$, $\MCond$ be as in \eqref{eq:constants},
    and choose $\coPrimal_1 \in (0, 1)$. 
    If 
    \begin{equation}
        \label{eq:norm-switch-ass}
        \norm{\thisq-\realoptq} \le \Rad/4
        \quad
        \text{and}
        \quad
        \norm{\thisq-\thisoptq} \le 
        C
    \end{equation}
    for a suitable constant $C=C(\gamma_{F^*}, \accel, \coPrimal_1, \MMin, L_2, \MCond, \Rad)$, then 
    \begin{equation}
        \label{eq:descent-norm-switch}
        \tag{${\widehat{\mathrm{D}}}^2$-M}
        \norm{\thisq-\thisoptq}_{M_i}^2
        \ge \coPrimal_1 \norm{\nextq-\thisq}_{M_i}^2
        + \norm{\nextq-\thisoptq}_{M_{i+1}}^2.
    \end{equation}
    holds.
\end{lemma}
\begin{proof}
    Using \eqref{eq:norm-switch-ass} and the property $\norm{\realoptq} \le \Rad/2$ from \eqref{eq:rad}, we have
    \begin{equation}
        \label{eq:norm-switch-thisq-est}
        \norm{\thisq} \le 
        \norm{\thisq-\realoptq} + \norm{\realoptq}
        \le 3\Rad/4.
    \end{equation}
    The estimate \eqref{eq:linear-descent-strong-fstar} implies
    \[
        \norm{\nextq-\thisq} \le \MCond \norm{\thisoptq-\thisq}.
    \]
    Choosing $C\le \Rad/(4\kappa)$ and using \eqref{eq:norm-switch-ass} and \eqref{eq:norm-switch-thisq-est}, we thus get
    \[
        \norm{\nextq}
        \le
        \norm{\nextq-\thisq}+\norm{\thisq}
        \le
        \MCond \norm{\thisoptq-\thisq}+\norm{\thisq}
        \le \Rad.
    \]
    As both $\norm{\thisq}\le \Rad$ and $\norm{\nextq} \le \Rad$, by local Lipschitz continuity of $\grad K$ we have again
    \begin{equation}
        \label{eq:norm-switch-k-estim}
        \norm{\grad K(\nextu)-\grad K(\thisu)} \le L_2 \norm{\nextu-\thisu}.
    \end{equation}
    We now expand
    \begin{equation}
        \label{eq:norm-switch-0}
        \begin{aligned}
            \norm{\nextq-\thisoptq}_{M_i}^2 - \norm{\nextq-\thisoptq}_{M_{i+1}}^2
            &
            =
            -2 \iprod{\nextpsi-\thisoptpsi}{(\grad K(\nextu)-\grad K(\thisu))(\nextu-\thisoptu)}
            \\ \MoveEqLeft[-1]
            +
            (\inv\tau_i-\inv\tau_{i+1})\norm{\nextu-\thisoptu}^2
            +
            (\inv\sigma_i-\inv\sigma_{i+1})\norm{\nextpsi-\thisoptpsi}^2
            \\
            &
            \ge
            -2 \iprod{\nextpsi-\thisoptpsi}{(\grad K(\nextu)-\grad K(\thisu))(\nextu-\thisoptu)}
            \\ \MoveEqLeft[-1]
            +
            (\inv\sigma_i-\inv\sigma_{i+1})\norm{\nextpsi-\thisoptpsi}^2.
        \end{aligned}
    \end{equation}
    In the final step, we have used the fact that $\{\tau_i\}_{i\in\N}$ is non-decreasing.
    Using \eqref{eq:norm-switch-k-estim}, we further derive by application of Young's inequality
    \begin{equation}
        \label{eq:norm-switch-1}
        \begin{aligned}[t]
            \norm{\nextq-\thisoptq}_{M_i}^2 - \norm{\nextq-\thisoptq}_{M_{i+1}}^2
            &
            \ge
            (\inv\sigma_i-\inv\sigma_{i+1})\norm{\nextpsi-\thisoptpsi}^2
            \\ \MoveEqLeft[-1]
            -2L_2
            \norm{\nextpsi-\thisoptpsi}\norm{\nextq-\thisq}\norm{\nextq-\thisoptq}.
        \end{aligned}        
    \end{equation}
    Using once more Young's inequality, \eqref{eq:norm-switch-1}, and \cref{lemma:step-lengths_F}, we deduce
    \begin{equation}
        \label{eq:sigma-accel-rule-gamma-effect_F} 
        \begin{aligned}[t]
            \norm{\nextq-\thisoptq}_{M_i}^2  - \norm{\nextq-\thisoptq}_{M_{i+1}}^2
            +\gamma_{F^*}\norm{\nextpsi-\thisoptpsi}^2
            &
            \ge
            (\accel+\inv\sigma_i-\inv\sigma_{i+1})\norm{\nextpsi-\thisoptpsi}^2\\
            \MoveEqLeft[-1]
            -\frac{L_2^2}{\gamma_{F^*}-\accel}\norm{\nextq-\thisq}^2\norm{\nextq-\thisoptq}^2
            \\
            &
            \ge
            -\frac{L_2^2}{\gamma_{F^*}-\accel}\norm{\nextq-\thisq}^2\norm{\nextq-\thisoptq}^2.
        \end{aligned}
    \end{equation}
    By application of \eqref{eq:m-bounds} and \eqref{eq:linear-descent-strong-fstar}, we bound
    \begin{align}
        \norm{\nextq-\thisoptq}^2 & 
        \le \MMin^{-2} \norm{\nextq-\thisoptq}_{M_i}^2
        \le \MCond^2 \norm{\thisq-\thisoptq}^2
        \intertext{and}
        \norm{\nextq-\thisq}^2 & 
        \le \MMin^{-2} \norm{\nextq-\thisq}_{M_i}^2.
    \end{align}
    Setting
    \[
        c \defeq \frac{L_2^2}{\gamma_{F^*}-\accel}
        \quad\text{and}\quad
        C \defeq (1-\coPrimal_1)\frac{\theta^2}{c\kappa^2}
    \]
    and using \eqref{eq:norm-switch-ass} therefore yields
    \[
        c\norm{\nextq-\thisq}^2\norm{\nextq-\thisoptq}^2
        \le \frac{c\MCond^2}{\MMin^2} \norm{\nextq-\thisq}_{M_i}^2 \norm{\thisq-\thisoptq}^2
        \le (1-\coPrimal_1) \norm{\nextq-\thisq}_{M_i}^2.
    \]
    Using \eqref{eq:sigma-accel-rule-gamma-effect_F} and this estimate in \eqref{eq:linear-descent-strong-fstar}, we obtain \eqref{eq:descent-norm-switch}. 
\end{proof}

The remaining proof now proceeds as in \cite{Valkonen:2014}. Metric regularity---whose verification is the main difficulty in function spaces and will be investigated based on the results of \cite{ClasonValkonen15} at the end of this section---allows removing the squares from \eqref{eq:linear-descent-strong-fstar} and bridging from the perturbed local solutions $\thisoptq$ to local solutions $\thisrealoptq$. This is done through a sequence of technical lemmas in \cite[\S\,3.4--3.8]{Valkonen:2014} which culminate in the general descent estimate \eqref{eq:nonlinear-descent-general} of \cite[Thm.~3.1]{Valkonen:2014}. From there, a generic telescoping argument given in \cite[Thm.~2.1]{Valkonen:2014} yields convergence, which we summarize in the following statement.
\begin{theorem}
    \label{thm:convergence_F}
    Let \eqref{eq:constants} be satisfied with the corresponding constants $\Rad$, $\MMax$, $\MCond$ and $L_2$, and suppose $F^*$ is strongly convex with factor $\gamma_{F^*}$.
    Let $\realoptq$ solve $0 \in H_{\realoptu}(\realoptq)$
    and $H_{\realoptu}$ be metrically regular at $0$ for $\realoptq$ with 
    \begin{equation}
        \label{eq:pnl-realoptpsi-bound}
        \lipopt \MCond L_2 \norm{\realoptpsi} < 1-1/\sqrt{1+1/(2\lipopt^2\MMax^4)}.
    \end{equation}

    If $\accel \in [0, \gamma_{F^*})$ and we use the rule \eqref{eq:sigma-accel-rule} for $i=1,\ldots,N$ for some $N\in\N$, after which $\tau_i=\tau_N$ and $\sigma_i=\sigma_N$ for $i>N$, there exists $\delta > 0$ 
    such that for any $q^1 \in X \times Y$ with
    \begin{equation}
        \label{eq:u1-assumption}
        \norm{q^1-\realoptq} \le \delta,
    \end{equation}
    the iterates $\nextq = (\nextu, \nextpsi)$ generated by \cref{alg:nlpdegm-accel} converge to a solution $q^*=(u^*, \Dual^*)$ of~\eqref{eq:oc-v2}.
\end{theorem}
\begin{proof}
    The proof is identical to that of \cite[Thm.~2.1]{Valkonen:2014} and given here for the sake of completeness. 
    Under the given assumptions, we can, for some $\xi>0$, obtain from \eqref{eq:descent-norm-switch} as in the proof of \cite[Thm.~3.1]{Valkonen:2014} the inequality
    \begin{equation}
        \label{eq:nonlinear-descent-general}
        \norm{\thisq-\thisrealoptq}_{M_i}
        \ge \xi \norm{\nextq-\thisq}_{M_i}
        + \norm{\nextq-\nextrealoptq}_{M_{i+1}}.
        \tag{$\realopt{{\mathrm D}}$}
    \end{equation}
    It follows from \eqref{eq:nonlinear-descent-general} that
    \[
        \sum_{i=1}^\infty \norm{\nextq-\thisq}_{M_i} < \infty,
    \]
    and consequently an application of \eqref{eq:m-bounds} shows that
    \begin{equation}
        \sum_{i=1}^\infty \norm{\nextq-\thisq}
        \le \MMax \sum_{i=1}^\infty \norm{\nextq-\thisq}_{M_i} < \infty.
    \end{equation}
    This says that $\{\thisq\}_{i=1}^\infty$ is a Cauchy sequence and hence
    converges to some $\realoptq$.
    It also implies that
    \[
        \norm{M_{i}(\nextq-\thisq)} \le \MMax \norm{\nextq-\thisq} \to 0.
    \]

    Now \cite[Lem.~3.5]{Valkonen:2014} states that under the given assumptions, it follows from \eqref{eq:nonlinear-descent-general} that $\nu^i\to 0$ and hence that
    \[
        z^i \defeq \nu^i + M_{i}(\nextq-\thisq) \to 0.
    \]
    By \eqref{eq:prox-update-discrepancy}, we moreover have $-z^i \in H_{\thisu}(\nextq)$.
    Using $K \in C^1(X; Y)$ and the outer semicontinuity of the subgradient mappings
    $\subdiff G$ and $\subdiff F^*$, we see that
    \[
        \limsup_{i \to \infty} H_{\thisu}(\nextq) \subset H_{\realoptu}(\realoptq).
    \]
    Here the $\limsup$ is in the sense of an outer limit \cite{Rockafellar:1998},
    consisting of the limits of all converging subsequences of elements $v^i \in H_{\thisu}(\nextq)$.
    As by \eqref{eq:prox-update-discrepancy} we have $-z^i \in H_{\thisu}(\nextq)$,
    it follows in particular that $0 \in H_{\realoptu}(\realoptq)$, which is precisely \eqref{eq:oc-v2}.
\end{proof}

\begin{remark}
    \Cref{thm:convergence_F} holds if $F^*$ is merely strongly convex on the ``nonlinear'' subspace
    \begin{equation}
        Y_{NL}:=\{y\in Y: \iprod{z}{K(\cdot)}\in L(X,Y)\}^\bot,
    \end{equation}
    i.e., if \eqref{eq:stronglyconvex} holds merely for all $v,v'\in Y_{NL}$. In this case, $\realoptpsi$ in \eqref{eq:pnl-realoptpsi-bound} can be replaced by $P_{NL}\realoptpsi$, the orthogonal projection of $\realoptpsi$ on $Y_{NL}$.
    Indeed, \cite[Lem.~2.1]{Valkonen:2014} directly applies to $V=Y_{NL}\subsetneq Y$ to yield \eqref{eq:linear-descent-strong-fstar} for $P_{NL}(v^{i+1}-\widetilde v^i)$, and a straightforward modification of \cref{lemma:norm-switch_F} yields \eqref{eq:descent-norm-switch}. Since the Moreau--Yosida regularization, required for metric regularity in our examples, already implies strong convexity on the full space, we do not treat this more general case in detail.
\end{remark}

\subsection{Convergence for strongly convex \texorpdfstring{$\scriptstyle G$}{G}}
\label{sec:convergence:G}

In this case, we chose for $\accel\geq 0$ the acceleration sequence 
\begin{equation}
    \label{eq:tau-accel-rule}
    \sigma_{i+1} \defeq \sigma_i/\omega_i 
    \quad\text{and}\quad \tau_{i+1} \defeq \tau_i\omega_i
    \quad\text{with}\quad \omega_i \defeq 1/\sqrt{1+2\accel\tau_i}.
\end{equation}

Under the above assumptions, and if metric regularity holds for $H_{\realoptu}$, \cref{alg:nlpdegm-accel} converges to a solution of \eqref{eq:rad} as before. First, a trivial modification of the proof of \cite[Lem.~2.1]{Valkonen:2014} yields again the basic descent estimate.
\begin{lemma}
    \label{lem:convergence:descent_estim_g}
    Let $\thisq \in X \times Y$ and $\baseu \in X$.
    Suppose $\nextq \in X \times Y$ solves \eqref{eq:prox-update-discrepancy} 
    and that $\thisoptq\in X\times Y$ is a solution to
    \begin{equation}
        0\in H_{\thisu}(\thisoptq) + \nu^i.
    \end{equation}
    If $G$ is strongly convex on $X$ with constant $\gamma_{G}>0$, then we have 
    \begin{equation}
        \label{eq:linear-descent-strong-g}
        \tag{$\widehat{\mathrm{D^2}}$-loc-$\gamma$-G}
        \norm{\thisq-\thisoptq}_{M_i}^2
        \ge \norm{\nextq-\thisq}_{M_i}^2
        + \norm{\nextq-\thisoptq}_{M_i}^2
        + \gamma_G \norm{\nextu-\thisoptu}^2.
    \end{equation}
\end{lemma}
Analogously to \cref{lemma:step-lengths_F}, one now derives the following bounds.
\begin{lemma}
    \label{lemma:step-lengths_G}
    Let $\{\tau_i\}_{i\in\N}$ satisfy \eqref{eq:tau-accel-rule}.
    Then $\accel+(\inv\tau_i - \inv\tau_{i+1}) \ge 0$.
\end{lemma}
We can now show the main lemma to account for the acceleration in the case of strongly convex $G$.
\begin{lemma}
    \label{lemma:norm-switch_G}
    Suppose \eqref{eq:constants} and \eqref{eq:linear-descent-strong-g} hold.
    Let $\Rad$, $L_2$, $\MCond$ be as in \eqref{eq:constants},
    and choose $\coPrimal_1 \in (0, 1)$. 
    If 
    \begin{equation}
        \norm{\thisq-\realoptq} \le \Rad/4
        \quad
        \text{and}
        \quad
        \norm{\thisq-\thisoptq} \le 
        C
    \end{equation}
    for a suitable constant $C=C(\gamma_{G}, \accel, \coPrimal_1, \MMin, L_2, \MCond, \Rad)$, then \eqref{eq:descent-norm-switch} holds.
\end{lemma}
\begin{proof}
    Proceeding as in the proof of \cref{lemma:norm-switch_F}, since now $\{\sigma_i\}_{i\in\N}$ is non-decreasing, we derive from \eqref{eq:linear-descent-strong-g} instead of \eqref{eq:norm-switch-1} the estimate
    \begin{equation}
        \label{eq:norm-switch-2}
        \begin{aligned}[t]
            \norm{\nextq-\thisoptq}_{M_i}^2 - \norm{\nextq-\thisoptq}_{M_{i+1}}^2
            &
            \ge
            (\inv\tau_i-\inv\tau_{i+1})\norm{\nextu-\thisoptu}^2
            \\ \MoveEqLeft[-1]
            -2L_2
            \norm{\nextq-\thisoptq}\norm{\nextq-\thisq}\norm{\nextu-\thisoptu}.
        \end{aligned}        
    \end{equation}
    Aplying Young's inequality, \eqref{eq:norm-switch-2}, and \cref{lemma:step-lengths_G}, we deduce
    \begin{multline}
        \label{eq:accel-rule-gamma-effect_G}
        \norm{\nextq-\thisoptq}_{M_i}^2  - \norm{\nextq-\thisoptq}_{M_{i+1}}^2
        +\gamma_G\norm{\nextu-\thisoptu}^2
        \\
        \begin{aligned}[t]
            &
            \ge
            (\accel+\inv\tau_i-\inv\tau_{i+1})\norm{\nextu-\thisoptu}^2
            -\frac{L_2^2}{\gamma_G-\accel}\norm{\nextq-\thisq}^2\norm{\nextq-\thisoptq}^2
            \\
            &
            \ge
            -\frac{L_2^2}{\gamma_G-\accel}\norm{\nextq-\thisq}^2\norm{\nextq-\thisoptq}^2.
        \end{aligned}
    \end{multline}
    We now conclude analogously to the proof of \cref{lemma:norm-switch_F}.
\end{proof}

The remaining proof now follows as in the case of strongly convex $F^*$, and we obtain the following convergence result.
\begin{theorem}
    \label{thm:convergence_G}
    Let \eqref{eq:constants} be satisfied with the corresponding constants $\Rad$, $\MMax$, $\MCond$ and $L_2$, and suppose $G$ is strongly convex with factor $\gamma_{G}$.
    Let $\realoptq$ solve $0 \in H_{\realoptu}(\realoptq)$
    and $H_{\realoptu}$ be metrically regular at $0$ for $\realoptq$ with 
    \begin{equation}
        \lipopt \MCond L_2 \norm{\realoptpsi} < 1-1/\sqrt{1+1/(2\lipopt^2\MMax^4)}.
    \end{equation}
    If $\accel \in [0, \gamma_{G})$ and we use the rule \eqref{eq:tau-accel-rule} for $i=1,\ldots,N$ for some $N\in\N$, after which $\tau_i=\tau_N$ and $\sigma_i=\sigma_N$ for $i>N$, there exists $\delta > 0$ 
    such that for any $q^1 \in X \times Y$ with
    \begin{equation}
        \norm{q^1-\realoptq} \le \delta,
    \end{equation}
    the iterates $\nextq = (\nextu, \nextpsi)$ generated by \cref{alg:nlpdegm-accel} converge to a solution $q^*=(u^*, \Dual^*)$ of~\eqref{eq:oc-v2}.
\end{theorem}

\subsection{Metric regularity}\label{sec:convergence:metric}

We finally address the verification of metric regularity in infinite-dimensional Hilbert spaces required for the convergence of \cref{alg:nlpdegm-accel}. Motivated by the problems considered in the next section, we assume that 
\begin{equation}
    \label{eq:pde-f-choice}
    F^*(\Dual)=\int_\Omega f^*(\Dual(x)) \,d x
\end{equation}
for a proper, convex, lower semicontinuous $f^*$ and (after rescaling $F+G$, see below)
\begin{equation}
    \label{eq:pde-g-choice}
    G(u)=\frac12 \norm{u}_{L^2}^2.
\end{equation}
We wish to apply the results from \cite{ClasonValkonen15}. Towards this end, we consider the Moreau--Yosida regularization \eqref{eq:my_def} of $F$ for some parameter $\gamma>0$, and assume (using \eqref{eq:my_conjugate}) that the convexified graphical derivative of the regularized subdifferential satisfies at least at non-degenerate points for some cone $\DerivConeF[\Dual|\coDual]$ and a pointwise-defined self-adjoint positive semi-definite linear superposition operator $T_\Dual:L^2(\Omega)\to L^2(\Omega)$---i.e., $[T_\Dual \Dual](x)=t_{\Dual(x)}(x) \Dual(x)$ for some $t:\R\to\R$---the expression 
\begin{equation}
    \label{eq:fstar-polar-form}
    \widetilde{D{[\subdiff F^*]}}(\Dual|\coDual)(\dir{\Dual})
    =
    \begin{cases}
        T_\Dual \dir{\Dual} + \polar{\DerivConeF[\Dual|\coDual]}, & \dir{\Dual} \in {\DerivConeF[\Dual|\coDual]}, \\
        \emptyset, & \dir{\Dual} \not\in {\DerivConeF[\Dual|\coDual]}.
    \end{cases}
\end{equation}
Using the sum rule for graphical coderivatives from \cite[Cor.~2.3]{ClasonValkonen15}, we deduce that $\widetilde{D{[\subdiff F_\gamma^*]}}$ has the same type of structure with
\begin{equation}
    \label{eq:fstar-polar-form-huber}
    \widetilde{D{[\subdiff F_\gamma^*]}}(\Dual|\coDual)(\dir{\Dual})
    =
    \widetilde{D{[\subdiff F^*]}}(\Dual|\coDual)(\dir{\Dual}) + \gamma \dir{\Dual}.
\end{equation}

For the Moreau--Yosida regularized problem, we denote the corresponding operator $H_\realoptu$ by $H_{\gamma,\realoptu}$.
Then we have the following result.
\begin{proposition}[\protect{\cite[Prop.~4.3]{ClasonValkonen15}}]
    \label{prop:pde-metric-regularity}
    Assume \eqref{eq:fstar-polar-form} holds and $K \in C^1(X; Y)$.
    Suppose further that $\realoptq$ solves $0 \in H_{\gamma,\realoptu}(\realoptq)$ for some $\bar F \ge 0$.  Then $H_{\gamma,\realoptu}$ is metrically regular at $0$ for $\realoptq$ if and only if $T_\Dual + \gamma I \succeq \fsquaredpart I$ for some $\fsquaredpart > 0$, or
    \begin{equation}
        \label{eq:kvbound-fstar-metric}
        \bar\kvbound(\realoptq|0; H_{\realoptu}) \defeq
        \sup_{t>0} \inf\left\{ \frac{\norm{\grad K(\realoptu)\grad K(\realoptu)^* \dcVar-\nu}}{\norm{\dcVar}}
            \,\middle|\, 
            \begin{array}{l}
                0 \ne \dcVar \in \DerivConeF[\alt{\Dual}|\alt{\coDual}],\, \nu \in \polar{\DerivConeF[\alt{\Dual}|\alt{\coDual}]}, \\
                \alt{\coDual} \in \subdiff F^*(\alt{\Dual}),\, \norm{\alt{\Dual}-\realopt{\Dual}} < t,\\ \norm{\alt{\coDual}-K(\realoptu)} < t
            \end{array}
        \right\}>0.
    \end{equation}
\end{proposition}

\begin{proof}
    In \cite[Prop.~4.3]{ClasonValkonen15}, we actually take $T_\Dual=0$. 
    However, the only place where this specific structure is used is \cite[Lem.~4.1]{ClasonValkonen15}. In \cref{lemma:general-polar-projection-lower-bound-general} in \cref{app:lem41}, we have updated the sufficient conditions of the former to be able to deal with general $T_\Dual \succeq 0$.
\end{proof}

This implies convergence for any choice of the Moreau--Yosida regularization parameter $\gamma > 0$. On the other hand, if $\gamma=0$, we typically have to prove existence of a lower bound for $\bar\kvbound$. This is significantly more difficult. We will address the issue of verifying---or disproving---the lower bound on $\bar\kvbound$ with specific examples in the next section.

\section{Application to PDE-constrained optimization problems}\label{sec:application}

We now discuss the application of the preceeding analysis to the motivating examples of $L^1$ fitting, $L^\infty$ fitting, and optimal control with state constraints. 
Since this will depend on the specific structure of the mapping $S$, we consider as a concrete example the problem of recovering the potential term in an elliptic equation. 

Let $\Omega\subset\mathbb{R}^d$, $d\leq 3$, be an open bounded domain with a Lipschitz boundary $\partial\Omega$, and set $X\defeq L^2(\Omega)$ as well as 
\begin{equation}
    \label{eq:feasible-set}
    U:= \{v\in L^\infty(\Omega):v(x)\geq \eps \text{ for a.e. } x\in\Omega\}\subset X
\end{equation}
for some $\eps>0$.
For a given coefficient $u\in U$ and $f\in L^2(\Omega)$ fixed, denote by $S(u)\defeq y\in H^1(\Omega)\subset L^2(\Omega)\eqcolon Y$ the weak solution of 
\begin{equation}\label{eq:forward}
    \iprod{\nabla y}{\nabla v} + \iprod{uy}{v} = \iprod{f}{v} \qquad (v\in H^1(\Omega)).
\end{equation}
This operator has the following useful properties \cite{Kroener:2009a}:
\begin{enumerate}[label=(\textsc{a}\arabic*), ref=\textsc{a}\arabic*]
    \item The operator $S$ is uniformly bounded in $U\subset{X}$ and completely
        continuous:
        If for $u\in U$, the sequence $\{u_n\}\subset U$ satisfies
        $u_n \wkto u$ in ${X}$, then
        \begin{equation}
            S(u_n)\to  S(u) \quad\text{ in } Y.
        \end{equation}        
    \item $S$ is twice Fr\'echet differentiable.
    \item\label{ass:a3} There exists a constant $C>0$ such that 
        \begin{equation}
            \norm{\grad{S}(u)h}_{L^2}\leq C\norm{h}_X\qquad (u\in U,h\in X).
        \end{equation}
    \item\label{ass:a4} There exists a constant $C>0$ such that 
        \begin{equation}
            \norm{\hess S(u)(h,h)}_{L^2} \le C \norm{h}_X^2\qquad (u\in U,h\in X).
        \end{equation}
\end{enumerate}
Furthermore, from the implicit function theorem, the directional Fréchet derivative $\grad{S}(u)h$ at $u\in U$ for given $h\in X$ can be computed as the solution $w\in H^1(\Omega)$ to
\begin{equation}\label{eq:forward_lin}
    \iprod{\nabla w}{\nabla v} + \iprod{uw}{v} = \iprod{-yh}{v} \qquad(v\in H^1(\Omega)).
\end{equation}
Similarly, the directional adjoint derivative $\grad{S}(u)^*h$ is given by $yz$, where $z\in H^1(\Omega)$ solves
\begin{equation}\label{eq:forward_adj}
    \iprod{\nabla z}{\nabla v} + \iprod{uz}{v} = \iprod{-h}{v} \qquad(v\in H^1(\Omega)).
\end{equation}
Similar expressions hold for $\hess {S}(u)(h_1,h_2)$ and $\grad{(\grad{S}(u)^*h_1)}h_2$. Hence, assumptions (\ref{ass:a3}--\ref{ass:a4}) hold for $\grad{S}^*$ and $\grad{(\grad{S}(u)^*v)}$ for given $v$ as well.

Other operators satisfying the above assumptions are mappings from a Robin or diffusion coefficient to the solution of the corresponding elliptic partial differential equation; cf.~\cite{ClasonJin:2011}.

\subsection{\texorpdfstring{$\scriptstyle L^1$}{L¹} fitting}\label{sec:application:l1}

First, we consider the $L^1$ fitting problem \eqref{eq:l1fit_problem}.
In order to make use of the strong convexity of the penalty term for the acceleration, we rewrite this equivalently as
\begin{equation}
    \label{eq:l1fit_problem2}
    \min_{u\in L^2} \frac1\alpha\norm{S(u) - \theData}_{L^1} + \frac{1}2 \norm{u}_{L^2}^2,
\end{equation}
i.e., we set $G(u) = \frac{1}2\norm{u}_{L^2}^2$, $K(u)=S(u) - \theData$, and $F(y) = \frac1\alpha\norm{y}_{L^1}$ in \eqref{eq:prob_convex}. Hence
\[
    [F^*(p)](x)=\ind_{[-\alpha^{-1},\alpha^{-1}]}(p(x))\qquad (\text{a.e. } x\in\Omega),
\]
where $\ind_C$ denotes the indicator function of the convex set $C$ in the sense of convex analysis \cite{Hiriart-Urruty:1993}. 

To guarantee metric regularity, we replace $F$ by its Moreau--Yosida regularization, which  coincides with the well-known Huber norm, i.e.,
\begin{equation}
    \label{eq:l1-fgamma}
    F_{\gamma}(y) = \int_\Omega |y(x)|_\gamma\,dx,\qquad
    |t|_\gamma = \begin{cases}
        \frac{1}{2\gamma} |t|^2 & \text{if }|t|\leq\frac\gamma\alpha,\\
        \frac{1}{\alpha} |t| - \frac{\gamma}{2\alpha} & \text{if }|t|>\frac\gamma\alpha.
    \end{cases}
\end{equation}
Using the calculus of Clarke's generalized derivative and \eqref{eq:my_conjugate}, i.e., $\partial F_\gamma^*(p) = \partial F^*(p) + \{\gamma p\}$, we obtain the corresponding regularized optimality conditions (cf.~also \cite[Thm.~2.7]{ClasonJin:2011})
\begin{equation}\label{eq:opt_cond:l1}
    \left\{\begin{aligned}
            S(u_\gamma)-y^\delta -\gamma p_\gamma &\in \partial F^*(p_\gamma),\\
            -S'(u_\gamma)^*p_\gamma &= u_\gamma.
    \end{aligned}\right.
\end{equation}

For $G$ and $F^*$ as above, the proximal mappings are given by
\begin{align}
    [\prox_{\tau G}(u)](x) &= \tfrac{1}{1+\tau} u(x),\\
    [\prox_{\sigma F^*}(\Dual)](x) &= \proj_{[-\alpha^{-1},\alpha^{-1}]}(\Dual(x)).
    \intertext{Using rule~\eqref{enum:proxrule2} above, we thus obtain for the Moreau--Yosida regularization $F_\gamma^*$}
    [\prox_{\sigma F^*_\gamma}(\Dual)](x) &= \proj_{[-\alpha^{-1},\alpha^{-1}]}\left(\tfrac{1}{1+\sigma\gamma}\Dual(x)\right).
\end{align}

Since $G$ is strongly convex with constant $\gamma_G=1$, we can use the acceleration scheme \eqref{eq:tau-accel-rule} for any $\accel < 1$.
\Cref{alg:nlpdegm-accel} thus has the following explicit form, where we denote the dual variable with $p$ instead of $v$ to be consistent with the notation in this section.
\begin{algorithm}
    \begin{algorithmic}[1]
        \State choose $u^0,p^0$
        \For {$i=0,\dots,N$}
        \State $z^{i+1} = \grad{S}(u^i)^*p^i$
        \State $u^{i+1} = \tfrac1{1+\tau_i}(u^i - \tau_i z^{i+1})$
        \State $\omega_i = 1/\sqrt{1+2\accel\tau_i}, \quad \tau_{i+1}=\omega_i\tau_i, \quad \sigma_{i+1} = \sigma_{i}/\omega_i$
        \State $\bar u^{i+1} = u^{i+1}+\omega_i(u^{i+1}-u^i)$
        \State $p^{i+1}  = \proj_{[-\alpha^{-1},\alpha^{-1}]}\left(\tfrac{1}{1+\sigma_{i+1}\gamma}(p^{i} + \sigma_{i+1} (S(\bar u^{i+1})-\theData))\right)$
        \EndFor
    \end{algorithmic}
    \caption{Accelerated primal-dual algorithm for $L^1$ fitting}\label{alg:l1fitting}
\end{algorithm}

To show convergence of \cref{alg:l1fitting} using \cref{thm:convergence_G,prop:pde-metric-regularity}, we have to verify the expression \eqref{eq:fstar-polar-form} for $\widetilde{D[\subdiff F^*]}$. This is the content of \cite[Cor.~2.11]{ClasonValkonen15}. 
However, as discussed in \cite[\S\,5.1]{ClasonValkonen15}, for $\gamma=0$ (i.e., no regularization), we in general have $\bar\kvbound(\realoptq|0; H_{\realoptu})=0$.

(We remark that in the case of finite-dimensional data $\theData\in Y_h\subset Y$, replacing $F$ by $F\circ P_h$ where $P_h$ denotes the orthogonal projection onto $Y_h$, there exists a constant $c>0$ such that 
$\bar\kvbound(\realoptq|0; H_{\realoptu,h})\geq c>0$ holds; see \cite[\S\,5.3]{ClasonValkonen15}. Hence, regularization is not necessary in this case.)

We summarize the above discussion on the convergence for the infinite-dimensional $L^1$ fitting problem \eqref{eq:l1fit_problem2} in the next corollary.
\begin{corollary}
    \label{cor:convergence-l1}
    Let $\gamma>0$ and $\accel\in[0,1)$ be arbitrary (setting $\accel=0$ after a finite number of iterations). Let $(u_\gamma,p_\gamma)\in L^2(\Omega)^2$ be a solution to \eqref{eq:opt_cond:l1}, and take $\tau_0,\sigma_0>0$ satisfying \eqref{eq:ass-k} for $K(u) = S(u)-y^\delta$. 
    Then there exists $\delta > 0$ such that for any initial iterate $(u^1,p^1) \in L^2(\Omega)^2$ with $\norm{(u^1,p^1)-(u_\gamma,p_\gamma)} \le \delta$, the iterates $(u^k,p^k)$ generated by \cref{alg:l1fitting} converge to a solution $(u^*,p^*)$ to \eqref{eq:opt_cond:l1}.
\end{corollary}
\begin{proof}
    Note that $G$ is strongly convex with factor $1$, while Moreau--Yosida regularization makes $F_\gamma^*$ strongly convex with factor $\gamma$.
    By \cref{prop:pde-metric-regularity}, $H_{\gamma,\realoptu}$ is metrically regular at $0$ for $\realoptq$.
    The claim now follows from \cref{thm:convergence_G}.
\end{proof}

\begin{remark}
    \label{rem:feasibility}
    In general, ensuring that the iterates generated by 
    \cref{alg:l1fitting} remain feasible, i.e., satisfy $u^i\in U$, requires adding an explicit constraint to \eqref{eq:l1fit_problem2}. This would lead to a nonsmooth $G(u)=\frac12\norm{u}^2_{L^2} + \iota_{[\eps,\infty)}(u)$ (where the indicator function is to be understood pointwise almost everywhere), which was not considered in \cite{ClasonValkonen15}. The analysis there could be extended to cover this case;
    specifically, all non-degenerate cases would be covered by improving \cite[Lem.~4.1]{ClasonValkonen15} to include the case $V_{\bar G}=\{0\}$ instead of just  $V_{\bar G}=X$, see \cref{lemma:general-polar-projection-lower-bound-general} in \cref{app:lem41}.

    However, to be able to directly apply the theory as stated in \cite{ClasonValkonen15}, and since in our numerical examples the iterates are always feasible as long as the minimizer and the initial guess are sufficiently far from the lower bound, we omit the constraint in our model problems.
\end{remark}

\subsection{\texorpdfstring{$\scriptstyle L^\infty$}{L∞} fitting}\label{sec:application:linf}

We next consider the $L^\infty$ fitting (``Morozov'') problem from \cite{Clason:2012},
\begin{equation}\label{eq:linffit_problem}
    \min_u \frac{1}2 \norm{u}_{L^2}
    \quad\text{s.\,t.}\quad |[S(u)](x) -\theData(x) |\leq \noise  \quad\text{a.\,e. in } \Omega,
\end{equation}
i.e., now $F(\Dual) = \ind_{[-\noise,\noise]}(\Dual)$ (again to be understood pointwise almost everywhere) with $G$ and $K$ as before.

Again, it is well-known that the Moreau--Yosida regularization of pointwise constraints is given by its quadratic penalization, i.e.,
\begin{equation}
    F_{\gamma}(y) = \frac{1}{2\gamma}\adaptnorm{\max\{0,|y|-\delta\}}_{L^2}^2.
\end{equation}
Hence, 
\begin{equation}\label{eq:opt_cond:linf}
    \left\{\begin{aligned}
            S(u_\gamma)-y^\delta -\gamma p_\gamma &\in \partial F^*(p_\gamma),\\
            -S'(u_\gamma)^*p_\gamma &= u_\gamma,
    \end{aligned}\right.
\end{equation}
where now $F^*(\Dual)=\noise \norm{\Dual}_{L^1}$.

In this case, the proximal mapping of $F^*$ is given by
\begin{equation}
    [\prox_{\sigma F^*}(\Dual)](x) = (|\Dual(x)| - \noise\sigma)^+ \sign(\Dual(x)).
\end{equation}
For the  Moreau--Yosida regularization $F^*_\gamma$, we obtain after some simplification
\begin{equation}
    [\prox_{\sigma F^*_\gamma}(\Dual)](x) = \frac{1}{1+\sigma\gamma}\left(|\Dual(x)| - \noise\sigma\right)^+ \sign(\Dual(x)).
\end{equation}

Again, we use the acceleration scheme \eqref{eq:tau-accel-rule} for $\accel<\gamma_G=1$.
\Cref{alg:nlpdegm-accel} now has the following explicit form.
\begin{algorithm}[H]
    \begin{algorithmic}[1]
        \State choose $u^0,p^0$
        \For {$i=0,\dots,N$}
        \State $z^{i+1} = \grad{S}(u^i)^*p^i$
        \State $u^{i+1} = \tfrac1{1+\tau_i}(u^i - \tau_i z^{i+1})$
        \State $\omega_i = 1/\sqrt{1+2\accel\tau_i}, \quad \tau_{i+1}=\omega_i\tau_i, \quad \sigma_{i+1} = \sigma_{i}/\omega_i$
        \State $\bar u^{i+1} = u^{i+1}+\omega_i(u^{i+1}-u^i)$
        \State $p^{i+1}  = \tfrac{1}{1+\sigma_i\gamma}(|r^{i+1}| - \noise\sigma_i)^+ \sign(r^{i+1})$
        \EndFor
    \end{algorithmic}
    \caption{Accelerated primal-dual algorithm for $L^\infty$ fitting}\label{alg:linffitting}
\end{algorithm}

As before, we deduce from the characterization of $\widetilde{D[\subdiff F^*]}$ from \cite[Cor.~2.13]{ClasonValkonen15} that \eqref{eq:fstar-polar-form} holds for $F^*$, while the discussion in \cite[\S\,5.2]{ClasonValkonen15} shows that metric regularity of $H_{\gamma,\realoptu}$ only holds for $\gamma>0$ (or finite-dimensional data). Summarizing, we have the following convergence result for the infinite-dimensional $L^\infty$ fitting problem \eqref{eq:linffit_problem}.

\begin{corollary}
    \label{cor:convergence-linf}
    Let $\gamma>0$ and $\accel\in[0,1)$ be arbitrary (setting $\accel=0$ after a finite number of iterations). Furthermore, let $(u_\gamma,p_\gamma)\in L^2(\Omega)^2$ be a solution to \eqref{eq:opt_cond:linf}, and take $\tau_0,\sigma_0>0$ satisfying \eqref{eq:ass-k} for $K(u) = S(u)-y^\delta$. 
    Then there exists $\delta > 0$ such that for any initial iterate $(u^1,p^1) \in L^2(\Omega)^2$ with $\norm{(u^1,p^1)-(u_\gamma,p_\gamma)} \le \delta$, the iterates $(u^k,p^k)$ generated by \cref{alg:linffitting} converge to a solution $(u^*,p^*)$ to \eqref{eq:opt_cond:linf}.
\end{corollary}

\subsection{State constraints}\label{sec:application:state}

Finally, we address the state-constrained optimal control problem 
\begin{equation}
    \label{eq:state_problem}
    \min_{u\in L^2}  \frac1{2\alpha}\norm{S(u)-\target}_{L^2}^2 + \frac{1}2 \norm{u}_{L^2}^2
    \quad\text{s.\,t.}\quad [S(u)](x) \leq c  \quad\text{a.\,e. in } \Omega.
\end{equation}
In this case, $G$ is as before and $F(y) = \frac1{2\alpha}\norm{\Dual-\target}_{L^2}^2 + \ind_{(-\infty,c]}(y)$ with $K(u)=S(u)$.
For simplicity, we assume here that the upper bound $c$ is constant; the extension to variable $c\in L^\infty(\Omega)$ (as well as lower bounds) is straightforward.

For $F_\gamma$, we directly use the definition \eqref{eq:my_def} to compute pointwise
\begin{equation}
    f_\gamma(x,v) = \begin{cases}
        \frac{1}{2\alpha}|c-y^d(x)|^2 + \frac{1}{2\gamma}|v-c|^2 &\text{if } v > (1+\frac{\alpha}{\gamma})c- \frac{\alpha}{\gamma}y^d(x),\\
        \frac{1}{2(\alpha+\gamma)}|v-y^d(x)|^2  &\text{if } v \leq (1+\frac{\alpha}{\gamma})c - \frac{\alpha}{\gamma}y^d(x),
    \end{cases}
\end{equation}
and obtain 
\begin{equation}
    F_\gamma(y) = \int_\Omega f_\gamma(x,y(x))\,dx.
\end{equation}
The corresponding regularized optimality conditions are again given by
\begin{equation}\label{eq:opt_cond:state}
    \left\{\begin{aligned}
            S(u_\gamma)-y^\delta -\gamma p_\gamma &\in \partial F^*(p_\gamma),\\
            -S'(u_\gamma)^*p_\gamma &= u_\gamma.
    \end{aligned}\right.
\end{equation}
It remains to compute $F^*$.  Since $\target\in L^2(\Omega)$ is measurable, 
\begin{equation}
    f(x,v) = \frac1{2\alpha}|v-\target(x)|^2 + \iota_{(-\infty,c]}(v)
\end{equation}
is a proper, convex, and normal integrand, and hence we can proceed by pointwise computation.

Let $x\in \Omega$ be arbitrary. For the Fenchel conjugate with respect to $y$,
\begin{equation}
    f^*(x,z) = \sup_{v\leq c} vz - \frac1{2\alpha}|v-\target(x)|^2,
\end{equation}
we consider the first-order necessary conditions for the maximizer 
\begin{equation}
    \bar v = \proj_{(-\infty,c]}\left(\alpha z+\target(x)\right).
\end{equation}
Inserting this into the definition and making the case distinction whether $\alpha v+\target(x)\leq c$ yields
\begin{equation}
    f^*(x,z) = \begin{cases} 
        cz-\frac1{2\alpha}|c-\target(x)|^2 & z>\alpha^{-1}(c-\target(x)),\\
        \frac\alpha2|z|^2 + z\target(x) & z\leq \alpha^{-1}(c-\target(x)).
    \end{cases}
\end{equation}
The subdifferential (with respect to $z$) is given by
\begin{equation}\label{eq:state:subdiff}
    \partial f^*(x,z) = \begin{cases}
        \{c\} & z>\alpha^{-1}( c-\target(x)),\\
        \{\alpha z+\target(x)\} & z\leq \alpha^{-1}(c-\target(x)).
    \end{cases}
\end{equation}
(Note that the cases agree for $z=\alpha c-\target(x)$, i.e., $z\mapsto \partial f^*(x,z)$ is single-valued and hence $z\mapsto f^*(x,z)$ is continuously differentiable for almost every $x\in\Omega$.)

To compute the pointwise proximal mapping $\prox_{\sigma f^*(x,\cdot)}(v)$ for given $x\in\Omega$, we use the resolvent formula
\begin{equation}
    \prox_{\sigma f^*(x,\cdot)}(v) = (\Id + \sigma\partial f^*(x,\cdot))^{-1}(v) =:w,
\end{equation}
i.e., $v\in \{w\} + \sigma\partial f^*(x,w)$, together with \eqref{eq:state:subdiff} and distinguish the two cases
\begin{enumerate}[label=(\roman{enumi})]
    \item $v = w+ \sigma c$, i.e., $w = v-\sigma c$, if $w>\alpha^{-1}( c-\target(x))$, i.e., if $v > \alpha^{-1}( c-\target(x))+\sigma c$;
    \item $v = w+\sigma(\alpha w+\target(x))$, i.e., $w = (1+\sigma\alpha)^{-1}(v-\sigma\target(x))$, if $w\leq \alpha^{-1}( c-\target(x))$, i.e., if 
        \begin{equation}
            v \leq \frac{1+\sigma\alpha}{\alpha}(c-\target(x))+\sigma z =\alpha^{-1}( c-\target(x))+\sigma c. 
        \end{equation}
\end{enumerate}
Together we obtain
\begin{equation}
    [\prox_{\sigma F^*}(v)](x) = 
    \begin{cases}
        v(x)-\sigma c& v(x) > \frac{1}{\alpha}(c-\target(x))+\sigma c,\\
        (1+\sigma\alpha)^{-1}(v(x)-\sigma\target(x)) & v(x)\leq \frac{1}{\alpha}(c-\target(x))+\sigma c.
    \end{cases}
\end{equation}
For the Moreau--Yosida regularization $f_\gamma^*(x,v) = f^*(x,v) + \frac\gamma2|v|^2$, we similarly obtain
\begin{equation}
    [\prox_{\sigma F_\gamma^*}(v)](x) = 
    \begin{cases}
        (1+\sigma\gamma)^{-1}(v(x)-\sigma c)& v(x) > \frac{1+\sigma\gamma}{\alpha}(c-\target(x))+\sigma c,\\
        (1+\sigma(\alpha+\gamma))^{-1}(v(x)-\sigma \target(x)) & v(x)\leq\frac{1+\sigma\gamma}{\alpha}(c-\target(x))+\sigma c.
    \end{cases}
\end{equation}

Again, we use the acceleration scheme \eqref{eq:tau-accel-rule} for $\accel<\gamma_G=1$.
\Cref{alg:nlpdegm-accel} now has the following explicit form, where $\llbracket P\rrbracket$ for a logical proposition $P$ depending on $x\in\Omega$, denotes the pointwise \emph{Iverson bracket}, i.e., $\llbracket P\rrbracket(x) = 1$ if $P(x)$ is true and $0$ else.
\begin{algorithm}[H]
    \begin{algorithmic}[1]
        \State choose $u^0,p^0$
        \For {$i=0,\dots,N$}
        \State $u^{i+1} = \tfrac1{1+\tau_i}(u^i - \tau_i \grad{S}(u^i)^*p^i)$
        \State $\omega_i = 1/\sqrt{1+2\accel\tau^i}, \quad \tau^{i+1}=\omega_i\tau^i, \quad \sigma_{i+1} = \sigma_{i}/\omega_i$
        \State $\bar u^{i+1} = u^{i+1}+\omega_i(u^{i+1}-u^i)$
        \State $r^{i+1} = p^i + \sigma_{i+1}S(\bar u^{i+1})$
        \State $\chi^{i+1} = \left\llbracket r^{i+1} > \tfrac{1+\sigma_{i+1}\gamma}{\alpha}(c-\target)+\sigma_{i+1} c\right\rrbracket$
        \State $p^{i+1} = \tfrac1{1+\sigma_{i+1}\gamma}\chi^{i+1}\left(r^{i+1}-\sigma_{i+1} c\right) + \tfrac1{1+\sigma_{i+1}(\alpha+\gamma)}(1-\chi^{i+1})\left(r^{i+1}-\sigma_{i+1} \target\right)$
        \EndFor
    \end{algorithmic}
    \caption{Accelerated primal-dual algorithm for state constraints}\label{alg:state}
\end{algorithm}

Let us assume that strict complementarity holds, i.e., $\alpha v(x) \ne c-\target(x)$ for a.e.~$x \in \Omega$. Then it follows from \cref{cor:g-form-state-constr} in \cref{app:stateconstraints} that \eqref{eq:fstar-polar-form} is satisfied for $F^*$.
Furthermore, since $t_v(x) \in \{0,\alpha\}$ for a.e.~$x \in \Omega$ and
$\DerivConeF[\Dual|\coDual] = L^2(\Omega)$ and $\polar{\DerivConeF[\Dual|\coDual]} = \{0\}$ locally in a neighbourhood of $(\realoptu, \realopt{\Dual})$, we deduce that
\begin{equation}
    \bar\kvbound(\realoptq|0; H_{\realoptu}) =
    \sup_{t>0} \inf\left\{ \frac{\norm{\grad S(\realoptu)\grad S(\realoptu)^* \dcVar}}{\norm{\dcVar}}
        \,\middle|\, 
        0 \ne \dcVar \in L^2(\Omega)
    \right\}.
\end{equation}
However, the lower bound
\[
    \norm{\grad S(\realoptu)^* \dcVar} \ge c \norm{\dcVar}
    \qquad (\dcVar \in L^2(\Omega)) 
\]
does not hold in general. This can be seen by taking any orthonormal basis of $L^2(\Omega)$, which converges weakly but not strongly to zero, and use the fact that $\grad S(u)$ is a compact operator from $L^2(\Omega)$ to $L^2(\Omega)$ due to the Rellich--Kondrachev embedding theorem.
Therefore, also $\bar\kvbound(\realoptq|0; H_{\realoptu})=0$.
By \cref{prop:pde-metric-regularity}, there is thus no metric regularity without regularization ($\gamma>0$).
(Similarly to the $L^1$ fitting problem, if the state constraints are only prescribed at a finite number of points, it is possible to show metric regularity for $\gamma=0$ as well.)

The next corollary, which follows similarly to \cref{cor:convergence-l1}, summarizes the convergence results from \cref{thm:convergence_F,thm:convergence_G} for the infinite-dimensional state-constrained optimal control problem \eqref{eq:state_problem}.
\begin{corollary}
    \label{cor:convergence-stateconstr}
    Let $\gamma>0$ and $\accel\in[0,1)$ be arbitrary (setting $\accel=0$ after a finite number of iterations). Furthermore, let $(u_\gamma,p_\gamma)\in L^2(\Omega)^2$ be a solution to \eqref{eq:opt_cond:state}, and take $\tau_0,\sigma_0>0$ satisfying \eqref{eq:ass-k} for $K(u) = S(u)-y^\delta$. 
    Then there exists $\delta > 0$ such that for any initial iterate $(u^1,p^1) \in L^2(\Omega)^2$ with $\norm{(u^1,p^1)-(u_\gamma,p_\gamma)} \le \delta$, the iterates $(u^k,p^k)$ generated by \cref{alg:state} converge to a solution $(u^*,p^*)$ to \eqref{eq:opt_cond:state}.
\end{corollary}

\section{Numerical results}\label{sec:results}

We now illustrate the convergence behavior of the primal-dual extragradient method for the three model problems in \cref{sec:application}. 
Since we are interested in the properties of the algorithm in function spaces, we consider here the case in $d=1$ dimension to allow for very fine discretizations with reasonable computational effort. We have also tested the model problems in $d=2$ dimensions and observed very similar behavior.

In each case, the operator $S$ corresponds to the solution of \eqref{eq:forward} for $\Omega = (-1,1)$ and constant right-hand side $f\equiv 1$. For the implementation, we use a finite element approximation of \eqref{eq:forward} on a uniform grid with $n=1000$ elements (unless stated otherwise) with a piecewise constant discretization of $u$ and a piecewise linear discretization of $y$ as in \cite{ClasonJin:2011}. 
The functional values 
\begin{equation}
    J_\gamma(u^i) = F_\gamma(K(u^i)) + G(u^i)
\end{equation}
are computed using an approximation of the integrals by mass lumping, which amounts to a proper scaling of the corresponding discrete sums. In this way, the functional values are independent of the mesh size.

The parameters in the primal-dual extragradient method are chosen as follows: The Moreau--Yosida parameter is fixed at $\gamma=10^{-12}$ unless otherwise stated, and we compare the two cases of $\accel = 0$ (no acceleration) and $\accel=1-10^{-16}<1=\gamma_G$ (full acceleration). We point out that this value of $\gamma$ is significantly smaller than those for which semismooth Newton methods tend to converge even with continuation; cf.~\cite{ClasonJin:2011,Clason:2012}.
As a starting value, we take in each case $u^0 \equiv 1$ and $p^0 \equiv 0$. The (initial) step sizes are set to $\sigma_0 = \tilde L^{-1}$ and $\tau_0 = 0.99 \tilde L^{-1}$, where $\tilde L=\max\{1,\norm{\grad S'(u^0)u^0}/\norm{u^0}\}$ is a very simple estimate of the Lipschitz constant of $\grad K=\grad S$. The algorithm (and the acceleration) is terminated after a prescribed number $N$ of iterations.
The MATLAB implementation used to generate the results in this section can be downloaded from \url{https://github.com/clason/nlpdegm}.

\subsection{\texorpdfstring{$\scriptstyle L^1$}{L¹} fitting}\label{sec:results:l1}

We first consider the $L^1$ fitting problem \eqref{eq:l1fit_problem2} using the example from \cite{ClasonJin:2011}: We choose the exact parameter $u^\dag(x) = 2-|x|$ and corresponding exact data $y^\dag = S(u^\dag)$ and add random-valued impulsive noise by setting
\begin{equation}
    y^\delta(x) = \begin{cases} 
        y^\dag(x) + \norm{y^\dag} \xi(x) & \text{with probability } r,\\
        y^\dag(x) & \text{with probability }1-r,
    \end{cases}
\end{equation}
where for each $x\in \Omega$, $\xi(x)$ is an independent normally distributed random value with mean $0$ and variance $\delta^2$. For the results shown, we take $r=0.3$ and $\delta=0.1$, i.e., $30\%$ of data points are corrupted by $10\%$ noise. 
We then apply \cref{alg:l1fitting} with $N=1000$ iterations and $\alpha=10^{-2}$ fixed.

\Cref{fig:l1:accel} compares the convergence behavior of the functional values with $\accel=0$ and $\accel\approx 1$ (for the same data $y^\delta$). The effect of acceleration can be seen clearly. Note that the convergence is nonmonotone due to the acceleration (and the aggressive choice of step lengths). 

The convergence behavior for different mesh sizes is illustrated in \cref{fig:l1:mesh}, which shows the functional values for $n\in\{100,1000,1000\}$ (as averages over $10$ different realizations of $y^\delta$ in order to mitigate the influence of the random data). As can be observed, the number of iterations to reach a given functional value is virtually independent of the mesh size. This property---shared by many function-space algorithms---is often referred to as \emph{mesh independence}.

Finally, we report on the effect of the Moreau--Yosida parameter $\gamma$ on the performance of the algorithm. \Cref{fig:l1:gamma} shows the convergence behavior for $\gamma\in\{10^{-1},10^{-3},10^{-6}\}$ with and without acceleration. (Note that the regularized functional $J_\gamma$ depends on $\gamma$ and hence the absolute function values $J_\gamma(u^i)$ are not directly comparable for different values of $\gamma$). Without acceleration ($\accel=0$), one can observe from \cref{fig:l1:gamma:0} that the strong convexity of $F^*_\gamma$---with $\gamma_{F^*_\gamma} =\gamma$---plays a significant role for the performance. In contrast, the case with full acceleration ($\accel\approx 1$, \cref{fig:l1:gamma:1}) which exploits the strong convexity of $G$---where $\gamma_{G} = 1 \gg \gamma$---is much less affected by the value of $\gamma$, showing equally improved performance for all values of $\gamma$.

\pgfplotsset{cycle list/Dark2-3}
\pgfplotsset{every axis label/.append style={font=\small}}
\begin{figure}[!t]
    \centering
    \begin{minipage}[t]{0.475\textwidth}
        \centering
        % cheap draft mode
    \input{l1_accel_clean.tikz}% 

        \caption{$L^1$ fitting: convergence without ($\accel=0$) and with ($\accel\approx 1$) acceleration}
        \label{fig:l1:accel}
    \end{minipage}%
    \hfill
    \begin{minipage}[t]{0.475\textwidth}
        \centering
        % cheap draft mode
    \input{l1_mesh_clean.tikz}% 

        \caption{$L^1$ fitting: convergence for different mesh sizes $n$ (average of $10$ realizations)}
        \label{fig:l1:mesh}
    \end{minipage}
\end{figure}
\begin{figure}[!t]
    \centering
    \begin{subfigure}[t]{0.475\textwidth}
        \centering
        % cheap draft mode
    \input{l1_gamma_mu0.tikz}% 

        \caption{no acceleration ($\accel=0$)}
        \label{fig:l1:gamma:0}
    \end{subfigure}%
    \hfill
    \begin{subfigure}[t]{0.475\textwidth}
        \centering
        % cheap draft mode
    \input{l1_gamma_mu1.tikz}% 

        \caption{full acceleration ($\accel\approx1$)}
        \label{fig:l1:gamma:1}
    \end{subfigure}
    \caption{$L^1$ fitting: convergence for different values of Moreau--Yosida regularization parameter~$\gamma$ without and with acceleration}
    \label{fig:l1:gamma}
\end{figure}

\subsection{\texorpdfstring{$\scriptstyle L^\infty$}{L∞} fitting}\label{sec:results:linf}

For the $L^\infty$ fitting problem \eqref{eq:linffit_problem}, we choose a test problem from \cite{Clason:2012}, where $y^\delta$ is obtained from $y^\dag=S(u^\dag)$ (with $u^\dag$ as above) by quantization. Specifically, we set
\begin{equation}
    y^\delta(x)= y_s\left[\frac{y^\dag(x)}{y_s}\right], \qquad
    y_s = n_b^{-1} \left(\sup_{x\in\overline\Omega}\left(y^\dag(x)\right) - \inf_{x\in\overline\Omega}\left(y^\dag(x)\right)\right),
\end{equation}
where $n_b$ denotes the number of bins and $[s]$ denoting the nearest integer to $s\in\R$ (i.e., the data are rounded to $n_b$ discrete equidistant values). Here we take $n_b = 11$ and apply \cref{alg:linffitting} for $N=10000$ iterations.

Again, \cref{fig:linf:accel} compares the functional values over the iteration without and with acceleration and demonstrates the significantly better performance of the latter. Similarly, the comparison of different mesh sizes in \cref{fig:linf:mesh} illustrates the mesh independence of the algorithm (with slightly faster convergence for $n=100$, which can be explained by the effect of coarse discretization on the rounding procedure).
Comparing the effect of $\gamma$ on the algorithm without (\cref{fig:linf:gamma:0}) and with full (\cref{fig:linf:gamma:1}) acceleration, one again sees improved robustness with respect to $\gamma$ for the latter.

\begin{figure}[!t]
    \centering
    \begin{minipage}[t]{0.475\textwidth}
        \centering
        % cheap draft mode
    \input{linf_accel_clean.tikz}% 

        \caption{$L^\infty$ fitting: convergence without ($\accel=0$) and with ($\accel\approx 1$) acceleration}
        \label{fig:linf:accel}
    \end{minipage}%
    \hfill
    \begin{minipage}[t]{0.475\textwidth}
        \centering
        % cheap draft mode
    \input{linf_mesh_clean.tikz}% 

        \caption{$L^\infty$ fitting: convergence for different mesh sizes $n$}
        \label{fig:linf:mesh}
    \end{minipage}
\end{figure}
\begin{figure}[!t]
    \centering
    \begin{subfigure}[t]{0.475\textwidth}
        \centering
        % cheap draft mode
    \input{linf_gamma_mu0.tikz}% 

        \caption{no acceleration ($\accel=0$)}
        \label{fig:linf:gamma:0}
    \end{subfigure}%
    \hfill
    \begin{subfigure}[t]{0.475\textwidth}
        \centering
        % cheap draft mode
    \input{linf_gamma_mu1.tikz}% 

        \caption{full acceleration ($\accel\approx1$)}
        \label{fig:linf:gamma:1}
    \end{subfigure}
    \caption{$L^\infty$ fitting: convergence for different values of Moreau--Yosida regularization parameter~$\gamma$ without and with acceleration}
    \label{fig:linf:gamma}
\end{figure}
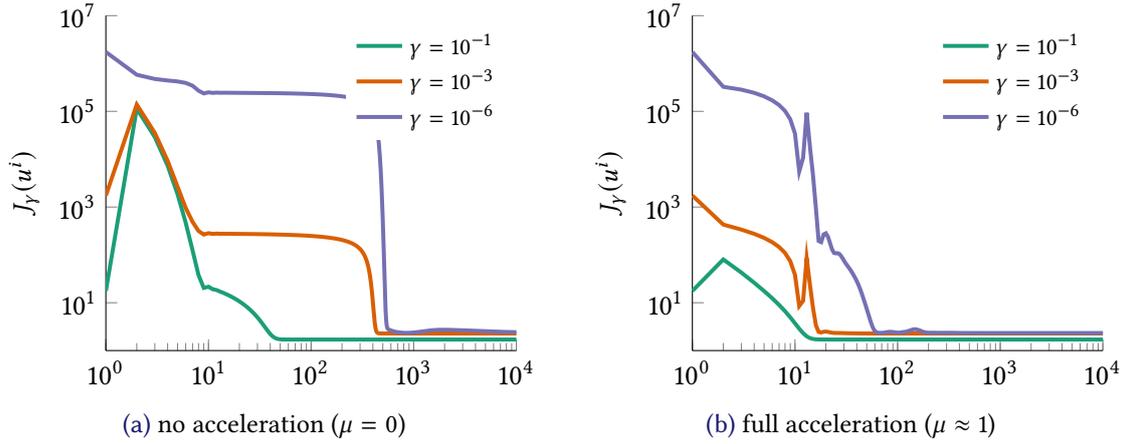

\subsection{State constraints}\label{sec:results:state}

Finally, we consider the state-constrained optimal control problem \eqref{eq:state_problem}. Here, we choose the desired state $y^d=S(u^\dag)$ (with $u^\dag$ again as before) and the constraint $c=0.68$. The control costs are set to $\alpha=10^{-12}$, and we again terminate acceleration (and the algorithm) after $N=10000$ iterations.

As before, \cref{fig:state:accel,fig:state:mesh} illustrate the benefit of acceleration and the mesh independence of the algorithm, respectively.
Since in this example, the solution only becomes feasible for very small values of $\gamma$, the visual comparison of the effect of $\gamma$ on the performance is more difficult. Nevertheless, comparing \cref{fig:state:gamma} with \cref{fig:state:accel} shows again that acceleration reduces the influence of $\gamma$ on the performance, although the effect is much less pronounced in this example.

\begin{figure}[!t]
    \centering
    \begin{minipage}[t]{0.475\textwidth}
        \centering
        % cheap draft mode
    \input{state_accel_clean.tikz}% 

        \caption{State constraints: convergence without ($\accel=0$) and with ($\accel\approx 1$) acceleration}
        \label{fig:state:accel}
    \end{minipage}%
    \hfill
    \begin{minipage}[t]{0.475\textwidth}
        \centering
        % cheap draft mode
    \input{state_mesh_clean.tikz}% 

        \caption{State constraints: convergence for different mesh sizes $n$}
        \label{fig:state:mesh}
    \end{minipage}
\end{figure}
\begin{figure}[!t]
    \centering
    \begin{subfigure}[t]{0.475\textwidth}
        \centering
        % cheap draft mode
    \input{state_gamma_mu0.tikz}% 

        \caption{no acceleration ($\accel=0$)}
        \label{fig:state:gamma:0}
    \end{subfigure}%
    \hfill
    \begin{subfigure}[t]{0.475\textwidth}
        \centering
        % cheap draft mode
    \input{state_gamma_mu1.tikz}% 

        \caption{full acceleration ($\accel\approx1$)}
        \label{fig:state:gamma:1}
    \end{subfigure}
    \caption{State constraints: convergence for different values of Moreau--Yosida regularization parameter~$\gamma$ without and with acceleration}
    \label{fig:state:gamma}
\end{figure}
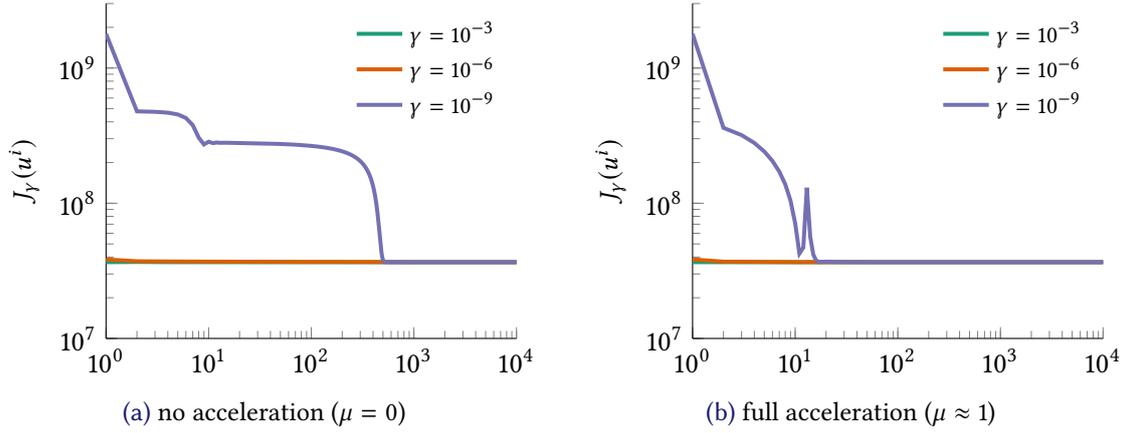

\section{Conclusion}

Accelerated primal-dual extragradient methods with nonlinear operators can be formulated and analyzed in function space. Their convergence rests on metric regularity of the corresponding saddle-point inclusion, which can be verified for the class of PDE-constrained optimization problems considered here after introducing a Moreau--Yosida regularization. Unlike semismooth Newton methods (which also require Moreau--Yosida regularization in function space, cf., e.g., \cite{ClasonJin:2011,Clason:2012,ItoKunischState}), however, in practice it is not necessary for convergence to choose $\gamma$ sufficiently large. Hence, no continuation or warm starts are required. In addition, formulating and analyzing the algorithm in function space leads to mesh independence.
These properties are observed in our numerical examples.

This work can be extended in a number of directions. We plan to investigate the possibility of obtaining convergence estimates on the primal variable alone under lesser assumptions.
An alternative would be to exploit the uniform stability with respect to regularization for fixed discretization, and with respect to discretization for fixed regularization, to obtain a combined convergence for a suitably chosen net $(\gamma,h)\to (0,0)$. This is related to the adaptive regularization and discretization of inverse problems \cite{KaltenbacherKirchnerVexler11,ClasonKaltenbacherWachsmuth15}.
Furthermore, it would be of interest to extend our analysis to include nonsmooth regularizers $G$, which were excluded in the current work for the sake of the presentation.

\appendix

\section{A slightly improved version of \texorpdfstring{\cite[Lem.~4.1]{ClasonValkonen15}}{[9, Lem. 4.1]}}
\label{app:lem41}

Here we improve the sufficient condition \cite[Lem.~4.1 (i)]{ClasonValkonen15} to allow for a more general linear operator $\bar F$ than $\bar F=\gamma I$ as well as for a more general cone $\DerivConeGSimple$ than $\DerivConeGSimple=X$ arising from the graphical derivative of $\partial F^*_\gamma$ and $\partial G$, respectively. These modifications are necessary for the treatment of state constraints; the latter is also the basis for extending the analysis in \cite{ClasonValkonen15} to cover pointwise constraints on the primal variable as mentioned in \cref{rem:feasibility}.

\begin{lemma}
    \label{lemma:general-polar-projection-lower-bound-general}
    Let $\DerivCone = \DerivConeGSimple \times \DerivConeFSimple \subset X \times Y$ be a cone, and let
    $\bar G: X \to X$, $\bar F: Y \to Y$, and $\bar K: X \to Y$ be bounded linear operators with $\bar G$ and $\bar F$ self-adjoint.
    Define
    \begin{equation}
        T \defeq 
        \begin{pmatrix}
            \bar G & \bar K^* \\
            - \bar K & \bar F
        \end{pmatrix}.
    \end{equation}
    \clearpage
    Suppose both $\bar G \succeq_{\DerivConeGSimple}\! c_G^2 I$ and $\bar F \succeq_{\DerivConeFSimple}\! c_F^2 I$ for some $c_G, c_F > 0$, i.e.,
    \begin{equation}
        \iprod{\bar G\coPrimal}{\coPrimal}  \ge c_G^2 \norm{\coPrimal}^2 \qquad (\coPrimal \in \DerivConeGSimple)\qquad\text{and} \qquad
        \iprod{\bar F\coDual}{\coDual} \ge c_F^2 \norm{\coDual}^2 \qquad (\coDual \in \DerivConeFSimple).
    \end{equation}
    Then there exists $c>0$ with
    \begin{equation}
        \label{eq:general-polar-projection-lower-bound}
        \inf_{\dcVar \in \polar\DerivCone} \norm{T^* w - \dcVar}^2 \ge c\norm{w}^2 \qquad (w \in \DerivCone).
    \end{equation}
\end{lemma}

\begin{proof}
    With $w=(\coPrimal,\coDual) \in \DerivCone = \DerivConeGSimple \times \DerivConeFSimple$, and $\dcVar=(\mu, \nu) \in \polar\DerivCone$,
    we expand
    \begin{equation}
        \label{eq:general-polar-projection-lower-bound-est0}
        \begin{aligned}[t]
            \norm{T^*w-\dcVar}^2
            & =
            \norm{\bar G \coPrimal-\bar K^*\coDual-\mu}^2
            +
            \norm{\bar K \coPrimal+\bar F\coDual-\nu}^2
            \\
            & =
            \norm{\bar G \coPrimal}^2 +\norm{\bar K^*\coDual+\mu}^2-2\iprod{(\bar K\bar G-\bar F \bar K)\coPrimal}{\coDual}
            \\ \MoveEqLeft[-1]        
            +\norm{\bar K\coPrimal-\nu}^2
            +\norm{\bar F\coDual}^2
            -2\iprod{\bar F\coDual}{\nu}
            -2\iprod{\bar G\xi}{\mu}.
        \end{aligned}
    \end{equation}
    Let $\lambda,\beta > 0$ be arbitrary.
    We can insert $0=\iprod{(\Lambda \bar K - \Lambda \bar K)\coPrimal}{\coDual}$ and $0=\iprod{(\bar K M - \bar K M)\coPrimal}{\coDual}$ into \eqref{eq:general-polar-projection-lower-bound-est0}.
    We can also use $\iprod{\coDual}{(\Lambda-\bar F)\nu} \ge 0$ for all $\nu \in \polar {\DerivConeFSimple}$ and $\coDual \in {\DerivConeFSimple}$ and similarly $\iprod{\coPrimal}{(M-\bar G)\mu} \ge 0$ for all $\mu \in \polar {\DerivConeGSimple}$ and $\coPrimal \in {\DerivConeGSimple}$.
    Thus
    \begin{equation}
        \label{eq:general-polar-projection-lower-bound-lambda-introduce}
        \begin{aligned}[t]
            \norm{T^*w-\dcVar}^2
            & =
            \norm{\bar G \coPrimal}^2+\norm{\bar K^*\coDual+\mu}^2-2\iprod{(\bar K \bar G - \bar K M + \Lambda\bar K- \bar F\bar K)\coPrimal}{\coDual}
            \\ \MoveEqLeft[-1]
            +2\iprod{\bar K \coPrimal}{\Lambda\coDual}
            -2\iprod{\bar K^* \coDual}{M\coPrimal}
            +\norm{\bar K\coPrimal-\nu}^2
            +\norm{\bar F\coDual}^2
            -2\iprod{\bar F\coDual}{\nu}
            -2\iprod{\bar G\coPrimal}{\mu}
            \\
            & \ge
            \norm{\bar G \coPrimal}^2+\norm{\bar K^*\coDual+\mu}^2-2(\beta-\lambda)\iprod{\coPrimal}{\bar K^*\coDual}
            \\ \MoveEqLeft[-1]
            +2\iprod{\bar K\coPrimal-\nu}{\Lambda\coDual}
            -2\iprod{\bar K^* \coDual+\mu}{M\coPrimal}
            +\norm{\bar K\coPrimal-\nu}^2
            +\norm{\bar F\coDual}^2.
        \end{aligned}
    \end{equation}
    This we further estimate using Young's inequality as
    \begin{equation}
        \norm{T^*w-\dcVar}^2
        \ge
        \bigl(\norm{\bar G \coPrimal}^2-\norm{M\coPrimal}^2\bigr)
        +\bigl(\norm{\bar F\coDual}^2-\norm{\Lambda\coDual}^2\bigr) 
        -2(\beta-\lambda)\iprod{\coPrimal}{\bar K^* \coDual}.
    \end{equation}
    Expanding the definitions of $M$ and $\Lambda$ and using $\bar G \succeq_{\DerivConeGSimple}\! c_G^2 I$ and $\bar F \succeq_{\DerivConeFSimple}\! c_F^2 I$, we get
    \begin{equation}
        \label{eq:general-polar-projection-lower-bound-est2}
        \norm{T^*w-\dcVar}^2
        \ge
        (2c_G^2 \beta - \beta^2)\norm{\coPrimal}^2
        +(2c_F^2 \lambda-\lambda^2)\norm{\coDual}^2
        -2(\beta-\lambda)\iprod{\coPrimal}{\bar K^* \coDual}.
    \end{equation}
    Taking $\beta=\lambda$ and $0 < \lambda < 2\min\{c_G^2, c_F^2\}$, we thus see that \eqref{lemma:general-polar-projection-lower-bound-general} holds with $c=c(c_F, c_G)$.
\end{proof}

\section{Second-order generalized derivative for state constraints}\label{app:stateconstraints}

In this appendix we give the the pointwise characterization of the graphical derivative of $\partial F^*$ given by 
\begin{equation}
    \label{eq:state-constraints-fstar}
    \partial f^*(x,z) = \begin{cases}
        \{c\} & z>\alpha^{-1}( c-\target(x)),\\
        \{\alpha z+\target(x)\} & z\leq \alpha^{-1}(c-\target(x)),
    \end{cases}
\end{equation}
required for the verification of \eqref{eq:fstar-polar-form}.
We begin with the (convexified) graphical derivative of \eqref{eq:state-constraints-fstar}, where from now on we suppress the dependence on $x\in\Omega$ for the sake of presentation.
\begin{lemma}
    For $\partial f^*$ as in \eqref{eq:state-constraints-fstar}, we have 
    \begin{equation}
        \label{eq:d-subdiff-state-constr-1d}
        D(\subdiff f^*)(\Dual|\coFdVar)(\dir{\Dual})
        =
        \begin{cases}
            0, & \alpha \Dual > c - \target,\, \coFdVar=c, \\
            \alpha\dir{\Dual}, & \alpha \Dual < c - \target,\, \coFdVar=\alpha v+\target, \\
            0, & \alpha v=c-\target,\, \coFdVar=c,\, \dir{\Dual} \ge 0, \\
            \alpha\dir{\Dual}, & \alpha v=c-\target,\, \coFdVar=c,\, \dir{\Dual} < 0,
        \end{cases}
    \end{equation}
    and
    \begin{equation}
        \label{eq:d-subdiff-state-constr-1d-conv}
        \widetilde{D(\subdiff f^*)}(\Dual|\coFdVar)(\dir{\Dual})
        =
        \begin{cases}
            0, & \alpha \Dual > c - \target,\, \coFdVar=c, \\
            \alpha \dir{\Dual}, & \alpha \Dual < c - \target,\, \coFdVar=\alpha v+\target, \\
            (-\infty, 0], & \alpha v=c-\target,\, \coFdVar=c,\, \dir{\Dual} \ge 0, \\
            \alpha \dir{\Dual}+(-\infty, 0], & \alpha v=c-\target,\, \coFdVar=c,\, \dir{\Dual} < 0.
        \end{cases}
    \end{equation}
\end{lemma}
\begin{proof}
    The claim is best seen by inspecting \cref{fig:state-constr-1d};
    for completeness we however sketch the proof based on casewise inspection of \eqref{eq:state-constraints-fstar}.
    \begin{enumerate}[label=(\roman*)]
        \item     If $\alpha \Dual \ne c-\target$, we have $\partial f^*(\Dual)=\{\grad {(f^*)}(\Dual)\}$ with $\grad {(f^*)}(\Dual)$ differentiable. Computing these differentials yields the first two cases of \eqref{eq:d-subdiff-state-constr-1d}, where the constraints on $\coFdVar$ come from $\coFdVar=\grad {(f^*)}(\Dual)$.

        \item    If $\alpha \Dual = c-\target$, we have $\partial f^*(\Dual)=\{c\}$, so we need $\coFdVar=c$. Approaching $\Dual$ with $v^i=v+t^i \dir\Dual$ with $\dir\Dual \ge 0$ and $t^i \downto 0$, we have
            \[
                \limsup_{i \to \infty} \frac{\subdiff f^*(\Dual^i)-\zeta}{t^i}
                =
                \limsup_{i \to \infty} \frac{c-c}{t^i}=\{0\}.
            \]
            This gives the third case of \eqref{eq:d-subdiff-state-constr-1d}.

        \item If $\dir\Dual<0$, we obtain
            \[
                \limsup_{i \to \infty} \frac{\subdiff f^*(\Dual^i)-\zeta}{t^i}
                =
                \limsup_{i \to \infty} \frac{\alpha(\Dual+t^i\dir\Dual)+\target-c}{t^i}
                =
                \limsup_{i \to \infty} \frac{\alpha t^i\dir\Dual}{t^i}=\{\alpha\dir\Dual\}.
            \]    
            This gives the fourth case of \eqref{eq:d-subdiff-state-constr-1d}.

    \end{enumerate}

    Finally, the first two cases of the convexification \eqref{eq:d-subdiff-state-constr-1d-conv} correspond directly to those of \eqref{eq:d-subdiff-state-constr-1d}, while the last two cases come from taking the convex hull of the set
    \[
        A:=([0, \infty) \times \{0\})
        \union
        \left\{(\dir\Dual, \alpha\dir\Dual) \mid \dir\Dual < 0\right\}
    \]
    corresponding to the last two cases of \eqref{eq:d-subdiff-state-constr-1d}, which is given by
    \begin{equation}
        \begin{split}
            \conv A =  ([0, \infty) \times (-\infty, 0])
            \union
            \left\{\{\dir\Dual\} \times (-\infty, \alpha\dir\Dual] \mid \dir\Dual < 0\right\}.
            \qedhere
        \end{split}
    \end{equation}
\end{proof}
\begin{figure}[t]
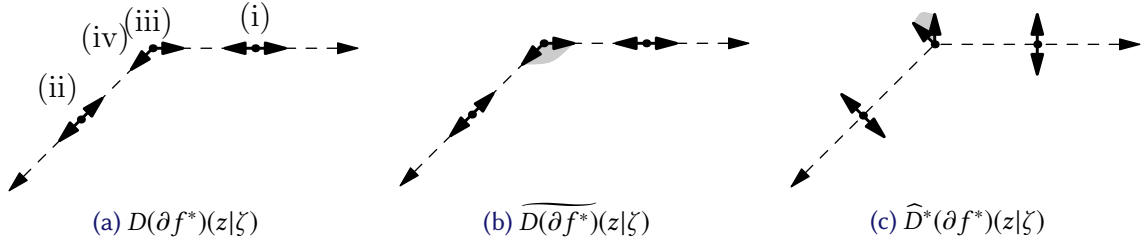

    \centering
    \begin{subfigure}{0.30\textwidth}
        \centering
        % DF
        \begin{asy}
            unitsize(45, 45);
            real l=1.7;
            real eps=0.25;
            real c=1;
            pair xx=(0, c);
            pair gfstart=xx+l*(-.7, -.7);
            pair gfend=(l, c);
            path gf=gfstart--xx--gfend;
            draw(gf, dashed, Arrows);

            draw((0.5*l-eps, c)--(0.5*l+eps, c), linewidth(1.1), Arrows);
            dot((0.5*l, c));
            label("(i)", (0.5*l, c), 1.5*N);

            pair p=(xx+gfstart)/2;
            pair d=eps/sqrt(2)*(1,1);
            draw((p-d)--(p+d), linewidth(1.1), Arrows);
            dot(p);
            label("(ii)", p, 1.5*N+W);

            draw((xx-d)--xx--(xx+eps*(1, 0)), linewidth(1.1), Arrows);
            dot(xx);
            label("(iv)", xx-d, 1.5*N+W);
            label("(iii)", xx+eps*(1,0), N+W);
        \end{asy}
        \caption{$D(\subdiff f^*)(z|\zeta)$}
    \end{subfigure}
    \hfill
    \begin{subfigure}{0.30\textwidth}
        \centering
        % \widetilde DF
        \begin{asy}
            unitsize(45, 45);
            real l=1.7;
            real eps=0.25;
            real c=1;
            pair xx=(0, c);
            pair gfstart=xx+l*(-.7, -.7);
            pair gfend=(l, c);
            path gf=gfstart--xx--gfend;
            draw(gf, dashed, Arrows);

            draw((0.5*l-eps, c)--(0.5*l+eps, c), linewidth(1.1), Arrows);
            dot((0.5*l, c));
            //label("(i)", (0.5*l, c), 1.5*N);

            pair p=(xx+gfstart)/2;
            pair d=eps/sqrt(2)*(1,1);
            draw((p-d)--(p+d), linewidth(1.1), Arrows);
            dot(p);
            //label("(ii)", p, 1.5*N+W);

            fill((xx-d)--xx--(xx+eps*(1, 0))..controls(xx+(-d+eps*(1, 0)))..cycle, gray(.8));
            draw((xx-d)--xx--(xx+eps*(1, 0)), linewidth(1.1), Arrows);
            dot(xx);
            //label("(iv)", xx-d, 1.5*N+W);
            // alignment
            label("$\phantom{(iii)}$", xx+eps*(1,0), N+W);
        \end{asy}
        \caption{$\widetilde{D(\subdiff f^*)}(z|\zeta)$}
    \end{subfigure}
    \hfill
    \begin{subfigure}{0.30\textwidth}
        \centering
        % \hat D^* F
        \begin{asy}
            unitsize(45, 45);
            real l=1.7;
            real eps=0.25;
            real c=1;
            pair xx=(0, c);
            pair gfstart=xx+l*(-.7, -.7);
            pair gfend=(l, c);
            path gf=gfstart--xx--gfend;
            draw(gf, dashed, Arrows);

            draw((0.5*l, c-eps)--(0.5*l, c+eps), linewidth(1.1), Arrows);
            dot((0.5*l, c));
            //label("(i)", (0.5*l, c), 1.5*N);

            pair p=(xx+gfstart)/2;
            pair d=eps/sqrt(2)*(1,-1);
            draw((p-d)--(p+d), linewidth(1.1), Arrows);
            dot(p);
            //label("(ii)", p, 1.5*N+W);

            fill((xx-d)--xx--(xx+eps*(0, 1))..controls(xx+(-d+eps*(0, 1))/1.5)..cycle, gray(.8));
            draw((xx-d)--xx--(xx+eps*(0, 1)), linewidth(1.1), Arrows);
            dot(xx);
            //label("(iv)", xx-d, 1.5*N+W);
            // alignment
            label("$\phantom{(iii)}$", xx+eps*(1,0), N+W);
        \end{asy}
        \caption{$\widehat D^*(\subdiff f^*)(z|\zeta)$}
    \end{subfigure}
    \caption{Illustration of the graphical derivative and Fréchet coderivative for $\subdiff f^*$ as in \eqref{eq:state-constraints-fstar}. The dashed line is $\graph \subdiff f$. The dots indicate the base points $(z, \zeta)$ where the graphical derivative or coderivative is calculated, and the thick arrows and gray areas indicate the directions of $(\dir z, \dir \zeta)$ relative to the base point.
        The labels (i) etc. denote the corresponding case of \eqref{eq:d-subdiff-state-constr-1d}.
    }
    \label{fig:state-constr-1d}
\end{figure}

Since $f$ is proper, convex, and normal, so is $f^*$; see, e.g., \cite[Thm.~14.50]{Rockafellar:1998} for the former. Furthermore, for almost every $x\in\Omega$, the functional $f^*(x,\cdot)$ is piecewise affine, and hence $\partial f^*(x,\cdot)$ is proto-differentiable; see \cite[Prop.~13.9, Thm.~13.40]{Rockafellar:1998}. We can thus apply \cite[Cor.~2.7]{ClasonValkonen15} to obtain the following pointwise characterization of the second-order generalized derivatives of the corresponding integral functional $F^*$.
\begin{corollary}
    \label{cor:g-form-state-constr}
    Let $\partial f^*$ be as in \eqref{eq:state-constraints-fstar}, and
    \[
        \partial F^*(\Dual) \defeq \left\{\eta\in L^2(\Omega)\mid \eta(x) \in \partial f^*(\Dual(x))\text{ for a.e. }x\in\Omega \right\}.
    \]
    Suppose $\alpha v(x) \ne c-\target(x)$ for a.e.~$x \in \Omega$.
    Then
    \begin{align}
        \label{eq:g-form-state-constr-dconv}
        \widetilde{D[\subdiff F^*]}(\Dual|\coDual)(\dir\Dual)&=
        \begin{cases}
            T_{\Dual}\dir{\Dual}+\DerivConeF[\Dual|\coDual]^\circ, & \dir\Dual \in \DerivConeF[\Dual|\coDual] \text{ and } \coDual \in \subdiff F^*(\Dual),\\
            \emptyset, & \text{otherwise},
        \end{cases}
        \intertext{and}
        \label{eq:g-form-state-constr-cod}
        \frechetCod{[\subdiff F^*]}(\Dual|\coDual)(\dir\coDual)&=
        \begin{cases}
            T_{\Dual}^*\dir{\coDual}+{\DerivConeF[\Dual|\coDual]}^\circ, & -\dir\coDual \in \DerivConeF[\Dual|\coDual] \text{ and } \coDual \in \subdiff F^*(\Dual),\\
            \emptyset, & \text{otherwise},
        \end{cases}
    \end{align}
    for the cone
    \begin{align}
        \DerivConeF[\Dual|\coDual] &= L^2(\Omega),
        \intertext{its polar}
        \polar{\DerivConeF[\Dual|\coDual]}&= \{0\} \subset L^2(\Omega),
    \end{align}
    and the linear operator $T_{\Dual}$ defined by
    \[
        [T_{\Dual}\dir{\Dual}](x) \defeq t_\Dual(x) \dir{\Dual}(x),
        \qquad
        t_\Dual(x)
        \defeq
        \begin{cases}
            0, & \alpha v(x)>c-\target(x), \\
            \alpha, & \alpha v(x)<c-\target(x).
        \end{cases}
    \]    
\end{corollary}

\begin{remark}
    We have excluded the case $\alpha v(x) = c-\target(x)$---which amounts to a strict complementarity assumption for $v$---because the calculations of \cite{ClasonValkonen15} only apply when the polarity relationships in \eqref{eq:g-form-state-constr-dconv} and \eqref{eq:g-form-state-constr-cod} regarding $V$\/hold.
    We have verified that the calculations could be improved to handle this non-strictly complementary case. However, since non-strictly complementary solutions can be replaced by strictly complementary solutions by infinitesimal modifications of $v$, we have decided for conciness to simply exclude the case.
\end{remark}

\section*{Acknowledgments}

While TV was in Cambridge, he was supported by the King Abdullah University of Science and Technology (KAUST) Award No.~KUK-I1-007-43, and EPSRC grants Nr.~EP/J009539/1 ``Sparse \& Higher-order Image Restoration'' and Nr.~EP/M00483X/1 ``Efficient computational tools for inverse imaging problems''. Part of this work was also done while TV was in Quito, where he was supported by a Prometeo scholarship of the Senescyt (Ecuadorian Ministry of Science, Technology, Education, and Innovation).
CC is supported by the German Science Foundation DFG under grant Cl~487/1-1.

\section*{\texorpdfstring{\normalsize}{}A data statement for the EPSRC}

{\small
    There is no additional data supporting this publication. All source codes used to generate the results in \cref{sec:results} are archived at \url{http://dx.doi.org/10.5281/zenodo.398822}.
}

\printbibliography

\end{document}

%% file: l1_accel_clean.tikz
\begin{tikzpicture}

\begin{axis}[%
    width=\linewidth,
xmode=log,
xmin=1,
xmax=1000,
xminorticks=true,
ymode=log,
ymin=1,
ymax=100000,
yminorticks=true,
ylabel={$J_\gamma(u^i)$},
axis x line*=bottom,
axis y line*=left,
legend style={legend cell align=left,align=left,draw=none,font=\footnotesize}
]
\addplot +[line width=1.5pt]
  table[row sep=crcr]{%
1	263.820947334336\\
2	38988.1444753368\\
3	9949.31810404171\\
4	2611.32535509658\\
5	754.457743857396\\
6	279.523924732963\\
7	149.371617622959\\
8	109.964164217437\\
9	114.758421691903\\
10	112.095370041099\\
11	113.17067782802\\
12	112.43784687391\\
13	112.597378577543\\
15	112.256339445055\\
28	110.450717670931\\
38	109.024944700661\\
48	107.546682313398\\
59	105.853846501039\\
70	104.083437388053\\
81	102.226262720364\\
92	100.271500851328\\
103	98.2063208553945\\
114	96.0153944875563\\
125	93.680264871286\\
135	91.4135206930412\\
145	88.988748257769\\
154	86.6510503436972\\
163	84.1437704844719\\
171	81.7518297817272\\
179	79.1837618389788\\
187	76.4140009108672\\
194	73.8014949610392\\
201	70.9883300638606\\
208	67.9481046068215\\
214	65.1392990896664\\
220	62.1223468608828\\
226	58.8770855556074\\
232	55.3845642801344\\
238	51.6297502858436\\
244	47.6058712645365\\
250	43.3208573893338\\
256	38.8089778077057\\
262	34.15300815451\\
268	29.4649114472668\\
275	24.1846201094898\\
283	18.7799304800216\\
303	10.0434589816831\\
310	8.46795291660586\\
316	7.53714740086154\\
322	6.88969400700666\\
325	6.68218407968385\\
328	6.57007431532636\\
331	6.5239253604377\\
334	6.51318854496471\\
363	6.54881121663291\\
385	6.53329062604398\\
493	6.44604631278474\\
545	6.41219413381237\\
558	6.37398600648351\\
581	6.34371653439059\\
640	6.33294595761818\\
676	6.33204012302088\\
722	6.27247006438025\\
750	6.26656459143165\\
766	6.27470941746632\\
781	6.26958574087906\\
813	6.25598759926218\\
847	6.27128373724347\\
938	6.27149518001159\\
971	6.2597329402138\\
1000	6.24972901741627\\
};
\addlegendentry{$\accel=0$};

\addplot +[line width=1.5pt]
  table[row sep=crcr]{%
1	264.286911080509\\
2	219.840991469247\\
3	161.656631860907\\
5	112.924605619108\\
6	98.1939424471569\\
7	85.2789673687707\\
8	72.6157073331342\\
9	58.7583707002998\\
10	41.6878302240981\\
11	18.1753030050745\\
12	23.5390373835584\\
13	64.0943160055055\\
14	26.8579259209894\\
15	14.9526814133513\\
16	8.8762032378143\\
17	7.11030528456357\\
18	7.32255422391679\\
19	7.6526589390271\\
20	7.68913246219816\\
21	7.47257588746873\\
22	7.18088963432021\\
23	6.95400423395732\\
24	6.82229840181994\\
25	6.7643228503869\\
27	6.69057986914486\\
29	6.57492208251663\\
37	6.16316204799334\\
44	5.83191895628216\\
50	5.58927452277467\\
52	5.53440264918102\\
58	5.40509183590358\\
60	5.35727243744662\\
62	5.34613964736288\\
64	5.36053550589494\\
69	5.45655882624713\\
74	5.51943327098143\\
78	5.54697954610545\\
83	5.55569684516272\\
88	5.54315407140314\\
94	5.50698002919138\\
101	5.4438477920423\\
109	5.35697980738878\\
115	5.2971064942571\\
119	5.2855408406345\\
124	5.2992941884678\\
138	5.37275489155625\\
146	5.38102261816279\\
154	5.36903871091441\\
164	5.33087753830561\\
178	5.28060143730048\\
183	5.284987363245\\
193	5.32753868356192\\
202	5.35415545106487\\
211	5.35963890867561\\
220	5.34500578513005\\
232	5.30402686130473\\
242	5.2749368260229\\
251	5.27689977737373\\
276	5.30010404500865\\
291	5.28153989839695\\
306	5.27003341168621\\
317	5.28674419502226\\
331	5.3047343757477\\
344	5.29991001722136\\
380	5.26896455058518\\
402	5.27957759536048\\
422	5.26213442761941\\
434	5.26154777730005\\
466	5.28087812518887\\
487	5.26085464624157\\
499	5.25886938087341\\
526	5.27225214351524\\
552	5.25774607489384\\
564	5.26145753054467\\
591	5.27037354205878\\
643	5.26752535684096\\
670	5.26169748905772\\
690	5.26233994170818\\
729	5.26382602812722\\
768	5.26403014939406\\
846	5.26228704609728\\
952	5.25857331475855\\
1000	5.25472387424206\\
};
\addlegendentry{$\accel\approx 1$};

\end{axis}
\end{tikzpicture}%

%% file: l1_mesh_clean.tikz
\begin{tikzpicture}

\begin{axis}[%
    width=\linewidth,
xmode=log,
xmin=1,
xmax=1000,
xminorticks=true,
ymode=log,
ymin=1,
ymax=1000,
yminorticks=true,
ylabel={$J_\gamma(u^i)$},
axis x line*=bottom,
axis y line*=left,
legend style={legend cell align=left,align=left,draw=none,font=\footnotesize}
]
\addplot +[line width=1.5pt]
  table[row sep=crcr]{%
1	265.593804912123\\
2	216.721208556578\\
3	160.312686547817\\
5	112.399526087534\\
6	97.6815654422975\\
7	84.6649577773914\\
8	71.7974553896324\\
9	57.6021026242014\\
10	39.9758483420687\\
11	16.1281912898135\\
12	26.9759452164809\\
13	63.4131639798152\\
14	27.019248390655\\
15	15.2182377357395\\
16	9.17335926851178\\
17	7.60516294371529\\
18	7.93234213561071\\
19	8.29692706262201\\
20	8.34132763260832\\
21	8.12809195145113\\
22	7.83896478268502\\
23	7.60022807118971\\
24	7.47013896569592\\
25	7.41776803148111\\
29	7.30053150527009\\
36	7.10290040140801\\
43	6.85541414886506\\
46	6.74036850447584\\
54	6.45864392974311\\
56	6.42951519943937\\
57	6.42896267585748\\
61	6.48079986304131\\
63	6.47994771379073\\
65	6.44304087458971\\
67	6.36821173443869\\
70	6.22382168393343\\
72	6.13248376164921\\
74	6.07572684088065\\
78	6.13374971869639\\
81	6.20151899218159\\
83	6.21252435273534\\
85	6.19385609879993\\
87	6.16693836270779\\
93	6.12463650483092\\
95	6.11106533494353\\
97	6.13372298016155\\
100	6.17383511604317\\
102	6.17867329379792\\
105	6.15787915090435\\
110	6.12332383009081\\
121	6.11577828779874\\
125	6.06676641168764\\
130	5.99912001444531\\
132	5.99354266852733\\
141	6.00841469194482\\
146	5.98859063030937\\
151	5.97866101911362\\
161	5.98591819726871\\
180	5.98084587084979\\
200	5.95339845310009\\
224	6.02281842532682\\
233	6.01503671589418\\
249	5.9804466498543\\
270	5.93575225812759\\
303	5.92463896892876\\
322	5.9065918828538\\
335	5.92495873826419\\
353	5.94255536587374\\
369	5.93491441038984\\
408	5.90479951391245\\
431	5.90301522167202\\
453	5.8985760492109\\
491	5.92629280788952\\
599	5.91292364072798\\
625	5.91487816326606\\
717	5.89470392776417\\
754	5.90203065168027\\
911	5.89185490161928\\
1000	5.89325658182411\\
};
\addlegendentry{$n=100$};

\addplot +[line width=1.5pt]
  table[row sep=crcr]{%
1	264.630897467534\\
2	219.590366500157\\
3	161.367990779432\\
5	112.618012674154\\
6	97.8896360613476\\
7	84.9839691706724\\
8	72.3413970457804\\
9	58.5271642757566\\
10	41.5772297670191\\
11	18.5606450096486\\
12	22.9570679064501\\
13	63.3537296140172\\
14	28.1370967386718\\
15	15.9057960061388\\
16	9.70133443634067\\
17	7.76796827854245\\
18	7.95979375071432\\
19	8.33035594072376\\
20	8.41088650958114\\
21	8.19981404645614\\
22	7.8916218691443\\
23	7.63388042734413\\
24	7.47505779365412\\
25	7.40067694170754\\
27	7.31869309051875\\
30	7.14916459771743\\
34	6.95511390687498\\
37	6.78828784677878\\
47	6.30288660624367\\
52	6.11324421294574\\
55	6.0330260429178\\
57	6.00289915061097\\
58	6.01770513284712\\
65	6.19002236017605\\
71	6.28036192086187\\
76	6.32710502938051\\
81	6.34558741977008\\
86	6.34040330302802\\
91	6.31322805333099\\
97	6.25835243955059\\
105	6.16390458654518\\
119	6.00796541561758\\
124	5.99857821611346\\
129	6.01531266516387\\
143	6.07434407762258\\
152	6.07720075468124\\
161	6.05620471663516\\
185	5.98559232355229\\
193	6.00232254673623\\
206	6.04055447153989\\
217	6.04669091093538\\
229	6.02705651090177\\
251	5.98157828261586\\
265	5.99409819875547\\
283	6.00275512124075\\
300	5.98318369270448\\
316	5.97134316312846\\
332	5.98961380779675\\
349	5.99797480797882\\
367	5.97992305923812\\
386	5.96780569796084\\
423	5.97715588389508\\
455	5.96702555198795\\
481	5.98091563577045\\
653	5.96445503583972\\
696	5.96436950442391\\
723	5.9647396783183\\
772	5.96319739390414\\
922	5.96064913817662\\
1000	5.95986559902243\\
};
\addlegendentry{$n=1000$};

\addplot +[line width=1.5pt]
  table[row sep=crcr]{%
1	264.284257934231\\
2	220.06856823259\\
3	161.68724088281\\
5	112.86321820171\\
6	98.1313139445443\\
7	85.2292280291229\\
8	72.5932052050112\\
9	58.7843792818035\\
10	41.8158412696854\\
11	18.6482062578594\\
12	23.0404663129192\\
13	63.8569900704235\\
14	27.358484721068\\
15	15.414628137874\\
16	9.31135613402176\\
17	7.477793637288\\
18	7.68354087232721\\
19	8.03476895429705\\
20	8.08959742048543\\
21	7.87572540973949\\
22	7.57844822415757\\
23	7.33886997706206\\
24	7.19978398706827\\
25	7.13837342978596\\
27	7.06850144105833\\
29	6.9579557918808\\
32	6.80956145777918\\
35	6.66895078874881\\
44	6.26020733035688\\
51	6.00727689220312\\
57	5.84507638910692\\
60	5.79873320157095\\
63	5.78792490406849\\
65	5.80640548362703\\
71	5.89445149309327\\
76	5.93778502100541\\
81	5.95533621481595\\
87	5.94862707882464\\
94	5.91438320567235\\
102	5.85403042796414\\
114	5.76487438419991\\
119	5.76562444891287\\
130	5.8010782779371\\
141	5.82604107236439\\
151	5.82353010827824\\
162	5.79790739376757\\
179	5.75686373303256\\
188	5.77148951025247\\
206	5.81264336362783\\
217	5.80856175467942\\
230	5.78039737138847\\
243	5.7539333732895\\
256	5.75584066064664\\
282	5.76484339720665\\
316	5.75349872616127\\
344	5.77160661085017\\
364	5.74987249776767\\
382	5.74244633116522\\
417	5.74791415657611\\
443	5.74135505161924\\
476	5.75302272957799\\
520	5.74480958461739\\
547	5.74443062215269\\
578	5.74184055072443\\
610	5.74452376698343\\
651	5.74370121215112\\
685	5.73979615667506\\
722	5.7445237960527\\
882	5.73839103092105\\
972	5.7373037128422\\
1000	5.73618100782654\\
};
\addlegendentry{$n=10000$};
\end{axis}
\end{tikzpicture}%

%% file: l1_gamma_mu0.tikz
\begin{tikzpicture}

\begin{axis}[%
    width=\linewidth,
xmode=log,
xmin=1,
xmax=1000,
xminorticks=true,
ymode=log,
ymin=1,
ymax=100000,
yminorticks=true,
ylabel={$J_\gamma(u^i)$},
axis x line*=bottom,
axis y line*=left,
legend style={legend cell align=left,align=left,draw=none,font=\footnotesize}
]
\addplot +[line width=1.5pt]
  table[row sep=crcr]{%
1	17.6895629636488\\
2	38858.8792017715\\
3	9820.96829656641\\
4	2484.77358254814\\
5	631.450720789664\\
6	163.47320684528\\
7	45.8572504748174\\
8	18.0470670316551\\
9	15.9984914206675\\
10	15.3414906639755\\
11	14.2031567854222\\
12	13.3453406102176\\
13	12.4442040815057\\
14	11.6299830997182\\
15	10.8443136543835\\
16	10.1078592425869\\
17	9.40861133980343\\
18	8.74851785975508\\
19	8.12398478916647\\
20	7.53403799363818\\
21	6.97690095950202\\
22	6.45143542397366\\
23	5.95653731166738\\
24	5.49139445741946\\
25	5.05537703674391\\
27	4.2692830108915\\
29	3.59716678633328\\
34	2.41948799858425\\
35	2.26809569050824\\
36	2.14223352369406\\
37	2.03981490528007\\
38	1.95831073065658\\
39	1.89489724821574\\
40	1.84664314527989\\
41	1.81069858475871\\
42	1.78445190735813\\
43	1.76563294327861\\
44	1.75235851159957\\
46	1.73679166806457\\
48	1.72958884388815\\
52	1.72496368187563\\
65	1.72392862276674\\
1000	1.72392616454349\\
};
\addlegendentry{$\gamma=10^{-1}$};

\addplot +[line width=1.5pt]
  table[row sep=crcr]{%
1	254.935720557619\\
2	38977.0097021335\\
3	9938.18309348003\\
4	2600.18957573773\\
6	268.389084297124\\
7	138.236823945589\\
8	98.8293873429431\\
9	103.580607736088\\
10	100.901540448529\\
11	101.943124075175\\
12	101.18701212373\\
13	101.317026749222\\
15	100.921828659034\\
26	99.0902821066428\\
34	97.7339995006019\\
43	96.162670816354\\
52	94.5384396687498\\
61	92.8557280758188\\
71	90.9096538894601\\
81	88.8732800186111\\
90	86.9534691100041\\
99	84.9405673617909\\
108	82.8225527347726\\
117	80.5852633743159\\
125	78.4829375280012\\
133	76.2592818632851\\
141	73.8980284297184\\
148	71.7041022794182\\
155	69.3753066742235\\
162	66.8944126112302\\
169	64.2415740632482\\
175	61.813704420302\\
181	59.2268778673787\\
187	56.4635345464647\\
193	53.504840207759\\
199	50.3310615361326\\
204	47.5081505046409\\
209	44.5129082501113\\
214	41.3393004127484\\
219	37.987406332744\\
224	34.4673628984236\\
229	30.8027932652111\\
234	27.0325358330039\\
239	23.2189238096059\\
244	19.4442073584319\\
249	15.8090398738062\\
256	11.3363407239253\\
266	7.12273375867926\\
271	5.90153453006725\\
276	5.10420944196523\\
280	4.68748448961121\\
284	4.40885786623116\\
288	4.22655580885321\\
292	4.10932881464537\\
296	4.03492327472571\\
300	3.98814398925604\\
305	3.95353882603148\\
311	3.9317542483811\\
320	3.91784902524446\\
338	3.91068753677179\\
427	3.9066299117069\\
1000	3.90367851384937\\
};
\addlegendentry{$\gamma=10^{-3}$};

\addplot +[line width=1.5pt]
  table[row sep=crcr]{%
1	264.775203100877\\
2	38987.1701995903\\
3	9948.34384026801\\
4	2610.35113003358\\
5	753.483480976563\\
6	278.54965235366\\
7	148.397342901899\\
8	108.98988865484\\
9	113.786306518246\\
10	111.124068264849\\
11	112.20108583417\\
12	111.46944594777\\
13	111.630490803392\\
15	111.292253212373\\
28	109.506071916091\\
38	108.096655833293\\
48	106.636157842229\\
59	104.964747678744\\
70	103.218079683306\\
82	101.216460755178\\
94	99.1013677663233\\
105	97.0491040324431\\
116	94.873009724272\\
127	92.555143505315\\
137	90.3067813100486\\
147	87.9036652803469\\
156	85.5890205977022\\
165	83.1090966156888\\
174	80.4391134919934\\
182	77.882401194836\\
190	75.1275462956182\\
197	72.5320009555161\\
204	69.7406716329215\\
211	66.7287181171299\\
218	63.4688611384304\\
224	60.4555574187562\\
230	57.2209783194763\\
236	53.7491190003423\\
242	50.0287899185635\\
248	46.0579880008989\\
254	41.8527306231822\\
260	37.4582639773854\\
266	32.9492089111243\\
272	28.4294868087816\\
280	22.7316182384965\\
291	16.1991378423101\\
306	10.392373986594\\
313	8.79863257047939\\
319	7.84654662758275\\
325	7.17052015942277\\
331	6.69471938673637\\
335	6.4729196555919\\
339	6.33408743238985\\
344	6.22339069190025\\
349	6.15282361394987\\
355	6.09963728896253\\
364	6.05234496995482\\
377	6.01370210678981\\
402	5.97164472187936\\
469	5.90166260129891\\
542	5.83067613994567\\
639	5.74566766086863\\
677	5.72432555478859\\
793	5.66691321653058\\
823	5.65996457832012\\
850	5.64997606761083\\
894	5.61256920382788\\
925	5.60162987259669\\
1000	5.5888849782799\\
};
\addlegendentry{$\gamma=10^{-6}$};
\end{axis}
\end{tikzpicture}%

%% file: l1_gamma_mu1.tikz
\begin{tikzpicture}

\begin{axis}[%
    width=\linewidth,
xmode=log,
xmin=1,
xmax=1000,
xminorticks=true,
ymode=log,
ymin=1,
ymax=400,
yminorticks=true,
ylabel={$J_\gamma(u^i)$},
axis x line*=bottom,
axis y line*=left,
legend style={legend cell align=left,align=left,draw=none,font=\footnotesize}
]
\addplot +[line width=1.5pt]
  table[row sep=crcr]{%
1	17.6943572405265\\
2	80.4117640011277\\
3	42.7351219218423\\
4	26.2527394818344\\
5	17.5945491479443\\
6	12.4223012722403\\
7	9.03562195856616\\
8	6.67640389721441\\
9	4.9760210908269\\
11	2.88307761884977\\
12	2.32236945795181\\
13	1.99908033930378\\
14	1.83842012724828\\
15	1.76890282828259\\
16	1.7413304850853\\
17	1.73065157415276\\
19	1.72470825375143\\
26	1.72336784494132\\
1000	1.72335239726731\\
};
\addlegendentry{$\gamma=10^{-1}$};

\addplot +[line width=1.5pt]
  table[row sep=crcr]{%
1	254.390542953971\\
2	209.365851603513\\
3	151.324745505514\\
5	102.66489558111\\
6	87.9366346621324\\
7	75.0162098634996\\
8	62.3492595571944\\
9	48.5121435745277\\
10	31.5956066865772\\
11	9.64841074716493\\
12	12.4246680668834\\
13	50.6726066173839\\
14	20.5155859473975\\
15	8.3062964602099\\
16	4.40877254940425\\
17	3.58186592747078\\
18	3.58604521200488\\
19	3.69472105593241\\
20	3.72160413049353\\
21	3.67068161663971\\
22	3.59463124872802\\
23	3.53526781052457\\
24	3.50595355644436\\
28	3.48857949183264\\
32	3.46298796074991\\
43	3.43347932543835\\
55	3.42482030545049\\
1000	3.42400135359638\\
};
\addlegendentry{$\gamma=10^{-3}$};

\addplot +[line width=1.5pt]
  table[row sep=crcr]{%
1	263.968127278283\\
2	220.053280251243\\
3	161.904125189854\\
5	113.188715639215\\
6	98.4559279239818\\
7	85.5323202226018\\
8	72.8498683232027\\
9	58.9524927755865\\
10	41.8175987353567\\
11	18.1225971003274\\
12	24.2449436463078\\
13	64.749098209718\\
14	25.8097610222685\\
15	14.6674301407057\\
16	8.94318950471537\\
17	7.31479462866941\\
18	7.51617857654296\\
19	7.80048581993479\\
20	7.79885322491689\\
21	7.57200994555499\\
22	7.29438324935688\\
23	7.09310647575703\\
24	6.98995953979233\\
27	6.8970450973003\\
29	6.76153493866359\\
41	6.08530095402624\\
50	5.69960864957058\\
57	5.56341099659991\\
58	5.56220329576871\\
65	5.73390586012273\\
69	5.80513539416587\\
73	5.84708752922336\\
78	5.86483473006035\\
83	5.85539634196108\\
88	5.8234085437243\\
94	5.76164676946073\\
111	5.55642554952268\\
114	5.54347245703342\\
120	5.55080848585882\\
145	5.61502553002266\\
154	5.59467369472112\\
181	5.51083874654103\\
201	5.57029232026187\\
210	5.57264461591429\\
220	5.55321348902662\\
244	5.49628584529148\\
269	5.52788451812462\\
307	5.48946688006291\\
342	5.51325943249602\\
375	5.48976584692189\\
402	5.50118718313742\\
708	5.47985850480608\\
1000	5.47321919077834\\
};
\addlegendentry{$\gamma=10^{-6}$};
\end{axis}
\end{tikzpicture}%

%% file: linf_accel_clean.tikz
\begin{tikzpicture}

\begin{axis}[%
width=\linewidth,
xmode=log,
xmin=1,
xmax=10000,
xminorticks=true,
ymode=log,
ymin=1,
ymax=10000000000000,
yminorticks=true,
ylabel={$J_\gamma(u^i)$},
axis x line*=bottom,
axis y line*=left,
legend style={legend cell align=left,align=left,draw=none,font=\footnotesize}
]
\addplot +[ line width=1.5pt]
  table[row sep=crcr]{%
1	1736900412084.4\\
2	442718030565.162\\
3	439184179267.963\\
4	432196251894.369\\
5	418466299130.764\\
6	391897683427.826\\
7	342905822808.243\\
8	270274974782.816\\
9	236566362676.968\\
10	249553091641.343\\
11	242698889760.63\\
12	245759816018.232\\
13	244015776189.616\\
14	244646183174.947\\
16	244150661799.753\\
46	239372260545.686\\
65	236154095806.145\\
86	232360832448.7\\
107	228289263739.194\\
129	223685184050.456\\
150	218921286081.776\\
171	213742757161.996\\
191	208366018645.813\\
210	202792848026.9\\
228	197028688460.877\\
245	191083654859.873\\
261	184973691588.068\\
276	178721856854.919\\
290	172359678245.553\\
304	165410829729.596\\
317	158353435259.6\\
329	151244939337.799\\
341	143483497586.887\\
352	135715416939.191\\
363	127237861139.724\\
374	117960169982.696\\
384	108751548390.594\\
394	98735231178.3941\\
404	87860416254.8045\\
413	77330540931.6681\\
422	66149601913.5724\\
431	54469757648.5052\\
440	42602192592.6367\\
449	31112656284.4719\\
458	20746049198.2318\\
467	12309790124.5216\\
476	6356789207.12308\\
486	2545742933.51381\\
498	664640145.348103\\
511	116260966.231633\\
545	1071758.11994429\\
552	652287.201568003\\
558	511046.347467691\\
564	444820.38049783\\
573	390292.201281492\\
642	172016.377814823\\
684	101021.548064119\\
721	63069.5485644552\\
771	33341.1968492597\\
811	20815.9665789242\\
829	17940.8531623989\\
844	16888.7831403429\\
853	16743.1430226237\\
858	16792.7424130833\\
863	17004.6355767757\\
867	17405.8661905535\\
885	20218.7761906568\\
988	58950.5464638109\\
1016	73244.6788505235\\
1078	108618.938186831\\
1108	126459.800249958\\
1138	145105.233651942\\
1178	169945.673914958\\
1249	211908.292233857\\
1309	242673.273325377\\
1324	250475.697328727\\
1383	286206.594615874\\
1465	326864.985468075\\
1493	338740.758726709\\
1535	354626.103825345\\
1605	375211.971997796\\
1638	381671.032759575\\
1674	390844.479937733\\
1705	395091.671367077\\
1744	401234.898748925\\
1759	403015.611703582\\
1781	405325.998845255\\
1799	406751.336655405\\
1819	406699.375170624\\
1845	408265.720690869\\
1917	408239.001190247\\
1953	406168.718173561\\
1973	405784.302754499\\
1998	405916.702686368\\
2032	405005.35790896\\
2062	405261.869750556\\
2104	400527.484953887\\
2205	397101.50986669\\
2239	395218.458099817\\
2252	394570.381368819\\
2297	391686.613522404\\
2333	391137.336255315\\
2360	390193.307138367\\
2399	385648.828658402\\
2424	385067.728097882\\
2469	383725.233684482\\
2503	380471.023515486\\
2553	376530.444086032\\
2585	373716.066113558\\
2609	372168.747405822\\
2632	368427.249172702\\
2733	360990.700485288\\
2820	352830.055583799\\
2852	349386.115061185\\
2893	347605.983307142\\
2924	347550.626090053\\
3010	339430.075396733\\
3083	335360.117307129\\
3157	329673.071465221\\
3229	324998.040214715\\
3295	323472.771484402\\
3342	322293.325532253\\
3426	312254.706178386\\
3454	310074.55959605\\
3493	307035.823018547\\
3559	302972.895555816\\
3592	299062.232791503\\
3893	280892.082402356\\
3947	278448.354081993\\
3963	276101.519046459\\
3992.99999999999	272769.863898852\\
4074	269811.999394365\\
4122	267645.796902991\\
4167	265161.560117247\\
4231	262165.881372024\\
4448.99999999999	253519.44914192\\
4476	252278.237910389\\
4577	246679.775410804\\
4669	243823.866148785\\
4705.99999999999	240625.991607819\\
4797	236658.643804057\\
5057	220686.899890767\\
5106	218506.948521495\\
5216	211993.673568899\\
5284	207631.464153753\\
5369.99999999999	203909.223459038\\
5443	199232.389043902\\
5669	189176.197401582\\
5775	184365.786807547\\
5818	182598.105206169\\
5880	180223.885528681\\
5970	178003.982872936\\
6085.99999999999	172846.27384103\\
6164	169817.840767026\\
6220	167526.082578336\\
6321	165523.973945783\\
6359	164358.024265798\\
6483	159750.430159886\\
6577	157346.196335464\\
6652	155568.181738556\\
6774	153059.365371337\\
6856	149757.592250963\\
7004	147025.98263558\\
7153	143839.464247171\\
7276	141839.478740645\\
7331	139928.725100962\\
7532	136659.574588394\\
7595	135129.735827142\\
7713.99999999999	133669.061579681\\
7809.99999999999	132364.269189283\\
8112	127547.264692317\\
8187	126619.30897045\\
8300	124735.775618275\\
8398	123686.501316686\\
8497	122362.868030536\\
8543	121623.175305296\\
8625	120484.1061448\\
8842.99999999999	118933.477901011\\
8931	118109.535025156\\
9014.99999999999	117157.388094264\\
9097	116251.091936014\\
9206.99999999999	115334.828727858\\
9280	114577.02634913\\
9376.99999999999	113407.930867555\\
9463	112708.285615359\\
9600	112024.22787741\\
9831	109972.907851781\\
9973.99999999999	107670.815219711\\
10000	107464.509359604\\
};
\addlegendentry{$\accel=0$};

\addplot +[ line width=1.5pt]
  table[row sep=crcr]{%
1	1736900412084.4\\
2	325431382746.353\\
3	283603543916.803\\
4	243351404000.856\\
5	205330132509.881\\
6	169499706938.051\\
7	135291525149.864\\
8	101791706948.112\\
9	68070503535.3361\\
10	34192868493.4559\\
11	5544626872.20058\\
12	10689705464.6893\\
13	93380509256.4224\\
14	17860469206.8732\\
15	4521123748.27675\\
16	929180290.035262\\
17	192750252.583867\\
18	182421196.698756\\
19	264962798.694289\\
20	280319590.191361\\
21	234752158.450365\\
23	126973517.544583\\
24	107914729.674012\\
25	105490066.413412\\
26	106317074.590499\\
27	102792753.847955\\
28	94265105.3691913\\
29	83614505.4447228\\
31	65669154.9075006\\
35	44907144.3220858\\
37	36460058.2461073\\
39	29307563.2544219\\
41	22971710.4533817\\
43	17364414.9564842\\
45	12714750.9773775\\
47	9041664.85035963\\
49	6219860.38288117\\
51	4121599.5168021\\
53	2603182.47072568\\
55	1531568.28228924\\
57	816191.781833857\\
59	388619.652472675\\
65	36954.0446993877\\
66	27950.6350235539\\
67	22700.9770886002\\
68	20224.7312282081\\
69	20003.9956630219\\
70	21510.1545279495\\
70.9999999999999	24206.2677813626\\
72	28203.9775338548\\
74	41259.4511383226\\
78	90033.6815797999\\
79	103371.613587778\\
80	115029.381038398\\
81	124137.562036694\\
82	130424.163073353\\
83	133974.076554739\\
84	135049.768051889\\
85	134213.534888447\\
86	131686.465309854\\
86.9999999999999	127389.716153955\\
89	114879.748990746\\
91	99785.2748700107\\
95	72018.1107038538\\
98	55549.8494728801\\
100	45799.1358052029\\
101	43220.6466648733\\
102	43007.7773228361\\
103	45160.764083537\\
104	49413.996246277\\
107	71723.7168588716\\
109	89605.4026711705\\
111	105151.863162166\\
113	116984.414498177\\
119	149844.899237531\\
122	179050.5611313\\
132	327893.987338968\\
135	375189.322346378\\
138	416866.629888033\\
141	449334.012260605\\
143	464303.493428009\\
145	473347.638276232\\
147	476417.858921754\\
148	475717.15846519\\
150	470023.259006659\\
152	459078.935328988\\
154	443525.89668356\\
157	412999.95275021\\
160	375941.835197293\\
164	320955.25452062\\
168	264861.295984522\\
173	200414.013122083\\
179	137396.485663677\\
188	74510.5412700063\\
197	42117.8542837383\\
199	38976.8067088788\\
201	37136.1444817959\\
203	36156.6457991061\\
205	35634.5935814198\\
209	34861.8620024804\\
218	32380.1122407222\\
219	32359.9490191922\\
220	32447.2320505766\\
222	32928.2617293237\\
225	34199.9224304409\\
230	36535.5416499682\\
231	36682.3695777065\\
232	36463.0517083861\\
234	34912.2166799897\\
236	32347.5321405625\\
243	23903.2508371149\\
245	23154.2149069601\\
246	23069.7197077658\\
247	23128.8120024017\\
249	23605.5047838093\\
252	24954.032950189\\
256	26834.5715149853\\
259	27669.0436140407\\
261	27876.4085514516\\
263	27832.3234131168\\
267	27650.6744439948\\
269	27770.6784417276\\
271	28089.0353766824\\
274	28913.842820537\\
286	33107.4354281925\\
290	33826.4246864869\\
297	34696.4841498597\\
302	35091.9819611643\\
305	35053.6275328667\\
308	34830.3175465149\\
312	34287.0485229654\\
317	33293.8102428955\\
322	31922.0683625899\\
328	29891.0580900246\\
333	27834.9602813324\\
341	24155.331381057\\
356	17988.2993484067\\
362	15548.2742316572\\
379	9969.84414381981\\
394	7382.94248878489\\
407	5372.92547450159\\
413	4901.90340192754\\
423	4240.55542899455\\
433	3625.33970247397\\
436	3570.04946142101\\
439	3584.7808089773\\
445	3635.14670048238\\
448	3629.84575004075\\
454	3582.5922319339\\
457	3472.14005573193\\
467	3042.85637391959\\
469	3023.32316369509\\
472	3048.85350764889\\
475	3130.38297975349\\
482	3453.13816778534\\
488	3704.39731547902\\
491	3749.63598879626\\
494	3719.66194129192\\
502	3581.59532129117\\
505	3594.04775846589\\
508	3659.21995654446\\
513	3871.79659990536\\
520	4191.32919960042\\
523.999999999999	4278.90064307125\\
527	4298.51171414583\\
532	4292.07010909226\\
536	4290.29908735917\\
540	4331.8563910811\\
545	4446.71576286367\\
556	4706.91508811692\\
560	4740.24363516004\\
563	4722.38115311388\\
567	4651.61445635466\\
582	4357.13352017042\\
592.999999999999	4235.43929309261\\
598	4127.0889885058\\
605	3881.06587488139\\
637	2811.36270671912\\
649	2419.2753807023\\
667	1904.47367461336\\
674	1814.29020015512\\
679	1759.90032480573\\
685	1646.70956290553\\
706	1257.01486703304\\
710	1242.71621880131\\
714	1242.8823630167\\
721	1249.89422643176\\
724.999999999999	1236.40650083716\\
731	1193.90581152295\\
737	1130.52417329371\\
756.999999999999	903.494902348749\\
760	904.311667949453\\
763	914.96347165852\\
776	979.904352005537\\
779	972.291971181293\\
784	940.765356847296\\
791	902.934520641313\\
794	900.148739957234\\
798	909.678288080644\\
802.999999999999	940.23107175062\\
820	1071.75715990957\\
824	1078.57168451304\\
829	1073.09667036071\\
847.999999999999	1034.30636726075\\
851.999999999999	1041.26585083195\\
857	1063.091984459\\
869.999999999999	1129.81878695246\\
875.999999999999	1139.85052120347\\
890	1145.53285870692\\
898	1170.05026036384\\
910.999999999999	1211.7724219572\\
915.999999999999	1212.93210779416\\
920	1201.77565505833\\
926	1164.83375199595\\
949	1003.76695510824\\
955.999999999999	995.351878140772\\
964	987.484571486853\\
970	970.023436865873\\
976.999999999999	931.298330116895\\
1009	749.435346014489\\
1025	708.921989126476\\
1033	659.959672219404\\
1058	502.517941455817\\
1062	504.381073605327\\
1069	523.604434151399\\
1076	536.505732756977\\
1080	533.186411811205\\
1085	515.926450040782\\
1095	456.612132326724\\
1107	404.001746527045\\
1113	397.762046363775\\
1117	400.773222760366\\
1124	415.931185853029\\
1135	439.828722723005\\
1138	440.787504511821\\
1142	433.033908135802\\
1149	398.185962880552\\
1164	331.04035109869\\
1168	327.149794054954\\
1172	329.719706088365\\
1178	344.168484100291\\
1197	404.497356294669\\
1202	402.349622449975\\
1208	388.964440460291\\
1223	355.781501181004\\
1226	354.74807311694\\
1231	358.262792153385\\
1237	373.306940796296\\
1262	461.87441394997\\
1266	457.722430295267\\
1274	436.671701150466\\
1294	375.726317570295\\
1299	376.34158236448\\
1305	384.654567493886\\
1314	411.627135112801\\
1328	454.062959669728\\
1331	455.255810738025\\
1335	451.008357453162\\
1342	431.907672360122\\
1360	386.111711102818\\
1367	382.899075232264\\
1373	386.55448135486\\
1386	399.582741234705\\
1391	396.232483601594\\
1398	381.510796387138\\
1427	315.984824469348\\
1434	314.277878765647\\
1449	317.851710012003\\
1455	314.088655317894\\
1463	300.654504450757\\
1474	270.328577399167\\
1497	212.408325350975\\
1506	206.257921727717\\
1512	207.350132029643\\
1529	217.411347938168\\
1534	215.453353650992\\
1541	206.535292405002\\
1571	165.930366708001\\
1577	167.363684712907\\
1587	175.018550580749\\
1611	193.274712366291\\
1617	192.754113388352\\
1625	187.185856893793\\
1649	169.044327326721\\
1655	169.860613444942\\
1680	180.522220348909\\
1684	178.89337363469\\
1692	169.149355386953\\
1711	148.090471770615\\
1715	147.604576546033\\
1721	149.424496104178\\
1730	157.28570462719\\
1753	181.680386488676\\
1759	180.807710653608\\
1768	174.089169457513\\
1785	162.920694944691\\
1790	162.475485794661\\
1797	164.082466685562\\
1808	170.873096128514\\
1825	181.091207328119\\
1831	180.776731037979\\
1839	176.793527127773\\
1866	162.160335441663\\
1874	162.248271392297\\
1886	165.289103275063\\
1897	167.237501123248\\
1904	165.978971358439\\
1915	159.765799069478\\
1943	144.002735052354\\
1952	143.216670419396\\
1970	143.589888525634\\
1979	141.347483858481\\
1989	134.939748291011\\
2018	116.717794499077\\
2029	112.865752713166\\
2048	102.090984252176\\
2072	90.9151962309844\\
2099	83.4916585312435\\
2139	71.2394513929471\\
2184	62.9626759870524\\
2203	59.6407072769375\\
2211	60.1753903690974\\
2224	63.1456149652716\\
2235	64.7672344155545\\
2241	64.1610054342971\\
2251	60.8184076082554\\
2280	51.5155978118963\\
2285	51.3008259946282\\
2292	51.7050884575499\\
2305	54.1218161847344\\
2320	56.3321966591011\\
2327	55.8852452998852\\
2338	53.4490789417149\\
2361	49.035476682077\\
2366	48.9417373253584\\
2375	49.5800017807835\\
2395	53.2074691566986\\
2408	54.4564581617382\\
2417	54.0205133412045\\
2443	51.8886014560065\\
2454	52.3994726274417\\
2463	52.1908614840149\\
2491	50.4892797069867\\
2504	50.63183911323\\
2518	51.3926295685798\\
2535	52.3076852186374\\
2544	51.8406983269176\\
2559	49.5406585677903\\
2592	45.0114288856882\\
2611	44.1406785897288\\
2623	43.4358321297448\\
2637	41.4069307629552\\
2686	34.1372277082962\\
2717	32.4659719002643\\
2742	29.2030957716871\\
2764	27.4621171470969\\
2777	27.2992879570175\\
2796	27.3642916943118\\
2808	27.0047758513658\\
2845	25.4470539994972\\
2855	25.6313935618491\\
2880	26.5956826939249\\
2891	26.3599360187013\\
2910	25.9202151763939\\
2920	26.1539935616632\\
2935	27.2649009752163\\
2976	30.8870676477673\\
3008	31.5905201137232\\
3027	33.1851975217461\\
3060	36.0041171654949\\
3080	36.7257332399527\\
3101	39.6886301671458\\
3118	41.1895432481841\\
3126	40.7689484585993\\
3142	38.1759345581807\\
3173	33.9932266049627\\
3182	33.817923397671\\
3203	34.1161976098842\\
3212	33.7280316069793\\
3225	32.0071023600062\\
3238	28.1441693038889\\
3276	19.1249668327802\\
3288	18.5573092902026\\
3296	18.8143906984016\\
3314	19.864656089055\\
3321	19.5100049448956\\
3333	17.4153651689015\\
3372	11.6587324379648\\
3380	11.5112108218034\\
3388	11.6927482793043\\
3404	12.9097162990536\\
3418	13.6712056380418\\
3425	13.4759600514331\\
3439	12.1067504965076\\
3467	10.0915941497831\\
3472	10.0563394545906\\
3481	10.2170709371641\\
3496	11.1032714660671\\
3518	12.3101521100055\\
3526	12.1003099888224\\
3542	10.8554950025263\\
3567	9.57284631415367\\
3577	9.63238897722342\\
3591	10.239857481968\\
3622	11.757219818063\\
3632	11.601274794811\\
3653	11.1353825161692\\
3662	11.2597070492211\\
3676	11.9676522106945\\
3746	17.2572816768712\\
3767	18.5478157473695\\
3826	23.2683164593164\\
3866	23.1261941823654\\
3891	23.3364286833216\\
3902.99999999999	23.1018518909868\\
3920	22.0162736193762\\
4013	14.7493966234245\\
4077	10.5861221302997\\
4109	9.44406285326258\\
4151	8.07749874416768\\
4158	8.05711334486605\\
4175	8.15044696810921\\
4187	8.19251769290361\\
4200	8.1055028382794\\
4235	7.71406412917906\\
4246.99999999999	7.79528126126808\\
4271	8.28913069424455\\
4308	8.91413652312047\\
4330	9.41820093220138\\
4363	11.0690663823854\\
4442.99999999999	15.9989473869065\\
4469	16.6923650105383\\
4500	17.9553079472838\\
4536	19.6138685416246\\
4547	19.3515671511475\\
4629	15.4365149139017\\
4652	13.482688202854\\
4763	6.414177026285\\
4795	5.65982496002171\\
4807	5.60896169879785\\
4822.99999999999	5.67168549846189\\
4851	5.76835189480739\\
4870	5.80965790974096\\
4931	6.10514782574389\\
4946	6.04546755552035\\
5014	5.56937466844389\\
5038.99999999999	5.6229497498108\\
5061	5.73599867765888\\
5082	6.07281766963084\\
5115	7.28783703815145\\
5213	12.6686008727933\\
5257	14.5512556241583\\
5296	15.2373829620869\\
5308	15.279233894054\\
5325	15.1279005602744\\
5351	14.4071898991226\\
5402	12.4015952550978\\
5447	9.43789993451032\\
5528	6.21895580152903\\
5557	5.76841728631101\\
5569.99999999999	5.84931245977672\\
5706	7.81792092917648\\
5720	7.73542774306364\\
5751	7.05037382213046\\
5783.99999999999	6.70788846283879\\
5805	6.63671215263824\\
5822.99999999999	6.33656576903905\\
5876	5.19318976006947\\
5888	5.26578445654561\\
5977	6.6528498388492\\
6060	9.5330634442997\\
6130	10.7887342487648\\
6147	10.6554643515872\\
6185	9.87709788418261\\
6241	8.41766924706205\\
6270	6.62675086964825\\
6306.99999999999	5.46630826009905\\
6317	5.53444419144467\\
6345	6.00721172562442\\
6356	5.88325793584438\\
6397.99999999999	4.90928418132709\\
6410	5.00330279839752\\
6437	5.74030207766002\\
6465	6.30791871752699\\
6480	6.19277810048122\\
6508	5.95828270989692\\
6523	6.04026343441989\\
6569.99999999999	6.70083712629354\\
6582.99999999999	6.59459851980466\\
6636	5.82550830426377\\
6651	5.89076941784742\\
6678	6.0728335878334\\
6691	5.99914233073592\\
6736	5.42575313530056\\
6748.99999999999	5.50419287216661\\
6781	6.19632843903958\\
6815	6.58383690797302\\
6836	6.69818152658557\\
6860.99999999999	7.20727869505664\\
6926	8.73804512048386\\
6942	8.77186782855817\\
6956	8.60128657199519\\
6998	8.03481271140998\\
7015	8.11898211395964\\
7050	8.56865096366507\\
7062	8.4347660265975\\
7090	7.44183497728382\\
7135.99999999999	6.49579324627622\\
7154	6.53954167172561\\
7166	6.59118368634079\\
7179	6.48248893838669\\
7204.99999999999	5.73085767505117\\
7241.99999999999	5.09432659449956\\
7256	5.18908474508052\\
7296	5.71792843555242\\
7310	5.61437890940636\\
7353	5.12674985972719\\
7368	5.20402660977351\\
7415.99999999999	5.77192374084306\\
7430.99999999999	5.67978450433187\\
7479	5.2127346498007\\
7497	5.28839002729626\\
7532	5.52933356660094\\
7548.99999999999	5.43162757363757\\
7589	4.81720592685128\\
7616	4.46733329658084\\
7638.99999999999	4.54227739885113\\
7660	4.76677707764464\\
7697.99999999999	5.42313803364773\\
7713	5.31414148598439\\
7735	5.15588557382126\\
7757	5.23926313031865\\
7795	5.52386548676486\\
7822	5.80522884608567\\
7837	5.69959695594747\\
7897	5.21048623224096\\
7921	5.13481377741584\\
7938	5.14811460943984\\
7954	5.06946872108338\\
8045.99999999999	4.22548863861968\\
8074	4.2729476100505\\
8091	4.21335848614791\\
8129	4.03715301947544\\
8155	4.08289832660744\\
8278.99999999999	4.43978670960607\\
8332	4.47211588684372\\
8357	4.4139324395542\\
8464	3.98713481535678\\
8574	3.55233306906605\\
8620	3.5871483598566\\
8657	3.72251664651747\\
8799	4.24706761521417\\
8837	4.26937059422807\\
8864	4.22429007466164\\
8945	3.91084928953765\\
9048	3.58805500510501\\
9069	3.57953318571735\\
9098.99999999999	3.62036463483077\\
9162	3.829710611212\\
9273.99999999999	4.19273728917256\\
9320.99999999999	4.21031573540062\\
9357	4.17680694990015\\
9403	4.03127049345631\\
9541	3.62477043834379\\
9574	3.60911666019309\\
9610.99999999999	3.64435258734186\\
9769.99999999999	3.92457097811746\\
9810	3.88997053038856\\
9867	3.76250403029644\\
9956.99999999999	3.40354790947132\\
10000	3.26563550132889\\
};
\addlegendentry{$\accel\approx 1$};
\end{axis}
\end{tikzpicture}%

%% file: linf_mesh_clean.tikz
\begin{tikzpicture}

\begin{axis}[%
width=\linewidth,
xmode=log,
xmin=1,
xmax=10000,
xminorticks=true,
ymode=log,
ymin=1,
ymax=10000000000000,
yminorticks=true,
ylabel={$J_\gamma(u^i)$},
axis x line*=bottom,
axis y line*=left,
legend style={legend cell align=left,align=left,draw=none,font=\footnotesize}
]
\addplot +[ line width=1.5pt]
  table[row sep=crcr]{%
1	1740380476096.3\\
2	326493729505.92\\
3	283744117798.604\\
4	242683114074.142\\
5	203973710602.875\\
6	167561984911.402\\
7	132857847526.484\\
8	98939822646.9424\\
9	64932668151.3452\\
10	31162273964.1731\\
11	3965603163.97986\\
12	13866988164.8996\\
13	89580578811.984\\
14	18306644910.9169\\
15	4624423943.28596\\
16	949726784.277586\\
17	196984368.182304\\
18	188585950.288676\\
19	276215256.94277\\
20	294328984.943921\\
21	248089922.193109\\
23	133697258.886336\\
24	111760228.910806\\
25	107887509.620919\\
26	108387260.714213\\
27	105056008.153283\\
28	96758374.9708375\\
29	86271163.9989878\\
31	68289992.2667405\\
35	47040018.9960353\\
37	38247048.0237137\\
39	30777386.1300517\\
41	24205284.3488339\\
43	18378074.8116485\\
45	13505295.803438\\
47	9628223.9076516\\
49	6639877.89311404\\
51	4419376.95342195\\
53	2810039.89547785\\
55	1675469.88794392\\
57	905087.003776238\\
59	441193.071317678\\
61	192037.754410579\\
64	51810.3697696592\\
66	24978.210282792\\
68	10469.7756638562\\
69	7871.79692688179\\
70	8035.74841631587\\
70.9999999999999	10462.0909585765\\
76	55999.4794595491\\
77	68684.8803428079\\
78	78836.8820307803\\
81	101260.98520742\\
82	110268.394959182\\
83	116136.72427194\\
84	117212.776002409\\
85	113946.833240999\\
86.9999999999999	103387.578257489\\
92	85759.0916600642\\
93	80879.9844028542\\
94	72937.1638562978\\
96	53189.9282772561\\
98	37675.4646469169\\
99	33943.7646249461\\
100	33522.7080785881\\
101	34988.4751845849\\
102	39153.6167684046\\
107	80780.3623298372\\
109	96267.3300161628\\
110	101420.104507217\\
112	107466.128977718\\
113	109374.418714616\\
115	111616.77007073\\
117	116430.606125926\\
118	120552.819293431\\
120	135307.498660874\\
123	173071.851266729\\
126	218735.266337953\\
129	261530.570616493\\
132	303144.371001258\\
138	383140.859089253\\
141	417015.964243567\\
143	433322.50319923\\
145	444317.298558365\\
147	449711.529204411\\
148	450221.602048398\\
149	449264.425921177\\
151	443090.878012065\\
153	431705.622538476\\
155	415817.987468434\\
158	385136.21751996\\
161	348486.628570499\\
165	295043.448262952\\
169	241825.838736177\\
174	181376.100518825\\
180	122314.504245756\\
188	69284.5739127282\\
194	46150.4858813888\\
196	42011.8298900993\\
198	40014.5240095398\\
200	39038.379701802\\
202	37517.0665666224\\
204	35383.8394419247\\
207	31107.2497004535\\
211	24908.8154428588\\
215	19498.6025979056\\
216	19134.0440181995\\
217	19185.1220823112\\
218	19617.5310130503\\
220	21455.7833887765\\
225	28221.0497692852\\
227	29763.2232895985\\
228	30017.0010264674\\
229	29920.6739431101\\
230	29505.8767740479\\
232	27952.5232699933\\
238	22565.8804118522\\
239	22416.7223764294\\
240	22726.6954543562\\
242	24490.7175848549\\
247	29869.4805862248\\
248	30074.2967535855\\
249	29703.5601982902\\
251	27597.9075970897\\
259	17982.0285483825\\
260	17739.7560508396\\
261	17738.7210656035\\
262	17963.8413032315\\
269	21041.6150586745\\
272	21588.6672147596\\
277	22332.6743485966\\
280	23108.0693102276\\
284	24666.3281004202\\
294	29285.0841630757\\
298	30494.505683801\\
302	31207.6988255143\\
307	31712.0612102841\\
310	31921.1812012072\\
312	31801.1801389454\\
315	31218.0616840202\\
328	28078.9561805733\\
335	26311.4649010216\\
339	24900.1058547873\\
344	22500.785935425\\
358	15858.4858298693\\
363	13311.9470152878\\
379	7282.4652049244\\
392	5103.69559913105\\
407	2861.52350347254\\
411	2670.43151750955\\
413	2392.10641755225\\
417	1605.07866467548\\
423.999999999999	787.999537714862\\
428	661.938930959004\\
430	641.159505146026\\
431	643.656424127915\\
433	680.743700013459\\
437	873.543624482027\\
441	1075.1370824289\\
443	1107.9612597693\\
444	1101.02691234236\\
445	1106.68465663396\\
447	1156.53550460271\\
451	1385.58420760772\\
457	1815.51231991132\\
460	1932.63811293827\\
462	1947.14048515631\\
464	1911.82147676638\\
467	1788.33859245815\\
476.999999999999	1369.544952335\\
479	1361.35561777161\\
481	1385.02130104473\\
484	1472.69115546765\\
495	1932.26392665442\\
498	1975.85564339797\\
500	1976.0108967435\\
503	1946.49767349024\\
508.999999999999	1877.68806730429\\
511	1876.06956676728\\
513	1891.43180065454\\
516	1947.68570143972\\
522	2155.18900334772\\
531	2490.32944030754\\
533	2509.95208160045\\
535	2457.51174964862\\
540	2177.42733300682\\
545	1980.07587408397\\
547	1961.23346977784\\
550	1992.82588799635\\
554	2129.94930739542\\
570	2979.10522028247\\
572	2977.51360820991\\
575	2923.7148150931\\
585	2697.56875399491\\
588	2689.29657178096\\
591	2711.27410936563\\
596	2799.90751352885\\
603	2930.61118361257\\
606	2938.76384769276\\
608	2920.41224034269\\
612	2832.49425464621\\
631	2357.88987518969\\
643	2252.71087527554\\
648	2145.23615945539\\
658	1854.24110313122\\
666	1687.35986464793\\
672	1603.24699518533\\
676	1410.23042125203\\
692	724.66318791213\\
696	687.267259410913\\
697	684.605431298897\\
699	688.458533644064\\
702	714.535124762018\\
709	779.836237941364\\
710	777.733421443061\\
712	759.830405398988\\
715.999999999999	674.163590950768\\
724.999999999999	422.197238698852\\
732	328.423282404685\\
735	319.496198996596\\
737	322.622764017265\\
741	343.747628420638\\
749	389.191258417908\\
750	390.091672865136\\
752	386.958128597036\\
754	362.43418186232\\
756.999999999999	273.241957429136\\
762	106.838146525472\\
766	55.6761371154884\\
768	59.0634956492216\\
774	85.8397614944443\\
776	78.6350043528948\\
780	44.57323493499\\
784	26.3344915617239\\
785.999999999999	35.0602223205539\\
792	93.1014073930446\\
796	113.049788373147\\
798	107.489818521935\\
801	82.9993565749267\\
806	34.9158176216128\\
815	6.91665055640234\\
817	6.48191976289202\\
818	6.50320336085142\\
820	7.48135085980265\\
839	106.390913551355\\
841	109.145730564877\\
843	105.718289407225\\
847	84.1082726898338\\
853	42.0713097684102\\
864	11.9478091362107\\
866	11.4460198853544\\
868	12.0375512541175\\
872	16.6634118974911\\
887	69.5365544541173\\
891	76.7349633195488\\
893	76.1553493978781\\
896	70.6866145075608\\
903	49.1069285789761\\
913	30.889605364992\\
915	30.4553355164746\\
916.999999999999	31.167988953961\\
921	36.0177582864181\\
940	86.6511332592765\\
943	88.3116908983125\\
946	86.8016957440471\\
951	79.4695856187664\\
963	63.4689659119271\\
966	63.950561753414\\
970	67.8526674342632\\
979	86.9111169152082\\
988	105.589931724552\\
993	108.42534937666\\
997	106.206118648286\\
1007	93.9158707273307\\
1014	89.4690583153014\\
1017	89.4953649537339\\
1021	91.1905381025762\\
1030	99.3148823774004\\
1038	104.460923564371\\
1042	103.8376698479\\
1047	99.8019396129739\\
1059	84.3625700999006\\
1069	76.5678887542321\\
1075	75.2042338965272\\
1084	74.5264799262555\\
1088	73.2515454770935\\
1094	69.075353929534\\
1104	57.9301186572626\\
1120	44.5555841055769\\
1132	40.89766674178\\
1140	37.9579169523666\\
1149	32.7712171270225\\
1170	22.8158292167764\\
1178	21.9178352462048\\
1188	21.2886256612158\\
1195	20.239598058731\\
1218	16.6609392085181\\
1223	16.8258059938323\\
1232	17.7033345213373\\
1242	18.4898256749932\\
1248	18.5851855127202\\
1258	18.5696134019609\\
1263	18.7800333489816\\
1270	19.5234910147395\\
1352	34.7914082388964\\
1367	36.325765171922\\
1378	36.83871700759\\
1385	36.7475870759865\\
1392	36.2455582976217\\
1402	34.8392415488195\\
1417	31.7201894164254\\
1438	26.4855734094296\\
1500	14.9193352621587\\
1513	14.2959293501318\\
1522	14.1946441391818\\
1531	14.277165041072\\
1557	14.6172525194899\\
1571	14.6470355418604\\
1582	14.5754390933573\\
1595	14.3517899320693\\
1621	13.8842223592161\\
1631	13.8970771188105\\
1641	14.0615061593974\\
1654	14.5150894433037\\
1673	15.6217616874629\\
1706	17.6831683426401\\
1719	18.0474268197991\\
1727	18.0515732376804\\
1736	17.8292145595933\\
1748	17.1522998768318\\
1763	15.7605824738207\\
1783	13.2881850076161\\
1856	6.78171804469613\\
1878	6.24429260505981\\
1905	5.68995519345171\\
1927	5.01884128673892\\
1955	4.00987636261996\\
2004	2.69523006515524\\
2011	2.70379394545056\\
2019	2.80547438795817\\
2033	3.18762416682944\\
2075	4.7658055455986\\
2091	5.05585073879031\\
2099	5.08711101463941\\
2107	5.04570506496488\\
2119	4.87237370482467\\
2181	3.75411589534553\\
2222	3.25437147815658\\
2241	3.15686594600709\\
2252	3.14757590366133\\
2263	3.17577858963551\\
2277	3.27149583480287\\
2295	3.50820703692838\\
2319	4.06142689458048\\
2361	5.7338741537317\\
2408	8.09104608927128\\
2436	9.12167290264841\\
2456	9.47973018737814\\
2468	9.52441159070688\\
2480	9.44760196529658\\
2496	9.18234319373175\\
2521	8.5136975651838\\
2598	6.64127452016023\\
2614	6.56312896740428\\
2626	6.60016239693759\\
2641	6.76500247280217\\
2662	7.21595062870425\\
2696	8.43221633715617\\
2762	11.3420771569399\\
2791	12.1075725271982\\
2807	12.2403660416691\\
2820	12.1640770256283\\
2836	11.8540306511631\\
2860	11.0482356997133\\
2912	8.94328248931351\\
2955	7.73537266821415\\
2978	7.47117065071177\\
2991	7.44930834155039\\
3006	7.52667173622165\\
3027	7.79407857402327\\
3056	8.43210487594538\\
3135	10.6794840138354\\
3147	10.738029229932\\
3159	10.6579021482577\\
3176	10.3086368079903\\
3200	9.42561764510113\\
3240	7.52489454990929\\
3302	5.42018406990737\\
3341	4.82977357694482\\
3369	4.64793486832409\\
3387	4.6233798195171\\
3402	4.65294216034534\\
3423	4.77356881194273\\
3474	5.10513530310977\\
3488	5.087260249762\\
3504	4.97116557862954\\
3526	4.64805263887417\\
3561	3.90237958843434\\
3636	2.70185988151923\\
3671	2.47357936325936\\
3700	2.39845873183868\\
3728	2.37583160514687\\
3758	2.37746263677283\\
3780	2.3929611210735\\
3798	2.45481585825689\\
3882	2.83825576075837\\
3901	2.83027702619297\\
3921	2.78919578677651\\
3965	2.62783132766593\\
4005	2.52519890616233\\
4028	2.51156173656333\\
4048	2.53087302239408\\
4072	2.59964823334231\\
4103	2.77851465170381\\
4170	3.42880672504163\\
4214	3.7687922569229\\
4229	3.79628002781051\\
4245	3.76413382224072\\
4271	3.62283664981308\\
4327	3.33495228755577\\
4341	3.32581474790964\\
4360	3.35860527627822\\
4390	3.49016066076353\\
4435	3.82446003828028\\
4484	4.45716278419953\\
4559	5.59530382749814\\
4586	5.73012248297672\\
4604	5.70123207611255\\
4629	5.54111456109707\\
4748	4.7190735175379\\
4787	4.64101425673692\\
4811	4.65069422789503\\
4838	4.72474482231188\\
4903	4.93332090531099\\
4921	4.89683304798422\\
4945	4.74225581016111\\
4983	4.30573218316221\\
5073	3.41152982896042\\
5112	3.27678267464768\\
5128	3.26849125609646\\
5148.99999999999	3.3012556875246\\
5180	3.43272830929843\\
5257	3.80572559285224\\
5272	3.8101798144205\\
5292	3.76342766682367\\
5322	3.59161975262251\\
5385	3.06472524895276\\
5446	2.72339701728968\\
5485	2.64659931422212\\
5514	2.63936920363838\\
5591.99999999999	2.66489076290339\\
5619	2.63096887944705\\
5658	2.52372404671382\\
5732	2.33430470275684\\
5761	2.32866232287083\\
5806	2.33707385767161\\
5843.99999999999	2.35998583440638\\
5878	2.41179421384484\\
5914	2.54282293619046\\
6013	2.9757070856831\\
6039.99999999999	2.99194570449247\\
6075	2.96621187587683\\
6123	2.93729831137307\\
6153.99999999999	2.95330790516345\\
6190	3.01541229036439\\
6235	3.17489867893108\\
6348	3.65460501906506\\
6373	3.64517413771907\\
6410	3.55935446331678\\
6492	3.3813101655174\\
6518.99999999999	3.37314787086321\\
6549	3.40775602170196\\
6597	3.54122123805558\\
6674	3.74850442208166\\
6697	3.7531367654215\\
6726	3.7071450455915\\
6773	3.54097218253771\\
6872	3.21817736375895\\
6909	3.19548398450928\\
6947	3.21922232446361\\
7012	3.26564339268899\\
7043	3.241416362452\\
7082	3.14861667881884\\
7171	2.82176621569527\\
7239	2.66898278131348\\
7288.99999999999	2.63404918539234\\
7342	2.6302377189172\\
7392	2.63272255599345\\
7430.99999999999	2.60476396554273\\
7529	2.51879067268494\\
7571	2.53242665097162\\
7726	2.60930203771884\\
7768	2.59904952702268\\
7818.99999999999	2.54976902482608\\
7908	2.46910789801405\\
7951	2.46524563231891\\
7988	2.49001324932942\\
8036	2.57585985543874\\
8151	2.80772144232993\\
8193	2.81696483840775\\
8262	2.81159405836714\\
8301.99999999999	2.84324043143913\\
8360	2.95329529991002\\
8451	3.12020503504308\\
8485	3.11219421216869\\
8528	3.03849411987403\\
8648.99999999999	2.8173196666939\\
8690	2.81498091197915\\
8736	2.84684491254362\\
8819	2.9081018081417\\
8859	2.88951326029833\\
8999	2.78706253355143\\
9142	2.76992030394511\\
9191	2.70831111854576\\
9433.99999999999	2.34994815273021\\
9492.99999999999	2.33858801621672\\
9620	2.33701641245944\\
9704.99999999999	2.35323704136436\\
9769.99999999999	2.39140363077703\\
9988	2.56106840448142\\
10000	2.56554616315216\\
};
\addlegendentry{$n=100$};

\addplot +[ line width=1.5pt]
  table[row sep=crcr]{%
1	1736900412084.4\\
2	325431382746.353\\
3	283603543916.803\\
4	243351404000.856\\
5	205330132509.881\\
6	169499706938.051\\
7	135291525149.864\\
8	101791706948.112\\
9	68070503535.3361\\
10	34192868493.4559\\
11	5544626872.20058\\
12	10689705464.6893\\
13	93380509256.4224\\
14	17860469206.8732\\
15	4521123748.27675\\
16	929180290.035262\\
17	192750252.583867\\
18	182421196.698756\\
19	264962798.694289\\
20	280319590.191361\\
21	234752158.450365\\
23	126973517.544583\\
24	107914729.674012\\
25	105490066.413412\\
26	106317074.590499\\
27	102792753.847955\\
28	94265105.3691913\\
29	83614505.4447228\\
31	65669154.9075006\\
35	44907144.3220858\\
37	36460058.2461073\\
39	29307563.2544219\\
41	22971710.4533817\\
43	17364414.9564842\\
45	12714750.9773775\\
47	9041664.85035963\\
49	6219860.38288117\\
51	4121599.5168021\\
53	2603182.47072568\\
55	1531568.28228924\\
57	816191.781833857\\
59	388619.652472675\\
65	36954.0446993877\\
66	27950.6350235539\\
67	22700.9770886002\\
68	20224.7312282081\\
69	20003.9956630219\\
70	21510.1545279495\\
70.9999999999999	24206.2677813626\\
72	28203.9775338548\\
74	41259.4511383226\\
78	90033.6815797999\\
79	103371.613587778\\
80	115029.381038398\\
81	124137.562036694\\
82	130424.163073353\\
83	133974.076554739\\
84	135049.768051889\\
85	134213.534888447\\
86	131686.465309854\\
86.9999999999999	127389.716153955\\
89	114879.748990746\\
91	99785.2748700107\\
95	72018.1107038538\\
98	55549.8494728801\\
100	45799.1358052029\\
101	43220.6466648733\\
102	43007.7773228361\\
103	45160.764083537\\
104	49413.996246277\\
107	71723.7168588716\\
109	89605.4026711705\\
111	105151.863162166\\
113	116984.414498177\\
119	149844.899237531\\
122	179050.5611313\\
132	327893.987338968\\
135	375189.322346378\\
138	416866.629888033\\
141	449334.012260605\\
143	464303.493428009\\
145	473347.638276232\\
147	476417.858921754\\
148	475717.15846519\\
150	470023.259006659\\
152	459078.935328988\\
154	443525.89668356\\
157	412999.95275021\\
160	375941.835197293\\
164	320955.25452062\\
168	264861.295984522\\
173	200414.013122083\\
179	137396.485663677\\
188	74510.5412700063\\
197	42117.8542837383\\
199	38976.8067088788\\
201	37136.1444817959\\
203	36156.6457991061\\
205	35634.5935814198\\
209	34861.8620024804\\
218	32380.1122407222\\
219	32359.9490191922\\
220	32447.2320505766\\
222	32928.2617293237\\
225	34199.9224304409\\
230	36535.5416499682\\
231	36682.3695777065\\
232	36463.0517083861\\
234	34912.2166799897\\
236	32347.5321405625\\
243	23903.2508371149\\
245	23154.2149069601\\
246	23069.7197077658\\
247	23128.8120024017\\
249	23605.5047838093\\
252	24954.032950189\\
256	26834.5715149853\\
259	27669.0436140407\\
261	27876.4085514516\\
263	27832.3234131168\\
267	27650.6744439948\\
269	27770.6784417276\\
271	28089.0353766824\\
274	28913.842820537\\
286	33107.4354281925\\
290	33826.4246864869\\
297	34696.4841498597\\
302	35091.9819611643\\
305	35053.6275328667\\
308	34830.3175465149\\
312	34287.0485229654\\
317	33293.8102428955\\
322	31922.0683625899\\
328	29891.0580900246\\
333	27834.9602813324\\
341	24155.331381057\\
356	17988.2993484067\\
362	15548.2742316572\\
379	9969.84414381981\\
394	7382.94248878489\\
407	5372.92547450159\\
413	4901.90340192754\\
423	4240.55542899455\\
433	3625.33970247397\\
436	3570.04946142101\\
439	3584.7808089773\\
445	3635.14670048238\\
448	3629.84575004075\\
454	3582.5922319339\\
457	3472.14005573193\\
467	3042.85637391959\\
469	3023.32316369509\\
472	3048.85350764889\\
475	3130.38297975349\\
482	3453.13816778534\\
488	3704.39731547902\\
491	3749.63598879626\\
494	3719.66194129192\\
502	3581.59532129117\\
505	3594.04775846589\\
508	3659.21995654446\\
513	3871.79659990536\\
520	4191.32919960042\\
523.999999999999	4278.90064307125\\
527	4298.51171414583\\
532	4292.07010909226\\
536	4290.29908735917\\
540	4331.8563910811\\
545	4446.71576286367\\
556	4706.91508811692\\
560	4740.24363516004\\
563	4722.38115311388\\
567	4651.61445635466\\
582	4357.13352017042\\
592.999999999999	4235.43929309261\\
598	4127.0889885058\\
605	3881.06587488139\\
637	2811.36270671912\\
649	2419.2753807023\\
667	1904.47367461336\\
674	1814.29020015512\\
679	1759.90032480573\\
685	1646.70956290553\\
706	1257.01486703304\\
710	1242.71621880131\\
714	1242.8823630167\\
721	1249.89422643176\\
724.999999999999	1236.40650083716\\
731	1193.90581152295\\
737	1130.52417329371\\
756.999999999999	903.494902348749\\
760	904.311667949453\\
763	914.96347165852\\
776	979.904352005537\\
779	972.291971181293\\
784	940.765356847296\\
791	902.934520641313\\
794	900.148739957234\\
798	909.678288080644\\
802.999999999999	940.23107175062\\
820	1071.75715990957\\
824	1078.57168451304\\
829	1073.09667036071\\
847.999999999999	1034.30636726075\\
851.999999999999	1041.26585083195\\
857	1063.091984459\\
869.999999999999	1129.81878695246\\
875.999999999999	1139.85052120347\\
890	1145.53285870692\\
898	1170.05026036384\\
910.999999999999	1211.7724219572\\
915.999999999999	1212.93210779416\\
920	1201.77565505833\\
926	1164.83375199595\\
949	1003.76695510824\\
955.999999999999	995.351878140772\\
964	987.484571486853\\
970	970.023436865873\\
976.999999999999	931.298330116895\\
1009	749.435346014489\\
1025	708.921989126476\\
1033	659.959672219404\\
1058	502.517941455817\\
1062	504.381073605327\\
1069	523.604434151399\\
1076	536.505732756977\\
1080	533.186411811205\\
1085	515.926450040782\\
1095	456.612132326724\\
1107	404.001746527045\\
1113	397.762046363775\\
1117	400.773222760366\\
1124	415.931185853029\\
1135	439.828722723005\\
1138	440.787504511821\\
1142	433.033908135802\\
1149	398.185962880552\\
1164	331.04035109869\\
1168	327.149794054954\\
1172	329.719706088365\\
1178	344.168484100291\\
1197	404.497356294669\\
1202	402.349622449975\\
1208	388.964440460291\\
1223	355.781501181004\\
1226	354.74807311694\\
1231	358.262792153385\\
1237	373.306940796296\\
1262	461.87441394997\\
1266	457.722430295267\\
1274	436.671701150466\\
1294	375.726317570295\\
1299	376.34158236448\\
1305	384.654567493886\\
1314	411.627135112801\\
1328	454.062959669728\\
1331	455.255810738025\\
1335	451.008357453162\\
1342	431.907672360122\\
1360	386.111711102818\\
1367	382.899075232264\\
1373	386.55448135486\\
1386	399.582741234705\\
1391	396.232483601594\\
1398	381.510796387138\\
1427	315.984824469348\\
1434	314.277878765647\\
1449	317.851710012003\\
1455	314.088655317894\\
1463	300.654504450757\\
1474	270.328577399167\\
1497	212.408325350975\\
1506	206.257921727717\\
1512	207.350132029643\\
1529	217.411347938168\\
1534	215.453353650992\\
1541	206.535292405002\\
1571	165.930366708001\\
1577	167.363684712907\\
1587	175.018550580749\\
1611	193.274712366291\\
1617	192.754113388352\\
1625	187.185856893793\\
1649	169.044327326721\\
1655	169.860613444942\\
1680	180.522220348909\\
1684	178.89337363469\\
1692	169.149355386953\\
1711	148.090471770615\\
1715	147.604576546033\\
1721	149.424496104178\\
1730	157.28570462719\\
1753	181.680386488676\\
1759	180.807710653608\\
1768	174.089169457513\\
1785	162.920694944691\\
1790	162.475485794661\\
1797	164.082466685562\\
1808	170.873096128514\\
1825	181.091207328119\\
1831	180.776731037979\\
1839	176.793527127773\\
1866	162.160335441663\\
1874	162.248271392297\\
1886	165.289103275063\\
1897	167.237501123248\\
1904	165.978971358439\\
1915	159.765799069478\\
1943	144.002735052354\\
1952	143.216670419396\\
1970	143.589888525634\\
1979	141.347483858481\\
1989	134.939748291011\\
2018	116.717794499077\\
2029	112.865752713166\\
2048	102.090984252176\\
2072	90.9151962309844\\
2099	83.4916585312435\\
2139	71.2394513929471\\
2184	62.9626759870524\\
2203	59.6407072769375\\
2211	60.1753903690974\\
2224	63.1456149652716\\
2235	64.7672344155545\\
2241	64.1610054342971\\
2251	60.8184076082554\\
2280	51.5155978118963\\
2285	51.3008259946282\\
2292	51.7050884575499\\
2305	54.1218161847344\\
2320	56.3321966591011\\
2327	55.8852452998852\\
2338	53.4490789417149\\
2361	49.035476682077\\
2366	48.9417373253584\\
2375	49.5800017807835\\
2395	53.2074691566986\\
2408	54.4564581617382\\
2417	54.0205133412045\\
2443	51.8886014560065\\
2454	52.3994726274417\\
2463	52.1908614840149\\
2491	50.4892797069867\\
2504	50.63183911323\\
2518	51.3926295685798\\
2535	52.3076852186374\\
2544	51.8406983269176\\
2559	49.5406585677903\\
2592	45.0114288856882\\
2611	44.1406785897288\\
2623	43.4358321297448\\
2637	41.4069307629552\\
2686	34.1372277082962\\
2717	32.4659719002643\\
2742	29.2030957716871\\
2764	27.4621171470969\\
2777	27.2992879570175\\
2796	27.3642916943118\\
2808	27.0047758513658\\
2845	25.4470539994972\\
2855	25.6313935618491\\
2880	26.5956826939249\\
2891	26.3599360187013\\
2910	25.9202151763939\\
2920	26.1539935616632\\
2935	27.2649009752163\\
2976	30.8870676477673\\
3008	31.5905201137232\\
3027	33.1851975217461\\
3060	36.0041171654949\\
3080	36.7257332399527\\
3101	39.6886301671458\\
3118	41.1895432481841\\
3126	40.7689484585993\\
3142	38.1759345581807\\
3173	33.9932266049627\\
3182	33.817923397671\\
3203	34.1161976098842\\
3212	33.7280316069793\\
3225	32.0071023600062\\
3238	28.1441693038889\\
3276	19.1249668327802\\
3288	18.5573092902026\\
3296	18.8143906984016\\
3314	19.864656089055\\
3321	19.5100049448956\\
3333	17.4153651689015\\
3372	11.6587324379648\\
3380	11.5112108218034\\
3388	11.6927482793043\\
3404	12.9097162990536\\
3418	13.6712056380418\\
3425	13.4759600514331\\
3439	12.1067504965076\\
3467	10.0915941497831\\
3472	10.0563394545906\\
3481	10.2170709371641\\
3496	11.1032714660671\\
3518	12.3101521100055\\
3526	12.1003099888224\\
3542	10.8554950025263\\
3567	9.57284631415367\\
3577	9.63238897722342\\
3591	10.239857481968\\
3622	11.757219818063\\
3632	11.601274794811\\
3653	11.1353825161692\\
3662	11.2597070492211\\
3676	11.9676522106945\\
3746	17.2572816768712\\
3767	18.5478157473695\\
3826	23.2683164593164\\
3866	23.1261941823654\\
3891	23.3364286833216\\
3902.99999999999	23.1018518909868\\
3920	22.0162736193762\\
4013	14.7493966234245\\
4077	10.5861221302997\\
4109	9.44406285326258\\
4151	8.07749874416768\\
4158	8.05711334486605\\
4175	8.15044696810921\\
4187	8.19251769290361\\
4200	8.1055028382794\\
4235	7.71406412917906\\
4246.99999999999	7.79528126126808\\
4271	8.28913069424455\\
4308	8.91413652312047\\
4330	9.41820093220138\\
4363	11.0690663823854\\
4442.99999999999	15.9989473869065\\
4469	16.6923650105383\\
4500	17.9553079472838\\
4536	19.6138685416246\\
4547	19.3515671511475\\
4629	15.4365149139017\\
4652	13.482688202854\\
4763	6.414177026285\\
4795	5.65982496002171\\
4807	5.60896169879785\\
4822.99999999999	5.67168549846189\\
4851	5.76835189480739\\
4870	5.80965790974096\\
4931	6.10514782574389\\
4946	6.04546755552035\\
5014	5.56937466844389\\
5038.99999999999	5.6229497498108\\
5061	5.73599867765888\\
5082	6.07281766963084\\
5115	7.28783703815145\\
5213	12.6686008727933\\
5257	14.5512556241583\\
5296	15.2373829620869\\
5308	15.279233894054\\
5325	15.1279005602744\\
5351	14.4071898991226\\
5402	12.4015952550978\\
5447	9.43789993451032\\
5528	6.21895580152903\\
5557	5.76841728631101\\
5569.99999999999	5.84931245977672\\
5706	7.81792092917648\\
5720	7.73542774306364\\
5751	7.05037382213046\\
5783.99999999999	6.70788846283879\\
5805	6.63671215263824\\
5822.99999999999	6.33656576903905\\
5876	5.19318976006947\\
5888	5.26578445654561\\
5977	6.6528498388492\\
6060	9.5330634442997\\
6130	10.7887342487648\\
6147	10.6554643515872\\
6185	9.87709788418261\\
6241	8.41766924706205\\
6270	6.62675086964825\\
6306.99999999999	5.46630826009905\\
6317	5.53444419144467\\
6345	6.00721172562442\\
6356	5.88325793584438\\
6397.99999999999	4.90928418132709\\
6410	5.00330279839752\\
6437	5.74030207766002\\
6465	6.30791871752699\\
6480	6.19277810048122\\
6508	5.95828270989692\\
6523	6.04026343441989\\
6569.99999999999	6.70083712629354\\
6582.99999999999	6.59459851980466\\
6636	5.82550830426377\\
6651	5.89076941784742\\
6678	6.0728335878334\\
6691	5.99914233073592\\
6736	5.42575313530056\\
6748.99999999999	5.50419287216661\\
6781	6.19632843903958\\
6815	6.58383690797302\\
6836	6.69818152658557\\
6860.99999999999	7.20727869505664\\
6926	8.73804512048386\\
6942	8.77186782855817\\
6956	8.60128657199519\\
6998	8.03481271140998\\
7015	8.11898211395964\\
7050	8.56865096366507\\
7062	8.4347660265975\\
7090	7.44183497728382\\
7135.99999999999	6.49579324627622\\
7154	6.53954167172561\\
7166	6.59118368634079\\
7179	6.48248893838669\\
7204.99999999999	5.73085767505117\\
7241.99999999999	5.09432659449956\\
7256	5.18908474508052\\
7296	5.71792843555242\\
7310	5.61437890940636\\
7353	5.12674985972719\\
7368	5.20402660977351\\
7415.99999999999	5.77192374084306\\
7430.99999999999	5.67978450433187\\
7479	5.2127346498007\\
7497	5.28839002729626\\
7532	5.52933356660094\\
7548.99999999999	5.43162757363757\\
7589	4.81720592685128\\
7616	4.46733329658084\\
7638.99999999999	4.54227739885113\\
7660	4.76677707764464\\
7697.99999999999	5.42313803364773\\
7713	5.31414148598439\\
7735	5.15588557382126\\
7757	5.23926313031865\\
7795	5.52386548676486\\
7822	5.80522884608567\\
7837	5.69959695594747\\
7897	5.21048623224096\\
7921	5.13481377741584\\
7938	5.14811460943984\\
7954	5.06946872108338\\
8045.99999999999	4.22548863861968\\
8074	4.2729476100505\\
8091	4.21335848614791\\
8129	4.03715301947544\\
8155	4.08289832660744\\
8278.99999999999	4.43978670960607\\
8332	4.47211588684372\\
8357	4.4139324395542\\
8464	3.98713481535678\\
8574	3.55233306906605\\
8620	3.5871483598566\\
8657	3.72251664651747\\
8799	4.24706761521417\\
8837	4.26937059422807\\
8864	4.22429007466164\\
8945	3.91084928953765\\
9048	3.58805500510501\\
9069	3.57953318571735\\
9098.99999999999	3.62036463483077\\
9162	3.829710611212\\
9273.99999999999	4.19273728917256\\
9320.99999999999	4.21031573540062\\
9357	4.17680694990015\\
9403	4.03127049345631\\
9541	3.62477043834379\\
9574	3.60911666019309\\
9610.99999999999	3.64435258734186\\
9769.99999999999	3.92457097811746\\
9810	3.88997053038856\\
9867	3.76250403029644\\
9956.99999999999	3.40354790947132\\
10000	3.26563550132889\\
};
\addlegendentry{$n=1000$};

\addplot +[ line width=1.5pt]
  table[row sep=crcr]{%
1	1736573145388.4\\
2	325320464374.521\\
3	283584672474.27\\
4	243413533467.808\\
5	205461427596.517\\
6	169689669365.335\\
7	135531873815.082\\
8	102075127500.397\\
9	68384819239.1311\\
10	34501320471.2551\\
11	5717996563.39348\\
12	10385068272.7189\\
13	93683213252.9006\\
14	17826355707.1963\\
15	4513902531.64678\\
16	927735692.933715\\
17	192288879.120623\\
18	181657716.916786\\
19	263726655.850512\\
20	278861495.784201\\
21	233383475.789382\\
23	126273816.671524\\
24	107464927.550932\\
25	105170902.900598\\
26	106030358.73128\\
27	102493967.088239\\
28	93953918.1368944\\
29	83300962.5817163\\
31	65376430.29635\\
35	44672882.1398684\\
37	36269152.3496907\\
39	29152941.6540387\\
41	22843276.410074\\
43	17260414.3700352\\
45	12634070.0999439\\
47	8981472.87229657\\
49	6177529.59619444\\
51	4094123.69145958\\
53	2586786.17949781\\
55	1522464.04999597\\
57	811298.041246077\\
59	385832.9916863\\
65	36635.0447532025\\
66	27685.2984501981\\
67	22420.9098294633\\
68	19946.8061176915\\
69	19744.173377545\\
70	21266.1148866128\\
70.9999999999999	24024.3566226495\\
72	28125.8161805184\\
74	41546.4643712234\\
77	76788.1521118786\\
79	103813.756322439\\
80	115177.945068077\\
81	124029.437604944\\
82	130071.122298206\\
83	133413.778875726\\
84	134457.320381289\\
85	133736.294128096\\
86	131449.152176119\\
86.9999999999999	127611.672036817\\
88	122380.337426138\\
90	108927.13462661\\
92	93820.3518249143\\
96	67156.4494757758\\
99	50646.4048182248\\
100	45893.4466059435\\
101	42789.695798521\\
102	42045.1757458146\\
103	43671.2050083245\\
104	47512.3630319892\\
106	60752.2461548703\\
109	87687.3148666907\\
111	104134.438619865\\
113	116773.370523107\\
120	158276.495337733\\
123	190254.112718738\\
131	310611.618756626\\
135	374204.090150686\\
138	415488.783023446\\
141	447351.299483255\\
143	462026.078318404\\
145	470895.617304872\\
147	473811.647238948\\
148	473066.779960966\\
150	467363.192491923\\
152	456431.36391917\\
154	440863.207954834\\
157	410293.084350708\\
160	373260.280409202\\
164	318355.334425641\\
168	262498.104931511\\
173	198314.634829352\\
179	135691.396059564\\
188	73234.7284083243\\
197	40949.9633736317\\
200	36571.6811106775\\
202	35293.5022759758\\
204	34877.7607431099\\
208	34718.5621836403\\
210	34331.1488075805\\
213	33281.7165229111\\
216	32243.9044799848\\
218	31893.7823376322\\
219	31869.9955228975\\
220	31954.4994227779\\
222	32437.7526534622\\
225	33775.9538116753\\
230	36384.3393131142\\
231	36576.4805447654\\
232	36403.9920252332\\
233	35845.9474024368\\
235	33727.9439679346\\
239	27870.9921695047\\
242	24499.7298030546\\
244	23266.697501755\\
246	22822.4180397607\\
247	22855.6997529345\\
249	23327.8807326558\\
253	25195.2499801893\\
257	26978.0056849042\\
259	27497.5391392797\\
261	27745.3881938701\\
263	27803.3126148762\\
267	27809.6223519445\\
269	27956.1588389578\\
271	28266.9299712136\\
274	29041.6711741466\\
280	31272.8952580817\\
286	33338.7768102164\\
290	34200.7662694731\\
294	34680.20890696\\
298	34922.1497490055\\
303	35012.6097291732\\
307	34893.5675257625\\
310	34649.7352780791\\
314	34059.6727169782\\
318	33137.0534577395\\
323	31558.8384801464\\
329	29239.1952985868\\
337	25823.6218721776\\
346	21849.0363280262\\
355	17925.0394191613\\
370	12404.8416772478\\
403	5766.88196953509\\
415	4696.44342172266\\
426	4051.1920417521\\
435	3700.40108441154\\
443	3491.79397305076\\
451	3353.03885889669\\
458	3280.01676175503\\
464	3249.57557000467\\
469	3246.95822104741\\
474	3266.83858520656\\
480	3312.28598476784\\
487	3396.45467297483\\
497	3561.13195238449\\
510	3835.23808435059\\
531	4297.4172155117\\
542	4482.91532014966\\
550	4573.5961816471\\
556.999999999999	4609.59405813584\\
563	4610.12132977193\\
569	4578.04008985977\\
576	4504.74722325998\\
582.999999999999	4389.74788675334\\
592	4190.83966811761\\
602	3915.09268453796\\
614	3533.52057426968\\
629	3031.97683999378\\
651	2337.62537235778\\
699	1359.60706379948\\
718	1165.0785219634\\
735	1051.02142381767\\
749	991.414083509759\\
761	963.00594423221\\
775	941.769184463615\\
782	937.232288446054\\
790	939.137061517662\\
801	949.977301079142\\
811	963.841458421014\\
831	1004.16925574969\\
866	1097.36046621687\\
885	1133.34291643695\\
895	1139.73467648992\\
905.999999999999	1137.97446500706\\
919	1127.25978763172\\
927	1113.75060760562\\
937	1083.31907667629\\
957	1005.66561030179\\
973	932.807631387473\\
994.999999999999	818.619166463817\\
1092	452.711035247357\\
1115	413.864525509823\\
1149	375.083086975412\\
1164	370.625941430301\\
1181	365.87994029658\\
1201	359.053489523853\\
1210	359.694904095775\\
1219	362.880568046514\\
1246	376.248774177111\\
1264	380.281669694667\\
1278	387.562505974679\\
1303	402.121504799847\\
1312	402.71876370613\\
1329	398.714636820426\\
1342	396.874599874017\\
1354	395.13522931944\\
1363	390.960376557616\\
1376	378.708101948762\\
1427	325.341174162622\\
1453	291.150513247383\\
1485	253.267210364004\\
1512	226.634958278088\\
1547	196.524201292982\\
1582	180.335440283009\\
1646	158.957959604926\\
1665	157.980740111349\\
1691	157.186315138317\\
1703	157.993915881955\\
1717	158.869169927459\\
1740	158.711795413458\\
1752	160.201110725571\\
1775	163.482014607046\\
1790	163.783609853177\\
1805	162.419036268003\\
1821	161.271077228564\\
1850	160.923262299388\\
1862	159.333024554224\\
1878	154.740646097609\\
1953	129.404766112309\\
2059	88.0314207886343\\
2138	73.3834034862167\\
2164	71.5507762990184\\
2225	66.0956692808354\\
2239	65.7735467942317\\
2254	66.2599108903372\\
2279	66.8542348060533\\
2298	67.3918938750418\\
2327	68.3955813686899\\
2360	68.4569944010852\\
2380	68.2958114298396\\
2409	68.8387778518099\\
2422	68.0838489556294\\
2484	63.2530826893669\\
2497	62.0917137840698\\
2516	58.7122847170175\\
2553	52.8430463888739\\
2605	47.2990604405586\\
2652	39.9573414857968\\
2664	39.8248052937805\\
2676	39.685228286994\\
2687	38.9871313544066\\
2702	36.8183015307562\\
2732	33.0340407787797\\
2742	32.7924054637918\\
2753	33.1394704604928\\
2776	34.0756755032914\\
2786	33.7242232998693\\
2801	32.2801723269745\\
2825	30.3061540441579\\
2835	30.4430066719409\\
2848	31.5455175247658\\
2871	33.5325149967153\\
2880	33.1728347401126\\
2918	30.1283935169575\\
2928	30.4372355656127\\
2944	31.9353470809471\\
2966	33.6680017477652\\
2976	33.4710780857898\\
3014	31.4107733346958\\
3026	31.6830037422305\\
3061	32.7446448224815\\
3075	32.5119635420705\\
3154	29.7378153147297\\
3198	27.5997532138142\\
3248	24.8035138068943\\
3297	22.3565132407157\\
3332	21.0283784252504\\
3374	19.1300591849787\\
3390	19.0627630288323\\
3408	19.0509724781801\\
3420	18.8012691485459\\
3444	17.5867146570558\\
3470	16.7371851320104\\
3482	16.789511255655\\
3497	17.2541033130961\\
3521	18.0182728446991\\
3531	17.8188230361236\\
3572	16.2175126404649\\
3583	16.3640946711346\\
3605	17.3086200466911\\
3626	17.9069600364077\\
3636	17.743233028307\\
3669	16.8621088039576\\
3681	17.008138505229\\
3701	17.7983197019469\\
3726	18.5837089201185\\
3738	18.4490674834972\\
3763	17.5888025616461\\
3785.99999999999	17.150211638634\\
3800	17.2713073221088\\
3831	17.7648836974889\\
3843	17.6216430610905\\
3927	15.917981297598\\
3954	15.3453003816232\\
4025	13.8164427366792\\
4045	13.3928919605007\\
4100	12.1700656694744\\
4128	12.2522601019766\\
4143	12.092920944032\\
4174	11.3211373199305\\
4212	10.5834531881465\\
4227	10.618632460976\\
4266	11.0024437253021\\
4280	10.9098815836461\\
4324	10.444819074892\\
4341	10.5334736476166\\
4392	10.9762614697532\\
4409	10.8783100911073\\
4441	10.6643422486689\\
4458	10.7658612941087\\
4498	11.2475269199208\\
4513	11.1468889428161\\
4554.99999999999	10.7780032920565\\
4573	10.8657589007322\\
4617	11.1924358994513\\
4635	11.0966125454745\\
4675	10.5361948730428\\
4737	9.97949795116983\\
4771	9.70208679109234\\
4836	8.82417352513604\\
4947	7.66323480867808\\
4966.99999999999	7.69975208382413\\
4990	7.76780277154705\\
5006	7.68938917705642\\
5067	7.12189425444869\\
5085	7.1867554753549\\
5101	7.23159554157734\\
5118	7.16040564969066\\
5170	6.74009705272391\\
5187	6.810321347777\\
5238	7.26943947690214\\
5255	7.19959640814763\\
5280.99999999999	7.08046970006602\\
5298	7.14309029493975\\
5342	7.47003173475656\\
5361	7.39760295449537\\
5397	7.25158652305134\\
5417	7.3193235802322\\
5462	7.57864365210686\\
5479	7.49926366634158\\
5536	7.08286321685309\\
5557.99999999999	7.14798244045356\\
5584	7.2322407677181\\
5599	7.16799298372873\\
5626	6.76702827644334\\
5665	6.3849891601203\\
5683	6.42291782227443\\
5715	6.65026719696317\\
5727.99999999999	6.57146120141395\\
5757	6.04019610048088\\
5794	5.64188147347113\\
5810.99999999999	5.68045758421315\\
5849	5.92331520645345\\
5863	5.84986907168915\\
5902	5.31379599146788\\
5935.99999999999	5.09197480856254\\
5954	5.14886731037557\\
5995	5.40049953862975\\
6009	5.33208076081358\\
6062.99999999999	4.88691857019098\\
6078.99999999999	4.93362416651645\\
6105	5.20646056898249\\
6139	5.54454015755968\\
6153.99999999999	5.47014380165176\\
6201	5.10707704893054\\
6220.99999999999	5.16293705250748\\
6282	5.50392882013412\\
6304	5.44389332870954\\
6340.99999999999	5.33940063167556\\
6363	5.38869535625405\\
6413	5.58803301638723\\
6433.99999999999	5.53218850687386\\
6491	5.32401495693187\\
6518	5.37219097569514\\
6541	5.40945382520331\\
6560.99999999999	5.35054575183045\\
6639	4.99809095106435\\
6680	5.02711617108768\\
6703	4.96177837607222\\
6785	4.61321332258834\\
6824	4.62916027631638\\
6844	4.56642707279798\\
6909	4.2606060406906\\
6931	4.30319907409472\\
6972	4.41689419062347\\
6994	4.3680804937542\\
7062	4.145617264686\\
7085	4.1876830857927\\
7142	4.38059672367535\\
7167	4.33641682133298\\
7200	4.27836337964585\\
7225	4.32015045633701\\
7291	4.46435422917059\\
7348	4.44226350409495\\
7382	4.50032536531341\\
7433	4.58125587747091\\
7470	4.53925455437944\\
7598.99999999999	4.39376768624984\\
7640	4.36117784864022\\
7690	4.24043554535665\\
7755	4.12499166154558\\
7805.99999999999	4.05415689123973\\
7896	3.89519467849384\\
7949	3.88803124378228\\
7989	3.86107052538362\\
8050	3.90622522108252\\
8074	3.85615165054261\\
8128	3.73246355576801\\
8155	3.76952182877972\\
8207	3.88989876682644\\
8233	3.84981953493671\\
8282	3.76378703811444\\
8308.99999999999	3.79597420945704\\
8366	3.91542343424441\\
8394	3.87754767790734\\
8439	3.81662603593885\\
8469	3.85154678597082\\
8522	3.94490838092727\\
8552	3.90768945615912\\
8600	3.84842568324341\\
8633.99999999999	3.88460112801282\\
8675.99999999999	3.93547384719682\\
8699.99999999999	3.90186242579343\\
8781.99999999999	3.71980756081758\\
8815	3.75342367762089\\
8846	3.78816098691888\\
8873	3.74623830268563\\
8947	3.60590759552451\\
8976	3.63784685734181\\
9016.99999999999	3.71071736971691\\
9041.99999999999	3.66501524512329\\
9122	3.48299760797137\\
9155	3.51329132024637\\
9191.99999999999	3.55571967109799\\
9221	3.51971706663513\\
9280	3.43814233846124\\
9314	3.46984739408362\\
9361	3.52741346320595\\
9396	3.49172257693091\\
9444	3.46001802260373\\
9506	3.47892318245027\\
9592	3.51007923829041\\
9677.99999999999	3.52202954534635\\
9744	3.52791392188881\\
9937	3.46549799162237\\
10000	3.43319787045303\\
};
\addlegendentry{$n=10000$};
\end{axis}
\end{tikzpicture}%

%% file: linf_gamma_mu0.tikz
\begin{tikzpicture}

\begin{axis}[%
width=\linewidth,
xmode=log,
xmin=1,
xmax=10000,
xminorticks=true,
ymode=log,
ymin=1,
ymax=10000000,
yminorticks=true,
ylabel={$J_\gamma(u^i)$},
axis x line*=bottom,
axis y line*=left,
legend style={legend cell align=left,align=left,draw=none,font=\footnotesize}
]
\addplot +[ line width=1.5pt]
  table[row sep=crcr]{%
1	17.6216485742651\\
2	114362.219751257\\
3	28896.4774898075\\
4	7304.15036698665\\
5	1849.17573447142\\
6	471.379376939665\\
7	124.005293377862\\
8	37.8125192454592\\
9	20.3917913071825\\
10	21.6249753399069\\
11	19.1525061083095\\
12	18.4032099557564\\
13	17.0558495746748\\
14	16.0531009485512\\
16	14.0413087160915\\
17	13.12109192582\\
18	12.2593687514803\\
19	11.4413158524382\\
20	10.6691794918656\\
22	9.24756506492631\\
24	7.97674142155498\\
26	6.8426285012154\\
28	5.83458547103741\\
30	4.94541719798347\\
32	4.17133139747239\\
35	3.22480962424975\\
38	2.53823773119565\\
40	2.2203856828002\\
41	2.09924894584578\\
42	2.00064592946703\\
43	1.92252326938595\\
44	1.86203593393668\\
45	1.81623261959435\\
46	1.78227658407781\\
47	1.75759577019541\\
49	1.72759852953972\\
51	1.7131578868132\\
54	1.70474252523331\\
60	1.70159762082149\\
142	1.70134530884274\\
10000	1.70134530884334\\
};
\addlegendentry{$\gamma=10^{-1}$};

\addplot +[line width=1.5pt]
  table[row sep=crcr]{%
1	1737.15305653757\\
2	138508.349504646\\
3	35320.9466502508\\
4	9245.3695944936\\
6	956.63626478596\\
7	489.538332424299\\
8	315.483033308179\\
9	266.20937984442\\
10	284.054628373666\\
11	274.400052542375\\
12	278.553795742772\\
13	276.054940197326\\
14	276.843648332267\\
15	276.024146202011\\
18	275.348231754062\\
31	271.484198936188\\
44	267.520198076738\\
58	263.104664605076\\
73	258.189075267611\\
89	252.713072118874\\
105	246.969376002291\\
121	240.92536471485\\
137	234.542783643062\\
153	227.77646234705\\
168	221.036885494295\\
183	213.860196249318\\
197	206.712278539679\\
210	199.631893975208\\
223	192.065726375354\\
235	184.59141639071\\
247	176.58158475287\\
258	168.705546368432\\
269	160.250748614247\\
279	151.99838247288\\
289	143.140324292816\\
299	133.605897842772\\
308	124.385119589204\\
317	114.503089719897\\
326	103.914957180252\\
335	92.5969926329556\\
343	81.9385094099258\\
351	70.7842721533258\\
359	59.2819806588248\\
367	47.7006857064071\\
375	36.4680880194034\\
383	26.1732787066048\\
392	16.5753057819643\\
404	8.05252927873951\\
417	3.86772019065297\\
424	2.98356800586293\\
429	2.65854047168963\\
434	2.4807771956007\\
439	2.38844351368432\\
444	2.3409225323874\\
450	2.31174831806087\\
458	2.29428037699465\\
471	2.2845960460265\\
503	2.28012219817596\\
741	2.27579972356018\\
3782	2.27506591876424\\
10000	2.27506590671043\\
};
\addlegendentry{$\gamma=10^{-3}$};

\addplot +[line width=1.5pt]
  table[row sep=crcr]{%
1	1736900.6647286\\
2	581059.190870433\\
3	474135.552822969\\
4	441027.013094707\\
5	420698.320157616\\
6	392463.553292932\\
7	343052.760880226\\
8	270320.309536768\\
9	236596.172559009\\
10	249587.805375033\\
11	242730.809228227\\
12	245792.860265217\\
13	244048.081860894\\
14	244678.672685104\\
16	244182.841006865\\
46	239400.774978268\\
65	236180.170643423\\
86	232384.069925511\\
107	228309.49269422\\
129	223702.045553658\\
150	218934.68903348\\
171	213752.41576605\\
191	208371.791143414\\
210	202794.578686929\\
228	197026.208853568\\
245	191076.786357767\\
261	184962.249484096\\
276	178705.656527035\\
290	172338.543774861\\
304	165384.190903838\\
317	158321.076747501\\
329	151206.686185649\\
341	143438.659797969\\
352	135663.846745417\\
363	127178.809438938\\
374	117892.825420077\\
384	108675.951631397\\
394	98650.7830326175\\
404	87766.765377001\\
413	77228.7612712581\\
422	66040.6334689652\\
431	54355.8526679945\\
440	42487.5867910588\\
449	31004.0886319386\\
458	20652.4189757934\\
467	12238.9567604941\\
476	6311.64588147001\\
486	2524.53836473267\\
498	659.533673663873\\
511	117.048544296559\\
532	7.13137734560177\\
539	4.32588579711756\\
545	3.41948382799072\\
550	3.08440364486853\\
555	2.92057710426198\\
561	2.82566534794069\\
568	2.76945500148708\\
579	2.71390970666814\\
605	2.61715654225787\\
633	2.53640112683425\\
662	2.47475806731592\\
695	2.42584160809445\\
733	2.38898172975069\\
776	2.36294449235407\\
832	2.34591112432279\\
880	2.34510481148272\\
931	2.35764541924654\\
1006	2.38937963622035\\
1088	2.43376459942458\\
1214	2.50934251926119\\
1485	2.65163660241946\\
1706	2.71050395010769\\
1875	2.72286231900127\\
2266	2.70865003535241\\
2855	2.6647199418271\\
3647	2.61167504588686\\
7545	2.45554645780731\\
10000	2.42749346862265\\
};
\addlegendentry{$\gamma=10^{-6}$};
\end{axis}
\end{tikzpicture}%

%% file: linf_gamma_mu1.tikz
\begin{tikzpicture}

\begin{axis}[%
width=\linewidth,
xmode=log,
xmin=1,
xmax=10000,
xminorticks=true,
ymode=log,
ymin=1,
ymax=10000000,
yminorticks=true,
ylabel={$J_\gamma(u^i)$},
axis x line*=bottom,
axis y line*=left,
legend style={legend cell align=left,align=left,draw=none,font=\footnotesize}
]
\addplot +[ line width=1.5pt]
  table[row sep=crcr]{%
1	17.6216485742651\\
2	80.2837698565678\\
3	42.6493071853466\\
4	26.1860393303207\\
5	17.5393415136577\\
6	12.3754851029674\\
7	8.99586835234042\\
8	6.6431727661924\\
9	4.94911125819187\\
11	2.86794786694354\\
12	2.31154186490963\\
13	1.98606526142922\\
14	1.82125446983049\\
15	1.74906742719584\\
16	1.72021221827924\\
17	1.70899976165441\\
19	1.70276067376167\\
28	1.70135353506155\\
10000	1.70134530845395\\
};
\addlegendentry{$\gamma=10^{-1}$};

\addplot +[line width=1.5pt]
  table[row sep=crcr]{%
1	1737.15305653757\\
2	429.912270713931\\
3	337.644156506394\\
4	275.700504836668\\
5	226.568956596359\\
6	184.280792625861\\
7	145.923471262047\\
8	109.538215478311\\
9	73.7137159190681\\
10	38.3197336458668\\
11	8.55303734373891\\
12	10.950684188818\\
13	85.0818047065063\\
14	23.5058425637391\\
15	7.67557455046111\\
16	3.36947038924973\\
17	2.44444768687516\\
18	2.39092173140039\\
19	2.47209840655521\\
20	2.49583790915531\\
21	2.45667327957585\\
22	2.39927776170935\\
23	2.35658095702221\\
24	2.3380420761554\\
29	2.31992495465578\\
32	2.30483641618485\\
42	2.28369921532776\\
55	2.2758982105457\\
113	2.27508594586332\\
10000	2.27506590621336\\
};
\addlegendentry{$\gamma=10^{-3}$};

\addplot +[line width=1.5pt]
  table[row sep=crcr]{%
1	1736900.6647286\\
2	325536.213283009\\
3	283657.765751608\\
4	243383.86285472\\
5	205351.445543027\\
6	169514.542222649\\
7	135302.198097891\\
8	101799.48347267\\
9	68076.1653712259\\
10	34197.0015206256\\
11	5547.62638846183\\
12	10689.8323366043\\
13	93371.6903011663\\
14	17866.3490515217\\
15	4524.35624293342\\
16	931.642292449679\\
17	195.012653384145\\
18	184.638990232161\\
19	267.181777658599\\
20	282.548939885979\\
21	236.990866551081\\
23	129.218646897894\\
24	110.158607095051\\
25	107.732225447683\\
26	108.559113816924\\
27	105.036782942613\\
28	96.5128835650475\\
29	85.8671734311217\\
31	67.9299835988632\\
35	47.1795634931821\\
37	38.7378570851813\\
39	31.590486242787\\
41	25.2599326896008\\
43	19.6577350162665\\
45	15.012676912489\\
48	9.83617136093238\\
55	3.84383759342788\\
57	3.13049805848275\\
58	2.88585077233289\\
59	2.70478307075397\\
60	2.57576480177876\\
61	2.48834823987495\\
62	2.43190785553866\\
63	2.39577747980634\\
64	2.37321747647983\\
66	2.35175071624348\\
68	2.34626791710094\\
71	2.35334902002743\\
74	2.37266130167769\\
77	2.40823107614189\\
80	2.44730847498732\\
82	2.46264366881086\\
84	2.46710256363177\\
87	2.4590287826506\\
91	2.43047893803188\\
98	2.38328109898039\\
101	2.36961570353466\\
103	2.37037828597089\\
106	2.38640947255103\\
113	2.43690406229058\\
119	2.46827574597372\\
122	2.49684770854304\\
126	2.54934442549771\\
141	2.76524311683168\\
144	2.785402234495\\
147	2.79222408752319\\
150	2.78596783786035\\
154	2.7597373012839\\
158	2.71784550898114\\
164	2.63826171888147\\
177	2.47589212551406\\
183	2.42541409848042\\
189	2.3911149908801\\
196	2.36689791637283\\
202	2.35964942209475\\
224	2.35775033377814\\
236	2.35775924523571\\
248	2.34803027152258\\
326	2.35562339523241\\
425	2.33172816769278\\
739	2.32977990497708\\
10000	2.32954867742355\\
};
\addlegendentry{$\gamma=10^{-6}$};
\end{axis}
\end{tikzpicture}%

%% file: state_accel_clean.tikz
\begin{tikzpicture}

\begin{axis}[%
width=\linewidth,
xmode=log,
xmin=1,
xmax=10000,
xminorticks=true,
ymode=log,
ymin=10000000,
ymax=2000000000000,
yminorticks=true,
ylabel={$J_\gamma(u^i)$},
axis x line*=bottom,
axis y line*=left,
legend style={legend cell align=left,align=left,draw=none,font=\footnotesize}
]
\addplot +[line width=1.5pt]
  table[row sep=crcr]{%
1	1716906611220.19\\
2	222966375774.923\\
3	221193016817.381\\
4	217686343232.436\\
5	210796108864.546\\
6	197461706820.35\\
7	172868708657.541\\
8	136396107737.733\\
9	119464895864.905\\
10	125990404664.21\\
11	122547710184.179\\
12	124086473658.831\\
13	123211094938.87\\
14	123528733671.065\\
16	123281559892.818\\
47	120828762866.929\\
66	119228746189.051\\
86.9999999999999	117344261286.203\\
109	115223408893.541\\
131	112932291950.371\\
153	110446278010.509\\
174	107864064227.332\\
194	105185513191.429\\
213	102411927472.736\\
231	99546465988.0077\\
248	96594631995.2176\\
264	93564828650.2109\\
279	90468967767.6503\\
294	87087979005.4304\\
308	83635679209.2231\\
321	80134441047.4462\\
334	76305968697.8491\\
346	72439195568.5871\\
358	68206957081.2188\\
369	63962956478.0177\\
380	59326157737.9049\\
390	54732758677.385\\
400	49747507124.0096\\
410	44349098488.0017\\
419	39136779748.2307\\
428	33618421609.0426\\
437	27870348079.112\\
446	22042046899.0893\\
455	16380489417.4416\\
464	11229821118.4058\\
473	6967249141.67123\\
483	3586375541.87143\\
494	1472819457.95582\\
510	331357985.889768\\
525	90887222.7859268\\
532	62533152.7374839\\
538	52125134.1176402\\
543	48080738.0825237\\
548	46133808.9448267\\
553	45250857.6135172\\
558	44886082.7033625\\
566	44734280.0383713\\
609	44501712.1779253\\
872	42021421.8881796\\
1037	41264359.5021197\\
1293	40549860.425597\\
1845	39664610.56361\\
3706.99999999999	38200844.1744777\\
5540	37569675.755648\\
8000	37202003.9916429\\
10000	37064513.2313392\\
};
\addlegendentry{$\accel=0$};

\addplot +[line width=1.5pt]
  table[row sep=crcr]{%
1	1716906611220.19\\
2	164170530699.792\\
3	143187817471.246\\
4	122988035506.828\\
5	103902961537.973\\
6	85914728642.205\\
7	68739251789.8715\\
8	51917665499.7787\\
9	34975275240.6864\\
10	17904853996.2533\\
11	3220546884.84653\\
12	8061695483.3424\\
13	85920187240.4363\\
15	3729145796.66866\\
17	123059395.186608\\
18	166280399.455752\\
19	251954628.183593\\
20	276795451.984945\\
21	244888558.822532\\
22	188890332.988829\\
23	137764912.564962\\
24	105377331.230047\\
25	91880864.483021\\
26	88434587.7437103\\
27	86767304.4416075\\
28	83644657.2129951\\
30	76294245.3040695\\
31	74071926.1684681\\
33	71039745.7045219\\
34	69011922.505133\\
35	66444187.4645932\\
40	54004827.7792133\\
42	50765393.0598385\\
44	48152802.3885542\\
46	46032838.8011942\\
48	44332745.7912161\\
50	43039820.5545534\\
52	42159324.1360856\\
54	41652017.2125638\\
56	41446496.9975867\\
58	41477660.7016473\\
60	41695989.4991386\\
63	42278012.7087963\\
74	44738979.6485641\\
77	45042163.2040812\\
80	45103119.2080227\\
83	44921170.7231769\\
86	44521593.0198585\\
90	43728403.0066142\\
95	42500948.0692236\\
106	39989474.0673025\\
111	39272182.5112756\\
116	38863282.6395855\\
121	38693353.0372414\\
132	38638924.4234915\\
140	38474123.0149639\\
149	38062737.9452609\\
162	37497805.9193004\\
170	37389526.3309136\\
183	37480169.5483861\\
194	37469016.3661231\\
207	37223041.5410955\\
221	37028819.4562451\\
233	37092126.6845523\\
253	37211152.3784658\\
435	36813653.3595974\\
669	36807267.6418637\\
773	36732530.325782\\
1093	36703455.9060965\\
2230	36676041.5618061\\
10000	36661316.2890655\\
};
\addlegendentry{$\accel\approx 1$};
\end{axis}
\end{tikzpicture}%

%% file: state_mesh_clean.tikz
\begin{tikzpicture}

\begin{axis}[%
width=\linewidth,
xmode=log,
xmin=1,
xmax=10000,
xminorticks=true,
ymode=log,
ymin=10000000,
ymax=2000000000000,
yminorticks=true,
ylabel={$J_\gamma(u^i)$},
axis x line*=bottom,
axis y line*=left,
legend style={legend cell align=left,align=left,draw=none,font=\footnotesize}
]
\addplot +[ line width=1.5pt]
  table[row sep=crcr]{%
1	1720673349145.76\\
2	164700429125.843\\
3	143255501134.97\\
4	122649881932.017\\
5	103219430334.51\\
6	84939548374.6981\\
7	67515213658.74\\
8	50483387095.8033\\
9	33395241040.1924\\
10	16368677497.5512\\
11	2372154833.78907\\
12	10800608890.845\\
13	82979926138.7959\\
15	3755567385.34408\\
17	125114538.508434\\
18	171233001.415457\\
19	259704297.898341\\
20	285970498.850571\\
21	253926519.535561\\
22	196553339.56191\\
23	143473661.132382\\
24	109247798.851773\\
25	94593549.9085172\\
26	90715479.7809698\\
27	89031572.039145\\
28	85981325.96563\\
30	78615751.6509075\\
31	76382074.7467095\\
33	73445849.0681305\\
34	71465350.6481511\\
35	68903074.6430566\\
37	63118242.830527\\
39	58151591.0984015\\
41	54397489.1564124\\
43	51441507.7614654\\
45	49045707.7221702\\
47	47107882.1315991\\
49	45581280.304438\\
51	44472770.8369815\\
53	43771887.4870049\\
55	43420089.146363\\
57	43343568.4794808\\
59	43483185.0909953\\
62	43999034.2255783\\
66	45017249.4739951\\
72	46503996.3195272\\
75	47001745.7595198\\
78	47266369.3149255\\
81	47277053.2767003\\
84	47038989.837988\\
86.9999999999999	46583126.3727489\\
91	45721221.5237878\\
96	44424553.2426953\\
106	42018644.0769882\\
111	41221260.1317093\\
116	40744725.5331145\\
121	40526913.6002883\\
129	40453394.1589854\\
139	40318261.3280848\\
147	39992657.0499719\\
165	39244910.0465872\\
174	39200525.3149988\\
197	39247920.1904745\\
230	38834991.0826877\\
259	38976973.9935904\\
296	38827857.1585432\\
321	38743758.2270852\\
351	38683066.0587448\\
1851	38463111.1168895\\
10000	38441723.5614358\\
};
\addlegendentry{$n=100$};

\addplot +[ line width=1.5pt]
  table[row sep=crcr]{%
1	1716906611220.19\\
2	164170530699.792\\
3	143187817471.246\\
4	122988035506.828\\
5	103902961537.973\\
6	85914728642.205\\
7	68739251789.8715\\
8	51917665499.7787\\
9	34975275240.6864\\
10	17904853996.2533\\
11	3220546884.84653\\
12	8061695483.3424\\
13	85920187240.4363\\
15	3729145796.66866\\
17	123059395.186608\\
18	166280399.455752\\
19	251954628.183593\\
20	276795451.984945\\
21	244888558.822532\\
22	188890332.988829\\
23	137764912.564962\\
24	105377331.230047\\
25	91880864.483021\\
26	88434587.7437103\\
27	86767304.4416075\\
28	83644657.2129951\\
30	76294245.3040695\\
31	74071926.1684681\\
33	71039745.7045219\\
34	69011922.505133\\
35	66444187.4645932\\
40	54004827.7792133\\
42	50765393.0598385\\
44	48152802.3885542\\
46	46032838.8011942\\
48	44332745.7912161\\
50	43039820.5545534\\
52	42159324.1360856\\
54	41652017.2125638\\
56	41446496.9975867\\
58	41477660.7016473\\
60	41695989.4991386\\
63	42278012.7087963\\
74	44738979.6485641\\
77	45042163.2040812\\
80	45103119.2080227\\
83	44921170.7231769\\
86	44521593.0198585\\
90	43728403.0066142\\
95	42500948.0692236\\
106	39989474.0673025\\
111	39272182.5112756\\
116	38863282.6395855\\
121	38693353.0372414\\
132	38638924.4234915\\
140	38474123.0149639\\
149	38062737.9452609\\
162	37497805.9193004\\
170	37389526.3309136\\
183	37480169.5483861\\
194	37469016.3661231\\
207	37223041.5410955\\
221	37028819.4562451\\
233	37092126.6845523\\
253	37211152.3784658\\
435	36813653.3595974\\
669	36807267.6418637\\
773	36732530.325782\\
1093	36703455.9060965\\
2230	36676041.5618061\\
10000	36661316.2890655\\
};
\addlegendentry{$n=1000$};

\addplot +[line width=1.5pt]
  table[row sep=crcr]{%
1	1716529215583.19\\
2	164116041070.261\\
3	143179398465.307\\
4	123020172449.804\\
5	103969724207.602\\
6	86010850359.1602\\
7	68860577067.7372\\
8	52060550181.0785\\
9	35133767017.2122\\
10	18061217871.4591\\
11	3312959441.23603\\
12	7802157158.33435\\
13	86135049071.2171\\
15	3729956996.38258\\
17	122892952.566104\\
18	165728526.511198\\
19	251150017.070569\\
20	275888666.507352\\
21	244018397.925646\\
22	188159852.866594\\
23	137219462.863045\\
24	105001188.213128\\
25	91609652.9596526\\
26	88202332.4370686\\
27	86536692.6707958\\
28	83407699.1362862\\
30	76058318.8453289\\
31	73836552.3423554\\
33	70795353.8884382\\
34	68763544.3783926\\
35	66195864.7956779\\
40	53787001.6918652\\
42	50554031.5157558\\
44	47945722.473818\\
46	45829124.0201011\\
48	44131808.8895968\\
50	42841196.0829265\\
52	41962230.9775037\\
54	41455445.8176618\\
56	41249622.5497429\\
58	41279918.6661438\\
60	41496891.243034\\
63	42075862.3281607\\
74	44520245.4164176\\
77	44819528.9195383\\
80	44877387.7023201\\
83	44693431.8986284\\
86	44293121.6014471\\
90	43500978.4240711\\
95	42277766.5013572\\
106	39782624.8114267\\
111	39073093.973211\\
116	38670597.429509\\
121	38505099.78295\\
144	38124776.1869434\\
167	37220520.7297068\\
176	37230020.5836035\\
193	37292131.7684089\\
205	37080019.5015\\
220	36848209.7011525\\
232	36898643.8434784\\
254	37022216.1218893\\
395	36605579.5898593\\
484	36647003.1197037\\
4498	36481318.7716386\\
10000	36477431.2020913\\
};
\addlegendentry{$n=10000$};
\end{axis}
\end{tikzpicture}%

%% file: state_gamma_mu0.tikz
\begin{tikzpicture}

\begin{axis}[%
width=\linewidth,
xmode=log,
xmin=1,
xmax=10000,
xminorticks=true,
ymode=log,
ymin=10000000,
ymax=3000000000,
yminorticks=true,
ylabel={$J_\gamma(u^i)$},
axis x line*=bottom,
axis y line*=left,
legend style={legend cell align=left,align=left,draw=none,font=\footnotesize}
]
\addplot +[line width=1.5pt]
  table[row sep=crcr]{%
1	36661473.6645807\\
2	36798236.9910093\\
4	36668966.249646\\
141	36659954.1015787\\
10000	36659723.8677087\\
};
\addlegendentry{$\gamma=10^{-3}$};
\addplot +[line width=1.5pt]
  table[row sep=crcr]{%
1	38411554.6904718\\
2	37239547.0157672\\
5	37079209.8784553\\
29	36900949.7167395\\
359	36792850.0566096\\
851	36659730.7721879\\
10000	36659724.4993015\\
};
\addlegendentry{$\gamma=10^{-6}$};
\addplot +[line width=1.5pt]
  table[row sep=crcr]{%
1	1787365258.64501\\
2	477994930.704864\\
3	474364645.861855\\
4	467364238.294522\\
5	453654651.70033\\
6	427137363.790161\\
7	378246092.047358\\
8	305775838.465119\\
9	272153439.924509\\
10	285110247.519896\\
11	278274292.304402\\
12	281329656.97062\\
13	279591502.556194\\
14	280222211.415664\\
16	279731437.440888\\
48	274698989.878428\\
69	271165992.120742\\
92	267014541.124996\\
116	262325627.291992\\
140	257217373.410403\\
163	251865722.427397\\
185	246260172.677562\\
206	240391444.672279\\
226	234252199.017198\\
245	227837975.290171\\
263	221148395.934398\\
280	214188675.141687\\
296	206971457.783488\\
311	199518981.045373\\
325	191865493.645902\\
338	184059777.267152\\
351	175477119.424224\\
363	166763369.749113\\
375	157182883.805742\\
386	147544110.003283\\
397	136999866.732363\\
408	125485749.959673\\
419	112988963.869543\\
430	99608972.4635388\\
441	85658927.086776\\
453	70581277.8574926\\
484	42768959.8048352\\
492	39768371.8357848\\
499	38256720.5169645\\
505	37521729.3614714\\
512	37063950.8237691\\
521	36808832.5027342\\
535	36695734.6347681\\
579.999999999999	36669772.4003518\\
10000	36660299.2258162\\
};
\addlegendentry{$\gamma=10^{-9}$};
\end{axis}
\end{tikzpicture}%

%% file: state_gamma_mu1.tikz
\begin{tikzpicture}

\begin{axis}[%
width=\linewidth,
xmode=log,
xmin=1,
xmax=10000,
xminorticks=true,
ymode=log,
ymin=10000000,
ymax=3000000000,
yminorticks=true,
ylabel={$J_\gamma(u^i)$},
axis x line*=bottom,
axis y line*=left,
legend style={legend cell align=left,align=left,draw=none,font=\footnotesize}
]
\addplot +[line width=1.5pt]
  table[row sep=crcr]{%
1	36661473.6645807\\
10000	36659723.8677088\\
};
\addlegendentry{$\gamma=10^{-3}$};
\addplot +[line width=1.5pt]
  table[row sep=crcr]{%
1	38411554.6904718\\
2	36984329.5769192\\
8	36761364.0961917\\
12	36670064.4083411\\
13	36753363.7744458\\
15	36664916.5178411\\
72	36659732.956552\\
10000	36659723.9908135\\
};
\addlegendentry{$\gamma=10^{-6}$};
\addplot +[line width=1.5pt]
  table[row sep=crcr]{%
1	1787365258.64501\\
2	360958579.313762\\
3	319266880.423506\\
4	279148611.227448\\
5	241264357.962465\\
6	205580455.971257\\
7	171536802.682966\\
8	138231941.336053\\
9	104745962.850282\\
10	71118987.1488369\\
11	42486406.4222238\\
12	46994767.3088551\\
13	130248232.113006\\
14	56856085.449943\\
15	41847433.359595\\
16	37701983.6235273\\
17	36824242.1346067\\
18	36837346.7468412\\
20	37009517.0795462\\
28	36746905.100627\\
55	36667339.0935269\\
10000	36659726.3185263\\
};
\addlegendentry{$\gamma=10^{-9}$};
\end{axis}
\end{tikzpicture}%